\documentclass[11pt]{amsart}
\usepackage[left=2.3cm, top=2.3cm,bottom=2.3cm,right=2.3cm]{geometry}
\usepackage[T1]{fontenc}
\usepackage[utf8]{inputenc}
\usepackage{amsmath,amssymb,amsthm}
\usepackage{mathtools}
\usepackage{mathrsfs}
\usepackage{bm}
\usepackage{enumitem}
\usepackage[hidelinks]{hyperref}
\usepackage{graphicx}
\usepackage{subcaption}

\newcommand{\R}{\mathbb R}
\newcommand{\N}{\mathbb N}
\newcommand{\Z}{\mathbb Z}
\newcommand{\Q}{\mathbb Q}

\newcommand{\Pp}{\mathbb P}
\newcommand{\Ee}{\mathbb E}

\newcommand{\bK}{\mathbb K}

\newcommand{\bW}{\mathbb W}

\newcommand{\bU}{\mathbb U}

\newcommand{\conv}{\operatorname{conv}}
\newcommand{\convclos}{\overline{\operatorname{conv}}}
\newcommand{\pos}{\operatorname{pos}}
\newcommand{\aff}{\operatorname{aff}}
\newcommand{\posclos}{\overline{\operatorname{pos}}}
\newcommand{\lin}{\operatorname{lin}}
\newcommand{\relint}{\operatorname{relint}}
\newcommand{\intr}{\operatorname{int}}
\newcommand{\bd}{\operatorname{bd}}
\newcommand{\dist}{\operatorname{dist}}

\newcommand{\Ord}{\operatorname{Ord}}

\newcommand{\ind}{\mathbf 1}

\newcommand{\dint}{{\rm d}}
\newcommand{\eee}{{\rm e}}

\newcommand{\arsinh}{\operatorname{arsinh}}

\newtheorem{theorem}{Theorem}[section]
\newtheorem{proposition}[theorem]{Proposition}
\newtheorem{lemma}[theorem]{Lemma}
\newtheorem{corollary}[theorem]{Corollary}

\theoremstyle{definition}

\newtheorem{example}[theorem]{Example}

\theoremstyle{remark}
\newtheorem{remark}[theorem]{Remark}

\title[Ornstein--Uhlenbeck spirals]{Ornstein--Uhlenbeck spirals: Absorption Probabilities and Conic Intrinsic Volumes}

\author{Zakhar Kabluchko}
\address{
Zakhar Kabluchko:
Institut f\"ur Mathematische Stochastik,
Universit\"at M\"unster,
Orl\'eans-Ring 10,
48149 M\"unster, Germany}
\email{zakhar.kabluchko@uni-muenster.de}

\date{}

\begin{document}

\begin{abstract}
Let $(U_t)_{t\in\R}$ be a stationary Ornstein--Uhlenbeck process
taking values in $\R^d$, and write
$U_t=(U_t^{(1)},\ldots,U_t^{(d)})$.
The coordinate processes $(U_t^{(i)})_{t\in\R}$, $1\le i\le d$,
are independent centered stationary Gaussian processes with covariance
\[
\Ee[U_s^{(i)}U_t^{(i)}]=\eee^{-|t-s|/2},
\qquad s,t\in\R.
\]
For every $T>0$, we obtain an explicit formula for the probability
that the convex hull $\conv\{U_t:0\le t\le T\}$ contains the origin.
We obtain a similar formula for a circular Ornstein--Uhlenbeck loop taking values
in $\R^d$, namely a stationary centered Gaussian process
on the circle $\R/T\Z$ whose covariance function is the normalized
$T$-periodization of the Ornstein--Uhlenbeck covariance function.

For each fixed dimension $d$, we determine the precise long-time
asymptotics of the non-absorption probabilities, which decay as
explicit constants times $T^{d-1}\eee^{-T/2}$ as $T\to\infty$.
We also show that, as $d\to\infty$, the absorption probabilities
in both models undergo a transition from $0$ to $1$ around
$T=2d$, with a common Gaussian limiting profile on the scale
$\sqrt d$.

A key step is a Gaussian-point absorption problem for Brownian motion
and Brownian bridge in $\R^d$. Let $Z\in\R^d$ be a standard Gaussian
vector independent of both processes. For every $\sigma\ge0$, we
explicitly determine the probabilities that $\sigma Z$ belongs to
the convex hull of the Brownian motion and to the convex hull of
the Brownian bridge.
The corresponding discrete-time Gaussian-point absorption problems
for Gaussian random walks and random-walk bridges can be expressed
in terms of the Grassmann angles of certain deterministic polyhedral
cones, which brings the methods of polyhedral geometry into play. We compute the conic intrinsic volumes of these cones
by analyzing their faces and normal cones. Passing to the scaling
limit then yields the Gaussian-point absorption formulas for the
convex hulls of Brownian motion and Brownian bridges, from which
the Ornstein--Uhlenbeck results follow by deterministic time changes
and positive rescalings. The large-$\sigma$ asymptotics of the Gaussian-point absorption
probabilities also allow us to recover the known formulas for the
expected volumes of the Brownian-motion and Brownian-bridge convex
hulls in $\R^d$.

We further consider deterministic Ornstein--Uhlenbeck spirals in real
Hilbert spaces.  For $T>0$, an Ornstein--Uhlenbeck spiral is a deterministic curve
$(\zeta(t))_{0\le t\le T}$ in the unit sphere of a Hilbert space with
Gram kernel
\[
    \langle\zeta(s),\zeta(t)\rangle=\eee^{-|t-s|/2},
    \qquad 0\le s,t\le T.
\]
We compute the conic intrinsic volumes of its closed positive hull
in Hilbert space and identify the Grassmann angles of this cone
with the corresponding Ornstein--Uhlenbeck absorption probabilities.  Analogous results are
obtained for the circular Ornstein--Uhlenbeck spiral. After appropriate rescaling, the small-$T$ limits of these
conic intrinsic volumes recover the formulas of Gao--Vitale
and Gao for the ordinary (non-conic) intrinsic volumes of
the Brownian body and its Brownian-bridge analogue, respectively.
\end{abstract}



\maketitle

\tableofcontents

\section{Introduction and main results}
\subsection{Background and motivation}
\label{subsec:introduction}

Convex hulls generated by random walks, Brownian motion, and more
general stochastic processes have been studied extensively in
probability, stochastic geometry, and statistical physics. Among
the questions considered are the moments and distributions of
geometric characteristics, face numbers, scaling limits, and
absorption probabilities.
For mathematical work on convex hulls of random walks, see,
for example,~\cite{WadeXu2015,TikhomirovYoussef2017,
KabluchkoVysotskyZaporozhets2017,
KabluchkoVysotskyZaporozhets2017Faces,KabluchkoMarynych2022}.
For convex hulls of Brownian motion, see
\cite{CranstonHsuMarch1989,Eldan2014,AlexanderEldan2019,
PanzoSocher2026,RandonFurlingMajumdarComtet2009};
for convex hulls of stationary Gaussian processes, see
\cite{DavydovDombry2012};
and for convex hulls of L\'evy and stable processes,
regularly varying processes, and random walks in stable
domains of attraction, see
\cite{MolchanovWespi2016,KampfLastMolchanov2012,
DavydovDombry2014,CyganSandricSebek2022}.
These references are by no means exhaustive. For the statistical-physics perspective, including connections
with extreme-value statistics, see the review by
Majumdar--Comtet--Randon-Furling~\cite{MajumdarComtetRandonFurling2010}.

\medskip
A classical result on origin absorption by random convex hulls
is Wendel's theorem~\cite{Wendel1962}; see also
Schneider--Weil~\cite[Theorem~8.2.1]{SchneiderWeil2008}. For a set $A\subseteq\R^d$, we write $\conv A$ for its convex hull.
Let $Y_1,\ldots,Y_N$ be independent identically distributed random
vectors in $\R^d$ whose common distribution is centrally symmetric
and assigns probability zero to every hyperplane through the origin.
Under these assumptions, Wendel's formula gives
\begin{equation}
\label{eq:intro-wendel-nonabsorption}
\Pp\bigl(0\notin\conv\{Y_1,\ldots,Y_N\}\bigr)
=\frac1{2^{N-1}}\sum_{k=0}^{d-1}\binom{N-1}{k}.
\end{equation}

A related problem concerns origin absorption by the convex hull
of the successive positions of a random walk.
Let $X_1,\ldots,X_n$ be independent identically distributed random
vectors in $\R^d$ whose common distribution is centrally symmetric
and assigns probability zero to every affine hyperplane.
Put $S_0:=0$ and $S_k:=X_1+\cdots+X_k$, $1\le k\le n$.
The following distribution-free formula was obtained
by Kabluchko--Vysotsky--Zaporozhets~\cite[Theorem~1.1]{KabluchkoVysotskyZaporozhets2017}:
\begin{equation}
\label{eq:intro-random-walk-nonabsorption}
\Pp\bigl(0\notin\conv\{S_1,\ldots,S_n\}\bigr)
=\frac{2}{2^n n!}\sum_{j\ge0}B[n,d-1-2j],
\end{equation}
where $B[n,k]$ are the Stirling numbers of type $B$, defined by $\prod_{j=1}^n(t+2j-1)=\sum_{k=0}^n B[n,k]t^k$,
with $B[n,k]=0$ for $k<0$ or $k>n$.
The starting point $S_0=0$ is excluded from the convex hull.

Both formulas are distribution-free: under their respective
assumptions, the non-absorption probabilities depend only on
the dimension and the number of points, but not on the
distributions of the random vectors or random-walk increments.

\medskip
In the present paper, we study related origin-absorption problems
for continuous-time Gaussian processes. In particular, we derive
explicit formulas for the probability that the origin belongs to
the convex hull of a stationary Ornstein--Uhlenbeck process in
$\R^d$ observed over a finite time interval, as well as for
a related circular Ornstein--Uhlenbeck process. We also determine
the origin-absorption probability for the convex hull of a
$d$-dimensional Brownian motion observed on $[s_1,s_2]$, where $0 < s_1 < s_2 < \infty$.

\subsection{Convex hulls of Ornstein--Uhlenbeck processes and loops}
\label{subsec:intro-ou-paths-and-loops}
Fix an integer $d\ge1$. We consider two related stationary
Gaussian processes taking values in $\R^d$, one indexed by the real line and the other
by a circle of prescribed circumference $T>0$.

\medskip
\noindent\emph{The stationary Ornstein--Uhlenbeck process in $\R^d$
and its circular counterpart.}
Let $U=(U_t)_{t\in\R}$ be a centered stationary
\emph{Ornstein--Uhlenbeck (OU) process} taking values in $\R^d$,
with continuous paths and covariance
\begin{equation}
\label{eq:intro-ordinary-ou-covariance}
\Ee[U_tU_s^{\mathsf T}]=\eee^{-|t-s|/2}I_d,
\qquad s,t\in\R,
\end{equation}
where $I_d$ denotes the $d\times d$ identity matrix.
Equivalently, writing $U_t=(U_t^{(1)},\ldots,U_t^{(d)})$,
the coordinate processes $(U_t^{(i)})_{t\in\R}$,
$1\le i\le d$, are independent centered stationary
Gaussian processes with covariance
\[
\Ee[U_s^{(i)}U_t^{(i)}]=\eee^{-|t-s|/2},
\qquad s,t\in\R,\quad 1\le i\le d.
\]

For $T>0$, we define the \emph{circular Ornstein--Uhlenbeck loop
of circumference $T$} as the centered Gaussian process
$V^{(T)}=(V_t^{(T)})_{0\le t\le T}$ taking values in $\R^d$,
with continuous paths and covariance
\begin{equation}
\label{eq:intro-circular-ou-covariance}
\Ee\bigl[V_t^{(T)}(V_s^{(T)})^{\mathsf T}\bigr]
=\frac{\cosh(T/4-|t-s|/2)}{\cosh(T/4)}I_d,
\qquad 0\le s,t\le T.
\end{equation}
Its coordinate processes are independent and have unit variance.
Moreover, $V_0^{(T)}=V_T^{(T)}$ almost surely, and the law of
$V^{(T)}$ is invariant under rotations of the circle obtained
from $[0,T]$ by identifying its endpoints.
Ornstein--Uhlenbeck processes indexed by a circle have been studied
by Norris~\cite{Norris1998} and
Brze\'zniak--van Neerven~\cite{BrzezniakVanNeerven2000}.
The relation to the ordinary OU process is expressed by the
covariance identity
\begin{equation}
\label{eq:intro-ou-covariance-periodization}
\frac{\cosh(T/4-h/2)}{\cosh(T/4)}
=\tanh(T/4)\sum_{k\in\Z}\eee^{-|h+kT|/2},
\qquad 0\le h\le T.
\end{equation}
Thus the covariance function of the circular OU loop is the
periodization of the ordinary OU covariance, normalized to have
value one at zero.
The circular OU loop should not be confused with the Gaussian
conditional law of $U$ given $U_0=U_T$: the latter has nonconstant
conditional variance and is not rotation invariant; see
Remark~\ref{rem:circular-ou-endpoint-conditioning}.

\medskip
\noindent\emph{Exact absorption probabilities for OU convex hulls.}
Define the random convex hulls
\begin{equation}
\label{eq:intro-ou-hulls}
\mathcal H_{d,T}^B:=\conv\{U_t:0\le t\le T\},
\qquad
\mathcal H_{d,T}^A:=\conv\{V_t^{(T)}:0\le t\le T\}.
\end{equation}
We refer to $\mathcal H_{d,T}^B$ and $\mathcal H_{d,T}^A$ as the
\emph{OU convex hulls} of the ordinary OU process and the
circular OU loop, respectively.   The superscripts $B$ and $A$ anticipate the random-walk (type $B$)
and random-bridge (type $A$) structures underlying the two models:
absorption problems for symmetric random walks and for exchangeable
random bridges are related to reflection arrangements and Weyl chambers of types
$B$ and $A$, respectively; see~\cite{KabluchkoVysotskyZaporozhets2017}.
The corresponding Brownian-motion and Brownian-bridge descriptions
of the OU process and the OU loop are given in
Proposition~\ref{prop:ou-brownian-representations}.

\medskip
We now state explicit formulas for the probabilities that the origin
belongs to these two convex hulls.

\begin{theorem}[Exact absorption probabilities]
\label{thm:intro-ou-origin-absorption}
For every integer $d\ge1$ and every $T>0$,
\begin{align}
\Pp(0\in\mathcal H_{d,T}^B)
&=\frac2\pi\int_0^{\arsinh\sqrt{\eee^T-1}}
\frac{\Psi_d(x)\,\dint x}{\sqrt{\eee^T-1-\sinh^2x}},
\label{eq:intro-ordinary-ou-absorption}
\\
\Pp(0\in\mathcal H_{d,T}^A)
&=\frac{\Psi_{d+1}(T/2)}{\sinh(T/2)},
\label{eq:intro-circular-ou-absorption}
\end{align}
where
\begin{equation}
\label{eq:intro-hyperbolic-remainders}
\Psi_m(x):=\sum_{j=0}^\infty\frac{x^{m+2j}}{(m+2j)!},
\qquad m=0,1,\ldots,\quad x\ge0.
\end{equation}
Here $\arsinh$ denotes the inverse hyperbolic sine. The integral in \eqref{eq:intro-ordinary-ou-absorption}
is absolutely convergent.
\end{theorem}

The proof is given in Subsection~\ref{subsec:ou-origin-absorption};
see Theorem~\ref{thm:ou-origin-absorption}.
The functions $\Psi_m$ satisfy $\Psi_0(x)=\cosh x$,
$\Psi_1(x)=\sinh x$, and
$\Psi_{m+2}(x)=\Psi_m(x)-x^m/m!$, $m\geq 0$.
Thus the ordinary OU-process formula~\eqref{eq:intro-ordinary-ou-absorption} involves a single integral,
whereas the circular-loop formula~\eqref{eq:intro-circular-ou-absorption} involves a hyperbolic function
with finitely many Taylor terms removed, divided by $\sinh(T/2)$.

\begin{example}[$d=1,2,3$]\label{exam:OU_convex_hull_nonabsorption_low_dim}
For $d=1$, formulas~\eqref{eq:intro-ordinary-ou-absorption}
and \eqref{eq:intro-circular-ou-absorption} yield
the following non-absorption probabilities:
\begin{equation}
\label{eq:intro-ou-absorption-dimension-one}
\Pp(0\notin\mathcal H_{1,T}^B)
=\frac2\pi\arcsin(\eee^{-T/2}),
\qquad
\Pp(0\notin\mathcal H_{1,T}^A)
=1-\tanh(T/4)=\frac{2}{1+\eee^{T/2}}.
\end{equation}
For a real-valued continuous path, non-absorption means that the
path remains strictly positive or strictly negative throughout
the interval $[0,T]$.
By symmetry, $\Pp(0\notin\mathcal H_{1,T}^B)$ is twice
the classical probability that the OU process stays positive on $[0,T]$; see
Slepian~\cite{Slepian1962} and
Aurzada--Baumgarten~\cite[Eq.~(2.5)]{AurzadaBaumgarten2011}.
In dimension $d=2$, the non-absorption probabilities  are given by
\begin{align}
\Pp(0\notin\mathcal H_{2,T}^B)
&=\frac2\pi\int_0^{\pi/2}
\frac{\dint\theta}{\sqrt{1+(\eee^T-1)\sin^2\theta}}
=\frac2\pi\,\eee^{-T/2}
K\bigl(\sqrt{1-\eee^{-T}}\bigr),
\label{eq:intro-ordinary-ou-absorption-dimension-two}
\\
\Pp(0\notin\mathcal H_{2,T}^A)
&=\frac{T/2}{\sinh(T/2)},
\label{eq:intro-circular-ou-absorption-dimension-two}
\end{align}
where $K(k):=\int_0^{\pi/2} (1-k^2\sin^2\theta)^{-1/2} \dint\theta$,  $0\le k<1$, is the complete elliptic integral of the first kind;  see
\cite[Eqs.~(19.2.4) and (19.2.8)]{NIST:DLMF}.
The corresponding absorption probabilities, which are the
complements of the probabilities displayed above, are given
in Corollary~\ref{cor:ou-origin-absorption-low-dimensions}.
In dimension $d=3$, the non-absorption probabilities can also be
expressed in terms of standard special functions:
\begin{align}
\Pp(0\notin\mathcal H_{3,T}^B)
&=\frac2\pi\left[
\arcsin(\eee^{-T/2})
+\sqrt{\eee^T-1}\,
{}_3F_2\left(
\begin{matrix}
1,1,1\\[2pt]
\frac32,\frac32
\end{matrix}
;1-\eee^T\right)
\right],
\\
\Pp(0\notin\mathcal H_{3,T}^A)
&=\frac{1+T^2/8-\eee^{-T/2}}{\sinh(T/2)},
\end{align}
where ${}_3F_2$ denotes the generalized hypergeometric function.
\end{example}

\medskip
\noindent\emph{Asymptotic results for OU absorption probabilities.}
The exact formulas stated in Theorem~\ref{thm:intro-ou-origin-absorption} yield the following
long-time asymptotics for the OU non-absorption probabilities in
each fixed dimension $d$. The proof is given in
Theorem~\ref{thm:ou-nonabsorption-fixed-dimension}.

\begin{theorem}[Long-time non-absorption in fixed dimension]
\label{thm:intro-ou-nonabsorption-fixed-d}
For every fixed integer $d\ge1$, as $T\to\infty$,
\begin{align}
\Pp(0\notin\mathcal H_{d,T}^B)
&\sim\frac{2^{2-d}}{\pi(d-1)!}\,
T^{d-1}\eee^{-T/2},
\label{eq:intro-ordinary-ou-nonabsorption-asymptotic}
\\
\Pp(0\notin\mathcal H_{d,T}^A)
&\sim\frac{2^{2-d}}{(d-1)!}\,
T^{d-1}\eee^{-T/2}.
\label{eq:intro-circular-ou-nonabsorption-asymptotic}
\end{align}
In particular, $\Pp(0\notin\mathcal H_{d,T}^A)/\Pp(0\notin\mathcal H_{d,T}^B)\to\pi$ as $T\to\infty$.
\end{theorem}

For the ordinary stationary OU process in $\R^d$, the theorem of
Davydov--Dombry~\cite[Theorem~1]{DavydovDombry2012}
implies that, for each fixed $d$, the rescaled convex hull
$(2\log T)^{-1/2}\mathcal H_{d,T}^B$ converges almost surely,
in Hausdorff distance, to the Euclidean unit ball in $\R^d$
as $T\to\infty$. Their theorem applies more generally to
stationary Gaussian processes whose covariance tends to zero,
under suitable regularity assumptions.
In particular, this convergence implies that the probability
of origin non-absorption tends to zero, but does not determine
the precise rate of decay given in
Theorem~\ref{thm:intro-ou-nonabsorption-fixed-d}.

\medskip
We next consider origin absorption by OU convex hulls when both
the dimension $d$ and the time horizon $T$ tend to infinity.
For motivation, we first return to Wendel's theorem.  As the dimension $d$ grows, one may ask how many random i.i.d.\ points
in $\R^d$ are needed for their convex hull to contain the
origin with a probability approaching a prescribed value. Under the assumptions of Wendel's theorem, the formula
\eqref{eq:intro-wendel-nonabsorption}, together with the
central limit theorem for binomial distributions, gives a
precise answer.
Let $N_d$ be a sequence of positive integers satisfying
\begin{equation}
\label{eq:intro-wendel-high-dimensional-scaling}
\frac{N_d-2d}{\sqrt{2d}}\longrightarrow\tau\in\R
\qquad\text{as }d\to\infty.
\end{equation}
For each $d$, let $Y_1,\ldots,Y_{N_d}$ be random vectors in $\R^d$
satisfying the assumptions of Wendel's theorem. Then
\begin{equation}
\label{eq:intro-wendel-high-dimensional-limit}
\begin{aligned}
\Pp\bigl(0\in\conv\{Y_1,\ldots,Y_{N_d}\}\bigr)
=\frac1{2^{N_d-1}}\sum_{k=d}^{N_d-1}
\binom{N_d-1}{k}
=\Pp\bigl(\operatorname{Bin}(N_d-1,1/2)\ge d\bigr)
\longrightarrow\Phi(\tau),
\end{aligned}
\end{equation}
where $\operatorname{Bin}(n,p)$ denotes a binomial random variable
with parameters $n$ and $p$, and
\begin{equation}\label{eq:Phi_def}
\Phi(\tau):=\frac1{\sqrt{2\pi}}\int_{-\infty}^\tau
\eee^{-x^2/2}\,\dint x
\end{equation}
is the standard normal distribution function.

It is natural to ask whether a similar high-dimensional transition
occurs for the convex hulls of Ornstein--Uhlenbeck processes.
The following theorem and its corollary show that both the ordinary OU process
and its circular counterpart exhibit a transition around $T=2d$,
with a common Gaussian limiting profile.

\begin{theorem}[Transition from non-absorption to absorption in high dimension]
\label{thm:intro-ou-high-dimensional-transition}
Let $T_d>0$ satisfy
\begin{equation}
\label{eq:intro-ou-critical-window}
\frac{T_d-2d}{2\sqrt d}\longrightarrow\tau\in\R
\qquad\text{as }d\to\infty.
\end{equation}
Then
\begin{equation}
\label{eq:intro-ou-gaussian-transition}
\Pp(0\in\mathcal H_{d,T_d}^B)\longrightarrow\Phi(\tau),
\qquad
\Pp(0\in\mathcal H_{d,T_d}^A)\longrightarrow\Phi(\tau),
\end{equation}
where $\Phi$ is the standard normal distribution function given by~\eqref{eq:Phi_def}.
\end{theorem}

\begin{corollary}[A linear-time threshold]
\label{cor:intro-ou-linear-time-threshold}
Let $T_d>0$ and suppose that $T_d/d\to c\in[0,\infty]$ with $c\ne2$.
Then, for $\bullet\in\{B,A\}$,
\begin{equation}
\label{eq:intro-ou-linear-time-threshold}
\lim_{d\to\infty}\Pp(0\in\mathcal H_{d,T_d}^\bullet)
=\begin{cases}
0,&c<2,\\
1,&c>2.
\end{cases}
\end{equation}
\end{corollary}

The proof is given in
Theorem~\ref{thm:ou-absorption-high-dimensional-transition} and Corollary~\ref{cor:ou-absorption-linear-time-threshold}. Thus, both the i.i.d.-point model of Wendel's theorem and
the two Ornstein--Uhlenbeck models exhibit a transition
from non-absorption to absorption, centered at $2d$.
The transitions in all three models have Gaussian limiting profiles. The widths of their transition
windows, however, differ: Wendel's model has scale $\sqrt{2d}$
in the number of points, whereas both OU models have scale
$2\sqrt d$ in the time horizon.

\subsection{Convex hulls of Brownian motion and bridges}
\label{subsec:intro-brownian-gaussian-absorption}
We next turn to absorption problems for convex hulls of Brownian
motion and Brownian bridges, which are closely connected with
the OU absorption problems considered above.
We present explicit formulas for absorption probabilities and discuss their
consequences for Brownian paths observed on subintervals,
comparisons with discrete-time random walks and bridges,
and the expected volumes of Brownian convex hulls.

\medskip
\noindent  \emph{Gaussian-point absorption by Brownian convex hulls.}
Let $W=(W_t)_{0\le t\le1}$ be a standard $d$-dimensional Brownian
motion starting at $0$, and define the associated Brownian bridge
from $0$ to $0$ by
$W_t^{\mathrm{br}}:=W_t-tW_1$, $0\le t\le1$. Define
\begin{equation}
\label{eq:intro-brownian-hulls}
\mathcal K_d^B:=\conv\{W_t:0\le t\le1\},
\qquad
\mathcal K_d^A:=\conv\{W_t^{\mathrm{br}}:0\le t\le1\}.
\end{equation}
We refer to $\mathcal K_d^B$ and $\mathcal K_d^A$ as the
\emph{Brownian convex hulls}. Both are almost surely compact.   Let $Z\sim N(0,I_d)$ be
independent of $W$. We determine the probabilities that $\sigma Z$ belongs to
$\mathcal K_d^B$ and to $\mathcal K_d^A$, respectively. The case $\sigma=0$ is immediate, since
both paths contain the origin; the following result concerns
$\sigma>0$.

\begin{theorem}[Gaussian-point absorption by Brownian convex hulls]
\label{thm:intro-brownian-gaussian-absorption}
For every integer $d\ge1$ and every $\sigma>0$,
\begin{align}
\Pp(\sigma Z\in\mathcal K_d^B)
&=\frac2\pi\int_0^{\arsinh(1/\sigma)}
\frac{\Psi_d(x)\,\dint x}{\sqrt{\sigma^{-2}-\sinh^2x}},
\label{eq:intro-type-B-brownian-gaussian-absorption}
\\
\Pp(\sigma Z\in\mathcal K_d^A)
&=\frac{\Psi_{d+1}(\theta_\sigma)}{\sinh\theta_\sigma},
\label{eq:intro-type-A-brownian-gaussian-absorption}
\end{align}
where $\theta_\sigma:=2\arsinh(1/(2\sigma))$ and $\Psi_m$ is
defined in \eqref{eq:intro-hyperbolic-remainders}.
The integral in
\eqref{eq:intro-type-B-brownian-gaussian-absorption}
is absolutely convergent.
\end{theorem}

The proof is given in
Subsections~\ref{subsec:finite-gaussian-hulls-absorption-formula}
and~\ref{subsec:finite-gaussian-hulls-absorption-proof-short};
see Theorem~\ref{thm:brownian-gaussian-absorption}.
In Proposition~\ref{prop:finite-BA-gaussian-absorption},
we determine the Gaussian-point absorption probabilities for
convex hulls of discrete-time Gaussian random walks and bridges.
The formulas of Theorem~\ref{thm:intro-brownian-gaussian-absorption}
are then obtained by passing to the Brownian scaling limit.

\medskip
Absorption of a \emph{deterministic} prescribed point, rather than an
independent Gaussian point, is treated in the companion paper~\cite{MultidimensionalBrownianCoverage}, which uses  Theorem~\ref{thm:intro-brownian-gaussian-absorption} as input
and recovers the prescribed-point absorption probabilities
by inverting the associated Gaussian radial averages.
From a geometric viewpoint, the absorption probability of a prescribed
deterministic point is arguably a more natural quantity than that
of an independent Gaussian point. The latter, however, is particularly
well suited to the methods of conic integral geometry and leads to
the explicit formulas obtained here. The prescribed-point problem for planar Brownian motion and Brownian
bridges was recently studied in~\cite{Kabluchko2026PlanarBrownianCoverage},
using a method based on planar convex geometry that does not directly extend to higher dimensions.

\medskip
\noindent\emph{The correspondence with Ornstein--Uhlenbeck processes.}
Suitable increasing time changes and positive deterministic rescalings
transform the shifted paths $W_t-\sigma Z$ and
$W_t^{\mathrm{br}}-\sigma Z$ into processes having the same laws
as the stationary OU process and the circular OU loop,
respectively; see Proposition~\ref{prop:ou-brownian-representations}.
Such time changes and rescalings preserve whether the convex hull
of a path contains the origin. Consequently, for every $T>0$,
\begin{align}
\Pp(0\in\mathcal H_{d,T}^B)
&=\Pp\left(\frac{Z}{\sqrt{\eee^T-1}}\in\mathcal K_d^B\right),
\label{eq:intro-ordinary-ou-brownian-correspondence}
\\
\Pp(0\in\mathcal H_{d,T}^A)
&=\Pp\left(\frac{Z}{2\sinh(T/4)}\in\mathcal K_d^A\right).
\label{eq:intro-circular-ou-brownian-correspondence}
\end{align}
Since the transformations between $T>0$ and the corresponding
Gaussian scale $\sigma>0$ are bijective, these identities establish
the equivalence of Theorems~\ref{thm:intro-brownian-gaussian-absorption}
and~\ref{thm:intro-ou-origin-absorption}.

\medskip
\noindent\emph{Origin absorption by Brownian motion and
Brownian bridges on subintervals.}
Theorem~\ref{thm:intro-brownian-gaussian-absorption} also gives an
explicit answer to a problem posed by Dospolova
in~\cite[pp.~59--60]{Dospolova2023}: to determine the probability
that the origin is contained in the convex hull of a Brownian motion
observed on $[t_0,1]$, where $0<t_0<1$.

More generally, consider a standard $d$-dimensional Brownian motion
$W$ observed on an interval $[s_1,s_2]$, where
$0<s_1<s_2<\infty$. By independent increments and Brownian scaling,
\begin{align}
\Pp\bigl(0\in\conv\{W_t:s_1\le t\le s_2\}\bigr)
=\Pp\left(
\sqrt{\frac{s_1}{s_2-s_1}}\,Z\in\mathcal K_d^B
\right)
=\frac2\pi
\int_0^{\arsinh\sqrt{\frac{s_2-s_1}{s_1}}}
\frac{\Psi_d(x)\,\dint x}
{\sqrt{\frac{s_2-s_1}{s_1}-\sinh^2x}}.
\label{eq:intro-delayed-brownian-absorption}
\end{align}
In particular, the probability depends only on the ratio $s_2/s_1$.
By \eqref{eq:intro-ordinary-ou-brownian-correspondence},
it also equals
$\Pp(0\in\mathcal H_{d,\log(s_2/s_1)}^B)$.

A similar connection holds for the Brownian bridge
$W^{\mathrm{br}}$ on $[0,1]$. Define the increasing bijection
$u:(0,1)\to\R$ by $u(t):=\log\frac{t}{1-t}$, $0<t<1$,
and consider the time-changed and rescaled process
$X=(X_u)_{u\in\R}$ defined by
\[
X_{u(t)}:=\frac{W_t^{\mathrm{br}}}{\sqrt{t(1-t)}},
\qquad 0<t<1.
\]
Then $X$ is a stationary OU process with covariance
\eqref{eq:intro-ordinary-ou-covariance}.
In particular, the interval $[s_1,s_2]$, where
$0<s_1<s_2<1$, corresponds to an OU time interval of length
\[
u(s_2)-u(s_1)
=\log\frac{s_2(1-s_1)}{s_1(1-s_2)}.
\]
Since positive rescalings preserve origin absorption, it follows
that, for $0<s_1<s_2<1$,
\begin{align}
\Pp\bigl(0\in\conv\{W_t^{\mathrm{br}}:
s_1\le t\le s_2\}\bigr)
=\Pp\left(
0\in\mathcal H_{d,\log\frac{s_2(1-s_1)}{s_1(1-s_2)}}^B
\right)
=\Pp\left(
\sqrt{\frac{s_1(1-s_2)}{s_2-s_1}}\,
Z\in\mathcal K_d^B
\right).
\label{eq:intro-restricted-bridge-absorption}
\end{align}
Thus, the origin-absorption probabilities for Brownian bridges
restricted to subintervals $[s_1,s_2]$, where $0<s_1<s_2<1$,
can be obtained from
\eqref{eq:intro-type-B-brownian-gaussian-absorption}
by substituting
$\sigma=\sqrt{s_1(1-s_2)/(s_2-s_1)}$.

\begin{example}[$d=2$]
In dimension $d=2$, the probability of origin non-absorption
by Brownian motion over $[s_1,s_2]$, where
$0<s_1<s_2<\infty$, simplifies to
\[
\Pp\bigl(0\notin\conv\{W_t:s_1\le t\le s_2\}\bigr)
=\frac2\pi\sqrt{\frac{s_1}{s_2}}\,
K\left(\sqrt{1-\frac{s_1}{s_2}}\right),
\]
where $K$ is the complete elliptic integral of the first kind,
defined in Example~\ref{exam:OU_convex_hull_nonabsorption_low_dim}.
In particular, for $s_1=1$ and $s_2=2$, the classical special value $K(1/\sqrt2)=\Gamma(1/4)^2/(4\sqrt\pi)$ (see \cite[Eq.~(23.5.5)]{NIST:DLMF})
gives
\[
\Pp\bigl(0\notin\conv\{W_t:1\le t\le2\}\bigr)
=\frac{\Gamma(1/4)^2}{2\sqrt2\,\pi^{3/2}}.
\]
For the planar Brownian bridge restricted to an interval
$[s_1,s_2]$, where $0<s_1<s_2<1$, the corresponding
non-absorption probability is
\[
\Pp\bigl(0\notin\conv\{W_t^{\mathrm{br}}:
s_1\le t\le s_2\}\bigr)
=\frac2\pi
\sqrt{\frac{s_1(1-s_2)}{s_2(1-s_1)}}\,
K\left(\sqrt{\frac{s_2-s_1}{s_2(1-s_1)}}\right).
\]
\end{example}

\medskip
\noindent\emph{Comparison with random walks and random-walk bridges.}
We now compare origin-absorption probabilities for continuous-time
Brownian paths with those for their discrete-time skeletons.
Let $W=(W_t)_{t\ge0}$ be a standard $d$-dimensional Brownian
motion starting at $0$, and let
$W^{\mathrm{br}}=(W_t^{\mathrm{br}})_{0\le t\le1}$ be the
associated Brownian bridge, defined by
$W_t^{\mathrm{br}}:=W_t-tW_1$.
For each integer $n\ge3$, we compare the convex hull of
$W_1,\ldots,W_n$ with that of the Brownian path on $[1,n]$,
and the convex hull of $W_{k/n}^{\mathrm{br}}$, $1\le k\le n-1$,
with that of the bridge path on $[1/n,1-1/n]$.
In each pair, the discrete-time hull is contained in the continuous-time hull.
The first discrete-time model is a Gaussian random walk; in the second,
the sampled positions, after multiplication by $\sqrt n$ and
inclusion of the two zero endpoints, form a Gaussian random-walk
bridge with $n$ increments. The endpoints are omitted from the
convex hull, and the common positive scaling does not affect origin
absorption. Thus the discrete-time absorption probabilities are Gaussian
instances of the distribution-free results
in~\cite{KabluchkoVysotskyZaporozhets2017}.

We first compare the non-absorption probabilities in fixed dimension $d\geq 1$. Combining \eqref{eq:intro-delayed-brownian-absorption} with
\eqref{eq:intro-ordinary-ou-brownian-correspondence}, and using
\eqref{eq:intro-restricted-bridge-absorption}, we see that the
two continuous-time probabilities are those of ordinary OU convex hulls
with time horizons $\log n$ and $2\log(n-1)$, respectively.
Theorem~\ref{thm:intro-ou-nonabsorption-fixed-d} therefore yields,
for every fixed dimension $d\ge1$, as $n\to\infty$,
\begin{align}
\Pp\bigl(0\notin\conv\{W_t:1\le t\le n\}\bigr)
&\sim\frac{2^{2-d}}{\pi(d-1)!}\,
\frac{(\log n)^{d-1}}{\sqrt n},
\label{eq:intro-continuous-walk-nonabsorption-asymptotic}
\\
\Pp\bigl(0\notin\conv\{W_t^{\mathrm{br}}:
1/n\le t\le1-1/n\}\bigr)
&\sim\frac{2}{\pi(d-1)!}\,
\frac{(\log n)^{d-1}}{n}.
\label{eq:intro-continuous-bridge-nonabsorption-asymptotic}
\end{align}
On the other hand, \cite[Theorem~5.1]{KabluchkoVysotskyZaporozhets2017}
gives, for every fixed dimension $d\ge2$, as $n\to\infty$,
\begin{align*}
\Pp\bigl(0\notin\conv\{W_1,\ldots,W_n\}\bigr)
&\sim\frac{2^{2-d}}{\sqrt\pi(d-1)!}\,
\frac{(\log n)^{d-1}}{\sqrt n},
\\
\Pp\bigl(0\notin\conv\{W_{k/n}^{\mathrm{br}}:
1\le k\le n-1\}\bigr)
&\sim\frac{2}{(d-1)!}\,
\frac{(\log n)^{d-1}}{n}.
\end{align*}
The planar cases of these discrete-time asymptotics were obtained earlier
by Vysotsky--Zaporozhets~\cite{VysotskyZaporozhets2018}.
Both asymptotic formulas remain valid for $d=1$, as follows
from Lemma~\ref{lem:type-B-ordinary-block-probabilities}
and Stirling's formula.
Consequently, for every fixed dimension $d\ge1$,
\begin{align*}
\lim_{n\to\infty}
\frac{
\Pp\bigl(0\notin\conv\{W_t:1\le t\le n\}\bigr)
}{
\Pp\bigl(0\notin\conv\{W_1,\ldots,W_n\}\bigr)
}
&=\frac1{\sqrt\pi},
\\
\lim_{n\to\infty}\frac{
\Pp\bigl(0\notin\conv\{W_t^{\mathrm{br}}:
1/n\le t\le1-1/n\}\bigr)
}{
\Pp\bigl(0\notin\conv\{W_{k/n}^{\mathrm{br}}:
1\le k\le n-1\}\bigr)
}
&=\frac1\pi.
\end{align*}
Thus continuous-time observation changes the leading constant, but not
the asymptotic order, of the non-absorption probability in each fixed
dimension.

The correspondence is even closer when the dimension grows.
For each $d\ge1$, let $n_d\ge3$ be an integer, and consider
the discrete-time and continuous-time convex hulls described above
in $\R^d$, with $n=n_d$.
For the random-walk and Brownian-motion pair,
the absorption probabilities of both the discrete-time hull and the
continuous-time hull converge to $\Phi(\tau)$ if
\begin{equation}
\label{eq:intro-walk-discrete-continuous-critical-window}
\frac{\log n_d-2d}{2\sqrt d}\longrightarrow\tau\in\R
\qquad\text{as }d\to\infty.
\end{equation}
For the random-walk-bridge and Brownian-bridge pair, the same
conclusion holds if
\begin{equation}
\label{eq:intro-bridge-discrete-continuous-critical-window}
\frac{\log n_d-d}{\sqrt d}\longrightarrow\tau\in\R
\qquad\text{as }d\to\infty.
\end{equation}
The discrete-time results follow from
\cite[Theorem~5.2]{KabluchkoVysotskyZaporozhets2017}.
The continuous-time results follow from
Theorem~\ref{thm:intro-ou-high-dimensional-transition}, using
$T_d=\log n_d$ and $T_d=2\log(n_d-1)$, respectively.
Thus the discrete-time and continuous-time models in each pair have matching
transition locations, transition
window scales, and Gaussian limiting profiles.

Earlier estimates for high-dimensional origin
absorption were obtained by
Tikhomirov--Youssef~\cite{TikhomirovYoussef2017}.
For suitable universal constants $0<c<C$, their results show that
$n\ge\eee^{Cd}$ steps suffice for interior absorption with probability
tending to one as $d\to\infty$, for both Gaussian random walks in $\R^d$ and
the standard lattice walk in $\Z^d$.
They also proved that the Brownian convex hull over $[1,\eee^{cd}]$ excludes
the origin with probability tending to one as $d\to\infty$.
For Gaussian random walks and Brownian motion observed on $[1,n_d]$,
the results discussed above provide a more precise description:
in both models, the absorption probability has a Gaussian
limiting profile when
$(\log n_d-2d)/(2\sqrt d)\to\tau$ as $d\to\infty$, with limit $\Phi(\tau)$.

\medskip
\noindent\emph{Large-$\sigma$ asymptotics and expected Brownian convex hull volumes.}
The large-$\sigma$ asymptotics of the Gaussian-point absorption
probabilities from Theorem~\ref{thm:intro-brownian-gaussian-absorption} allow us to recover the expected volumes of the
two Brownian convex hulls.
Write $\operatorname{vol}_d$ for $d$-dimensional Lebesgue measure and
$\kappa_d:=\pi^{d/2}/\Gamma(d/2+1)$ for the volume of the Euclidean
unit ball in $\R^d$.

\begin{corollary}[Large-$\sigma$ asymptotics and expected Brownian convex hull volumes]
\label{cor:intro-brownian-gaussian-large-scale-and-volumes}
For every fixed integer $d\ge1$, as $\sigma\to\infty$,
\begin{align}
\Pp(\sigma Z\in\mathcal K_d^B)
&\sim\frac{\kappa_d^2}{(2\pi)^d}\,\sigma^{-d},
\label{eq:intro-type-B-brownian-gaussian-large-scale}
\\
\Pp(\sigma Z\in\mathcal K_d^A)
&\sim\frac1{(d+1)!}\,\sigma^{-d}.
\label{eq:intro-type-A-brownian-gaussian-large-scale}
\end{align}
Consequently,
\begin{equation}
\label{eq:intro-brownian-hull-expected-volumes}
\Ee\,\operatorname{vol}_d(\mathcal K_d^B)
=\frac{\kappa_d^2}{(2\pi)^{d/2}},
\qquad
\Ee\,\operatorname{vol}_d(\mathcal K_d^A)
=\frac{(2\pi)^{d/2}}{(d+1)!}.
\end{equation}
\end{corollary}

The proof is given in
Corollary~\ref{cor:brownian-gaussian-large-scale-and-volumes}.
The connection with expected volumes follows directly by conditioning on
the convex hull:
\begin{equation}
\label{eq:intro-brownian-gaussian-density-integral}
\sigma^d\Pp(\sigma Z\in\mathcal K_d^\bullet)
=\frac1{(2\pi)^{d/2}}
\Ee\int_{\mathcal K_d^\bullet}
\eee^{-\|x\|^2/(2\sigma^2)}\,\dint x,
\qquad \bullet\in\{B,A\}.
\end{equation}
As $\sigma\to\infty$, the integrand increases to one.
By monotone convergence, the limit of the left-hand side
is $(2\pi)^{-d/2}$ times the expected hull volume.
The expected-volume formulas
\eqref{eq:intro-brownian-hull-expected-volumes}
are already known: see Eldan~\cite[Theorem~1]{Eldan2014}
for Brownian motion in $d\ge2$, and
Kabluchko--Zaporozhets~\cite[Section~1.5]{KabluchkoZaporozhets2016}
for both Brownian motion and Brownian bridges.

\subsection{Ornstein--Uhlenbeck spirals in a Hilbert space}
\label{subsec:intro-ou-spirals-conic-intrinsic-volumes}
A natural geometric viewpoint is to regard the random variables
of a stochastic process as vectors in a Hilbert space.
To make this precise, let $I\subseteq\R$ be an interval, and let
$(X_t)_{t\in I}$ be a real-valued, mean-square continuous
stochastic process with finite second moments, defined on
a probability space $(\Omega,\mathcal F,\Pp)$.
Regarding each $X_t$ as an element of
$L^2(\Omega,\mathcal F,\Pp)$, we obtain a continuous
deterministic curve $t\mapsto X_t$ in $L^2(\Omega,\mathcal F,\Pp)$ whose Gram kernel is
given by
\[
\langle X_s,X_t\rangle=\Ee[X_sX_t],
\qquad s,t\in I.
\]
This geometric point of view goes back to the work of
Kolmogorov~\cite{Kolmogorov1940} on Wiener spirals in Hilbert space.

\medskip
\noindent\emph{Ornstein--Uhlenbeck spirals  in Hilbert space and their cones.}
Applying this construction to the one-dimensional versions of
the ordinary and circular Ornstein--Uhlenbeck processes introduced
above, we obtain two deterministic curves in real Hilbert spaces.
For $T>0$, denote these curves on $[0,T]$ by
$\zeta_T^B=(\zeta_T^B(t))_{0\le t\le T}$ and
$\zeta_T^A=(\zeta_T^A(t))_{0\le t\le T}$, respectively.
Their Gram kernels are
\begin{align}
\langle\zeta_T^B(s),\zeta_T^B(t)\rangle
&=\eee^{-|t-s|/2},
\qquad 0\le s,t\le T,
\label{eq:intro-ordinary-ou-spiral-kernel}
\\
\langle\zeta_T^A(s),\zeta_T^A(t)\rangle
&=\frac{\cosh(T/4-|t-s|/2)}{\cosh(T/4)},
\qquad 0\le s,t\le T.
\label{eq:intro-circular-ou-spiral-kernel}
\end{align}
We call them the \emph{ordinary} and \emph{circular
Ornstein--Uhlenbeck spirals}, respectively. The curves consist of unit vectors, are continuous
in norm, and satisfy $\zeta_T^A(0)=\zeta_T^A(T)$ in the circular case.

For $T>0$, define the closed positive hulls of the two spirals by
\begin{equation}
\label{eq:intro-ou-spiral-cones}
\bU_{\infty, T}^\bullet
:=\posclos\{\zeta_T^\bullet(t):0\le t\le T\},
\qquad \bullet\in\{B,A\},
\end{equation}
where $\posclos M$ denotes the norm closure of the set of all
finite nonnegative linear combinations of elements of $M$. We refer to $\bU_{\infty, T}^B$ and $\bU_{\infty, T}^A$ as the
\emph{OU spiral cones} generated by the ordinary and circular
OU spirals, respectively.
These closed convex cones are nonzero and pointed, meaning that
$\bU_{\infty, T}^\bullet\cap(-\bU_{\infty, T}^\bullet)=\{0\}$;
see Proposition~\ref{prop:hilbert-brownian-cones-ou-spirals}
and Lemma~\ref{lem:hilbert-brownian-bodies-and-cones}.
The Gram kernels determine the cones up to linear isometry.

\medskip
We now study the conic intrinsic volumes and Grassmann angles
of these deterministic Hilbert-space cones.
Although these are deterministic geometric quantities,
the Grassmann angles turn out to coincide exactly with the
origin-absorption probabilities for OU convex hulls from
Subsection~\ref{subsec:intro-ou-paths-and-loops}.
Through the Brownian--OU correspondence, they are also related
to the Gaussian-point absorption probabilities studied in
Subsection~\ref{subsec:intro-brownian-gaussian-absorption}.

\medskip
\noindent\emph{Conic intrinsic volumes and Grassmann angles.}
We first recall these notions for closed convex cones in $\R^D$
and then describe their extension to separable Hilbert spaces.
Let $C\subseteq\R^D$ be a closed convex cone. For $0\le j\le D$,
its $j$th \emph{Grassmann angle} $\gamma_j(C)$ is the probability that
$C$ intersects a uniformly distributed linear subspace of
codimension $j$ in a nonzero vector. For $j>D$, put
$\gamma_j(C):=0$.
These probabilities are related to the \emph{conic intrinsic volumes}
$\upsilon_0(C), \upsilon_1(C),\ldots$ by the conic Crofton formula
\begin{equation}
\label{eq:intro-finite-conic-crofton}
\gamma_j(C)=2\sum_{\ell=0}^\infty
\upsilon_{j+1+2\ell}(C),\qquad j\ge0,
\end{equation}
provided that $C$ is not a linear subspace.
The conic intrinsic volumes are nonnegative, sum to one,
and vanish above dimension $D$.
For their definition and further properties, see
\cite{Schneider2022,McCoyTropp2014,
AmelunxenLotzMcCoyTropp2014,AmelunxenLotz2017}
and Section~\ref{sec:preliminaries} of the present paper.

These notions extend to infinite-dimensional Hilbert spaces.
For a closed convex cone $C$ in a real separable Hilbert space,
following Dospolova~\cite{Dospolova2023}, define
\begin{equation}
\label{eq:intro-hilbert-grassmann-angles}
\gamma_j(C):=\sup_{C_0\subseteq C}\gamma_j(C_0),
\qquad j\ge0,
\end{equation}
where $C_0$ ranges over finite-dimensional closed convex cones
and its Grassmann angles are computed in its linear span.
For a cone which is not a linear subspace, the conic intrinsic
volumes are then defined, following
Dospolova--Zaporozhets~\cite{DospolovaZaporozhets2024}, by
\begin{equation}
\label{eq:intro-hilbert-conic-intrinsic-volumes}
\upsilon_0(C):=1-\frac{\gamma_0(C)+\gamma_1(C)}2,
\qquad
\upsilon_k(C):=\frac{\gamma_{k-1}(C)-\gamma_{k+1}(C)}2,
\quad k\ge1.
\end{equation}
These definitions agree with the finite-dimensional ones
for cones which are not linear subspaces.
Further details  are
recalled in  Subsection~\ref{subsec:hilbert-conic-intrinsic-volumes}.

For every closed convex cone in a finite-dimensional Euclidean
space, the conic intrinsic volumes sum to one.
In infinite-dimensional Hilbert spaces, this need not be the case,
even for nonzero closed pointed convex cones; see
Remark~\ref{rem:hilbert-conic-intrinsic-volumes-total-mass}
for an example in which all conic intrinsic volumes vanish.

\medskip
\noindent\emph{Conic intrinsic volumes of the OU spiral cones.}
The following theorem gives the full intrinsic-volume sequences of
both cones in \eqref{eq:intro-ou-spiral-cones}.

\begin{theorem}[Conic intrinsic volumes of the OU spiral cones]
\label{thm:intro-ou-spiral-intrinsic-volumes}
For every $T>0$, the ordinary Ornstein--Uhlenbeck spiral cone satisfies
\begin{align}
\upsilon_0(\bU_{\infty, T}^B)
&=\frac1\pi\arcsin(\eee^{-T/2}),
\label{eq:intro-ordinary-ou-spiral-zeroth-volume}
\\
\upsilon_k(\bU_{\infty, T}^B)
&=\frac1{\pi(k-1)!}
\int_0^{\arsinh\sqrt{\eee^T-1}}
\frac{x^{k-1}\,\dint x}{\sqrt{\eee^T-1-\sinh^2x}},
\qquad k\ge1.
\label{eq:intro-ordinary-ou-spiral-intrinsic-volumes}
\end{align}
For the circular Ornstein--Uhlenbeck spiral cone,
\begin{align}
\upsilon_0(\bU_{\infty, T}^A)
&=\frac12\left(1-\tanh\frac T4\right),
\label{eq:intro-circular-ou-spiral-zeroth-volume}
\\
\upsilon_k(\bU_{\infty, T}^A)
&=\frac{(T/2)^k}{2k!\sinh(T/2)},
\qquad k\ge1.
\label{eq:intro-circular-ou-spiral-intrinsic-volumes}
\end{align}
For each cone, the sum over the even indices and the sum over the
odd indices are both $1/2$; in particular, the total mass is one.
\end{theorem}

The proof is given in
Theorem~\ref{thm:hilbert-brownian-ou-intrinsic-volumes}.
Explicit entire generating functions for both sequences are also
obtained in Subsection~\ref{subsec:conic-generating-functions}.
Since the conic intrinsic volumes of each spiral cone sum to one,
the defining relations~\eqref{eq:intro-hilbert-conic-intrinsic-volumes}
can be inverted by telescoping to recover the Grassmann angles.
Together with the absorption formulas from Theorem~\ref{thm:intro-ou-origin-absorption},
this yields the following result, proved in
Corollary~\ref{cor:hilbert-ou-grassmann-absorption}.

\begin{corollary}[Grassmann angles and origin absorption]
\label{cor:intro-ou-spiral-grassmann-absorption}
For every integer $d\ge1$, every $T>0$, and $\bullet\in\{B,A\}$,
\begin{equation}
\label{eq:intro-ou-spiral-grassmann-absorption}
\gamma_d(\bU_{\infty, T}^\bullet)
=2\sum_{\ell=0}^\infty\upsilon_{d+1+2\ell}(\bU_{\infty, T}^\bullet)
=\Pp(0\in\mathcal H_{d,T}^\bullet).
\end{equation}
\end{corollary}

\medskip
\noindent\emph{Random OU cones.}
We now consider the closed positive hulls of the ordinary OU process
and the circular OU loop in $\R^d$, defined by
\begin{equation}
\label{eq:intro-random-ou-cones}
\mathcal U_{d,T}^\bullet
:=\posclos\mathcal H_{d,T}^\bullet\subseteq\R^d,
\qquad \bullet\in\{B,A\}.
\end{equation}
We refer to these as the \emph{random OU cones}.
Unlike the deterministic OU spiral cones in
\eqref{eq:intro-ou-spiral-cones}, they are generated by
random trajectories in $\R^d$.

\begin{proposition}[Random OU cones]
\label{prop:intro-random-ou-expected-intrinsic-volumes}
For every integer $d\ge1$, every $T>0$, and
$\bullet\in\{B,A\}$, the random cone
$\mathcal U_{d,T}^\bullet$ is full-dimensional almost surely.
Moreover, almost surely, it is either pointed or equal to $\R^d$,
and
\begin{equation}
\label{eq:intro-random-ou-full-space-probability}
\Pp(\mathcal U_{d,T}^\bullet=\R^d)
=\Pp(0\in\mathcal H_{d,T}^\bullet)
=\gamma_d(\bU_{\infty, T}^\bullet).
\end{equation}
For every $k\in\{0,\ldots,d\}$, the expected conic intrinsic volumes of $\mathcal U_{d,T}^\bullet$ are
\begin{equation}
\label{eq:intro-random-ou-expected-intrinsic-volumes}
\Ee\,\upsilon_k(\mathcal U_{d,T}^\bullet)
=\begin{cases}
\upsilon_k(\bU_{\infty, T}^\bullet),&0\le k<d,\\[2pt]
\sum_{\ell=d}^\infty\upsilon_\ell(\bU_{\infty, T}^\bullet),&k=d.
\end{cases}
\end{equation}
The expected solid angle of $\mathcal U_{d,T}^\bullet$ is
\begin{equation}
\label{eq:intro-random-ou-expected-solid-angle}
\Ee\,\alpha(\mathcal U_{d,T}^\bullet)
=\Ee\,\upsilon_d(\mathcal U_{d,T}^\bullet)
=\sum_{\ell=d}^\infty\upsilon_\ell(\bU_{\infty, T}^\bullet)
=\frac{\gamma_{d-1}(\bU_{\infty, T}^\bullet)
+\gamma_d(\bU_{\infty, T}^\bullet)}2.
\end{equation}
Explicitly,
\begin{align}
\Ee\,\alpha(\mathcal U_{d,T}^B)
&=\frac1\pi\int_0^{\arsinh\sqrt{\eee^T-1}}
\frac{\Psi_{d-1}(x)+\Psi_d(x)}
{\sqrt{\eee^T-1-\sinh^2x}}\,\dint x,
\label{eq:intro-ordinary-random-ou-expected-solid-angle}
\\
\Ee\,\alpha(\mathcal U_{d,T}^A)
&=\frac{\Psi_d(T/2)+\Psi_{d+1}(T/2)}
{2\sinh(T/2)}.
\label{eq:intro-circular-random-ou-expected-solid-angle}
\end{align}
\end{proposition}

The almost-sure geometric assertions follow from
Lemma~\ref{lem:random-ou-cones-geometry}, while
\eqref{eq:intro-random-ou-full-space-probability}
follows from that lemma and
Corollary~\ref{cor:intro-ou-spiral-grassmann-absorption}.
The full-space probabilities are evaluated explicitly in
Theorem~\ref{thm:intro-ou-origin-absorption}.
Equation~\eqref{eq:intro-random-ou-expected-intrinsic-volumes}
follows by combining
Theorems~\ref{thm:random-ou-cones-expected-intrinsic-volumes}
and~\ref{thm:hilbert-brownian-ou-intrinsic-volumes}.
Since the random OU cones are full-dimensional almost surely,
their solid angles coincide with their top-dimensional conic
intrinsic volumes. The last equality in
\eqref{eq:intro-random-ou-expected-solid-angle}
follows by telescoping the relations between conic intrinsic
volumes and Grassmann angles and using the fact that
the intrinsic volumes of the OU spiral cones sum to one.
Finally, \eqref{eq:intro-ordinary-random-ou-expected-solid-angle}
and \eqref{eq:intro-circular-random-ou-expected-solid-angle}
follow from \eqref{eq:ordinary-random-ou-expected-top-volume}
and \eqref{eq:circular-random-ou-expected-top-volume},
respectively, in
Theorem~\ref{thm:random-ou-cones-expected-intrinsic-volumes}. Further results related to
\eqref{eq:intro-random-ou-expected-solid-angle},
including conditional expected solid angles given that
the random OU cone is not equal to $\R^d$,
are provided in
Corollary~\ref{cor:random-ou-cones-expected-solid-angles}.

\medskip
The general formulas for expected Grassmann angles  and solid angles of Gaussian images of convex cones were established
in finite dimensions by Götze--Kabluchko--Zaporozhets
\cite[Theorem~3.5 and Corollary~3.7]{GotzeKabluchkoZaporozhets2023}.
Dospolova~\cite[Theorems~5 and~6]{Dospolova2023}
extended this theory to suitable
cones in infinite-dimensional Hilbert spaces, using Gaussian
spectra defined through isonormal processes; see also
Dospolova--Zaporozhets~\cite{DospolovaZaporozhets2024}
for the theory of infinite-dimensional conic intrinsic volumes.
In our setting, let $\xi_1,\ldots,\xi_d$ be independent isonormal
processes on the closed linear span of $\zeta_T^\bullet$.
The random curve
\[
\bigl(\xi_1(\zeta_T^\bullet(u)),\ldots,
\xi_d(\zeta_T^\bullet(u))\bigr),\qquad 0\le u\le T,
\]
has a continuous version with the law of the corresponding
$d$-dimensional OU process or circular OU loop.
Its closed positive hull therefore has the law of
$\mathcal U_{d,T}^\bullet$.
Thus, the general relationship between the deterministic OU spiral cone $\bU_{\infty, T}^\bullet$ and the random OU cone $\mathcal U_{d,T}^\bullet$ underlying Proposition~\ref{prop:intro-random-ou-expected-intrinsic-volumes}
fits into the Gaussian-spectrum framework.
Here we prove Proposition~\ref{prop:intro-random-ou-expected-intrinsic-volumes} without invoking the general Gaussian-spectrum theorems from \cite{Dospolova2023}.

\medskip
\noindent\emph{Small-opening limits and ordinary intrinsic volumes.}
The conic intrinsic volumes are connected with the ordinary Euclidean
intrinsic volumes $V_j(K)$ of compact convex sets through a natural
limiting procedure. In finite
dimensions, the ordinary Euclidean intrinsic volumes of a nonempty compact convex set $K\subseteq\R^m$ are defined by the Steiner formula
\begin{equation}
\label{eq:intro-euclidean-steiner-normalization}
\operatorname{vol}_m(K+rB_m)
=\sum_{j=0}^m\kappa_{m-j}V_j(K)r^{m-j},
\qquad r\ge0,
\end{equation}
where  $B_m$ is
the closed unit ball in $\R^m$, and $\kappa_j$ denotes the volume of the unit
ball in $\R^j$, as in
Subsection~\ref{subsec:intro-brownian-gaussian-absorption}, with
$\kappa_0:=1$. For a nonempty compact convex set $K$ in a real
separable Hilbert space, the definition is extended by
$V_j(K):=\sup_{K'\subseteq K}V_j(K')$, where $K'$ ranges over
nonempty finite-dimensional compact convex subsets; the supremum
is allowed to be infinite.
The theory originates in the work of Sudakov~\cite{Sudakov1976} and
Chevet~\cite{Chevet1976} and was further developed by
Tsirelson~\cite{Tsirelson1986}, who related these intrinsic volumes
to Gaussian spectra; see also Vitale~\cite{Vitale2001} and
\cite[Sections~1.1 and~1.4]{KabluchkoZaporozhets2016}.

\medskip
Raising a compact base $K$ farther from the origin produces an
increasingly narrow cone. The following theorem recovers the
ordinary intrinsic volumes of the base from the rescaled conic
intrinsic volumes and Grassmann angles.

\begin{theorem}[Small-opening limits]
\label{thm:intro-small-opening-intrinsic-volumes}
Let $K$ be a nonempty compact convex subset of a real separable
Hilbert space $\mathbb H$, and suppose that $V_1(K)<\infty$.
Adjoin a unit vector $e_0$ orthogonal to $\mathbb H$, and set
\begin{equation}
\label{eq:intro-small-opening-cones}
C_\sigma(K):=\pos(\sigma e_0+K) \subseteq \R e_0 \oplus \mathbb H,\qquad \sigma>0.
\end{equation}
These cones are closed and pointed. For every integer $j\ge0$,
\begin{align}
\lim_{\sigma\to\infty}\sigma^j\gamma_j(C_\sigma(K))
&=\frac{2V_j(K)}{(j+1)\kappa_{j+1}},
\label{eq:intro-small-opening-grassmann-limit}
\\
\lim_{\sigma\to\infty}\sigma^j\upsilon_{j+1}(C_\sigma(K))
&=\frac{V_j(K)}{(j+1)\kappa_{j+1}}.
\label{eq:intro-small-opening-conic-limit}
\end{align}
\end{theorem}

The proof is given in
Theorem~\ref{thm:general-small-opening-intrinsic-volumes}.

\medskip
\noindent\emph{Brownian bodies and their intrinsic volumes.}
We now apply Theorem~\ref{thm:intro-small-opening-intrinsic-volumes}
to recover the ordinary intrinsic volumes of the Brownian-motion
and Brownian-bridge bodies. For $0\le t\le1$, consider the vectors
$h_t^B:=\ind_{[0,t]}$ and
$h_t^A:=\ind_{[0,t]}-t\ind_{[0,1]}$ in the real Hilbert space
$L^2[0,1]$, where $\ind$ denotes an indicator function.
The \emph{Brownian-motion} and \emph{Brownian-bridge bodies} are
\begin{equation}
\label{eq:intro-hilbert-brownian-bodies}
\bK_\infty^B:=\convclos\{h_t^B:0\le t\le1\},
\qquad
\bK_\infty^A:=\convclos\{h_t^A:0\le t\le1\}.
\end{equation}
Here $\convclos$ denotes the norm-closure of the convex hull. Both bodies
are compact.  They are deterministic Hilbert-space sets, unlike
 the random Brownian convex hulls $\mathcal K_d^B$ and $\mathcal K_d^A$
in \eqref{eq:intro-brownian-hulls}.

The connection with Theorem~\ref{thm:intro-small-opening-intrinsic-volumes} is given by the linear
isometries
\begin{align}
C_\rho(\bK_\infty^B)
&\cong \bU_{\infty, \log(1+\rho^{-2})}^B,
\qquad \rho>0,
\label{eq:intro-brownian-ordinary-ou-cone-isometry}
\\
C_\rho(\bK_\infty^A)
&\cong\bU_{\infty, 4\arsinh(1/(2\rho))}^A,
\qquad \rho>0,
\label{eq:intro-brownian-circular-ou-cone-isometry}
\end{align}
where $e_0$ is adjoined orthogonally to $L^2[0,1]$.
These identifications are proved in
Proposition~\ref{prop:hilbert-brownian-cones-ou-spirals}.
The finiteness condition $V_1(K)<\infty$ required in
Theorem~\ref{thm:intro-small-opening-intrinsic-volumes}
holds for both Brownian bodies, as verified in the proof of
Corollary~\ref{cor:hilbert-brownian-euclidean-intrinsic-volumes}.
Applying that theorem to the explicit conic intrinsic-volume
formulas of Theorem~\ref{thm:intro-ou-spiral-intrinsic-volumes}
therefore recovers the following identities, originally obtained
by Gao--Vitale~\cite{GaoVitale2001} for the Brownian-motion body
and Gao~\cite{Gao2003} for the Brownian-bridge body.

\begin{corollary}[Intrinsic volumes of the Brownian bodies]
\label{cor:intro-hilbert-brownian-intrinsic-volumes}
For every integer $j\ge0$,
\begin{equation}
\label{eq:intro-hilbert-brownian-intrinsic-volumes}
V_j(\bK_\infty^B)=\frac{\kappa_j}{j!},
\qquad
V_j(\bK_\infty^A)=\frac{\kappa_{j+1}}{2j!}.
\end{equation}
\end{corollary}

This deduction is given in
Corollary~\ref{cor:hilbert-brownian-euclidean-intrinsic-volumes}.
Thus the ordinary intrinsic volumes of the Brownian bodies
arise as limits of the appropriately rescaled conic intrinsic
volumes of the Ornstein--Uhlenbeck spiral cones.

Finally, Tsirelson's Gaussian-spectrum theorem~\cite{Tsirelson1986}
relates the ordinary intrinsic volumes of the deterministic
Hilbert-space bodies $\bK_\infty^B$ and $\bK_\infty^A$ to the expected volumes of the random
Brownian convex hulls $\mathcal K_d^B$ and $\mathcal K_d^A$;
see~\cite{KabluchkoZaporozhets2016}. The $d$-dimensional Gaussian
spectra of $\bK_\infty^B$ and $\bK_\infty^A$ have the laws of the Brownian-motion and Brownian-bridge
convex hulls $\mathcal K_d^B$ and $\mathcal K_d^A$, respectively; see
\cite[Theorem~1.9 and Section~1.5]{KabluchkoZaporozhets2016}.
In our normalization, Tsirelson's formula gives
\begin{equation}
\label{eq:intro-tsirelson-brownian-hull-volumes}
\Ee\,\operatorname{vol}_d(\mathcal K_d^\bullet)
=\frac{d!\kappa_d}{(2\pi)^{d/2}}V_d(\bK_\infty^\bullet),
\qquad d\ge1,\quad\bullet\in\{B,A\}.
\end{equation}
Substituting \eqref{eq:intro-hilbert-brownian-intrinsic-volumes}
gives
\[
\Ee\,\operatorname{vol}_d(\mathcal K_d^B)
=\frac{\kappa_d^2}{(2\pi)^{d/2}},
\qquad
\Ee\,\operatorname{vol}_d(\mathcal K_d^A)
=\frac{\kappa_d\kappa_{d+1}}{2(2\pi)^{d/2}}
=\frac{(2\pi)^{d/2}}{(d+1)!}.
\]
Those same formulas were obtained in
Corollary~\ref{cor:intro-brownian-gaussian-large-scale-and-volumes} directly
from the large-$\sigma$ asymptotics of the Gaussian-point absorption probabilities.

\subsection{Organization of the paper}
\label{subsec:organization}
The remainder of the paper is devoted mainly to the proofs of the
results stated in the introduction.
Section~\ref{sec:preliminaries} recalls the necessary
facts from convex geometry, including conic intrinsic volumes,
Grassmann angles, and the conic Crofton formula.
Section~\ref{sec:finite-dim-cones} introduces the finite-dimensional cones
associated with random walks and bridges and computes
their conic intrinsic volumes.
Section~\ref{sec:finite-BA-scaling-limits} establishes the
scaling limits of these conic intrinsic volumes.
Section~\ref{sec:brownian-ou-absorption} derives the absorption
formulas for Brownian motion, Brownian bridges, and
Ornstein--Uhlenbeck processes, together with their
asymptotic consequences and the expected conic intrinsic
volumes of random OU cones.
Finally, Section~\ref{sec:ornstein-uhlenbeck-spirals}
studies the deterministic OU spiral cones in Hilbert space,
computes their conic intrinsic volumes and generating functions,
and establishes the connection with ordinary intrinsic
volumes through the small-opening limit.

\section{Preliminaries}
\label{sec:preliminaries}

In this section we recall the facts from convex geometry and
probability used below. The reader may skip it and return to it whenever necessary.
Throughout the paper, empty sums and products are understood to be $0$ and $1$, respectively.

\subsection{Convex and polyhedral geometry}
\label{subsec:convex-cones-polyhedral-geometry}
For background on convex polytopes and polyhedra, we refer to
Br\o ndsted~\cite{Brondsted1983} and McMullen~\cite{McMullen2026}.
For convex cones and their angles, see Schneider~\cite[Chapters~1--2]{Schneider2022}. For polyhedral cones and their angles, see Amelunxen--Lotz~\cite[Section~2]{AmelunxenLotz2017}.

Let $E$ be a real Euclidean space of dimension $d$, with inner product
$\langle\cdot,\cdot\rangle$ and norm $\|\cdot\|$.
For a nonempty set $M\subseteq E$, write $\conv M$, $\pos M$, $\aff M$ and
$\lin M$ for its convex hull, positive hull, affine hull and linear span,
respectively. Thus $\conv M$ consists of finite convex combinations
of points of $M$, whereas $\pos M$ consists of finite linear
combinations with nonnegative coefficients. In particular,
$0\in\pos M$. The affine hull $\aff M$ is the minimal affine subspace containing $M$.
We write $\overline M$ for closure,  $\convclos M$ for the closed
convex hull, and $\posclos M$ for the closed positive hull. For a nonempty convex set $K$, its dimension is the
dimension of its affine hull, and $\relint K$ is its interior relative
to that affine hull. We use $\intr K$ and $\bd K$ for interior and
boundary in the specified ambient space. For a linear subspace
$L\subseteq E$, its orthogonal complement in $E$ is denoted by $L^\perp$.

A \emph{polyhedron} is an intersection of finitely many closed affine
half-spaces. A \emph{polytope} is the convex hull of a nonempty finite
set, equivalently a nonempty bounded polyhedron. A \emph{simplex} is
the convex hull of affinely independent points.
A \emph{convex cone} is a nonempty set closed under addition and
multiplication by nonnegative scalars; its dimension is
$\dim(\lin C)$. A cone $C$ is \emph{pointed} if
$C\cap(-C)=\{0\}$, and \emph{full-dimensional} in $E$ if $\lin C=E$.
A \emph{polyhedral cone} is an intersection of finitely many closed
half-spaces whose bounding hyperplanes pass through the origin, equivalently a positive hull of  finitely
many vectors. In particular, polyhedral cones are closed.
A cone is \emph{simplicial} if it is generated by linearly independent
vectors; such a cone is polyhedral and pointed.

For a closed convex cone $C\subseteq E$, its \emph{polar cone} is
\begin{equation}
\label{eq:prelim-polar-cone}
C^\circ:=\{y\in E:\langle y,x\rangle\le0\text{ for every }x\in C\}.
\end{equation}
It is a closed convex cone and satisfies $C^{\circ\circ}=C$.
For a linear subspace $L$, this definition gives $L^\circ=L^\perp$.
Polars, orthogonal complements, and normal cones are always taken in
the ambient space specified for the object under consideration;
in particular, this space need not be the largest Euclidean space
in which the object is embedded.

By a \emph{face} of a nonempty polyhedron $P$ we mean a nonempty face, including
$P$ itself: it has the form
$F=\{x\in P:\langle y,x\rangle=a\}$, where
$\langle y,x\rangle\le a$ for every $x\in P$.
The choice $y=0$, $a=0$ includes $F=P$.
For a polyhedral cone $C$, faces have the form
\begin{equation}
\label{eq:prelim-polyhedral-cone-faces}
F=C\cap\{x\in E:\langle y,x\rangle=0\},
\qquad y\in C^\circ.
\end{equation}
We denote the set of its $k$-dimensional faces by $\mathcal F_k(C)$.
The relative interiors of all faces form a disjoint partition of $C$.
If $C$ is pointed, $\{0\}$ is a face; every face of a simplicial cone
is generated by a subset of its linearly independent generators.

For a face $F$ of a polyhedron $P$ and $x_0\in\relint F$, define
its \emph{normal cone} by
\begin{equation}
\label{eq:prelim-normal-cone-definition}
N(F,P):=\{y\in E:\langle y,x-x_0\rangle\le0
                    \text{ for every }x\in P\}.
\end{equation}
This definition is independent of the choice of $x_0\in\relint F$.
For a polyhedral cone $C$, polar face duality gives
\begin{equation}
\label{eq:prelim-normal-cone-polar-intersection}
N(F,C)=C^\circ\cap(\lin F)^\perp.
\end{equation}
More precisely, $F\mapsto N(F,C)$ is an inclusion-reversing bijection
between the nonempty faces of $C$ and those of $C^\circ$, and
\begin{equation}
\label{eq:prelim-normal-cone-dimension}
\dim N(F,C)=d-\dim F,
\qquad
\lin N(F,C)=(\lin F)^\perp;
\end{equation}
see \cite[Section~2.1]{AmelunxenLotz2017}.
These statements are valid also for cones that are not full-dimensional,
with all dimensions and orthogonal complements interpreted in $E$.

A \emph{standard Gaussian vector} in a linear subspace $L$ of a Euclidean space has
independent $N(0,1)$ coordinates in any orthonormal basis of $L$;
for $L=\{0\}$ it is the zero vector. The \emph{solid angle} of a
closed convex cone $C$ is
\begin{equation}
\label{eq:prelim-solid-angle}
\alpha(C):=\Pp(G_{\lin C}\in C),
\end{equation}
where $G_{\lin C}$ is standard Gaussian in $\lin C$.
Thus angles are measured relative to the linear span, and
$\alpha(\{0\})=1$; see also
\cite[Section~2.3.1]{AmelunxenLotz2017}.
For closed convex cones with mutually orthogonal linear spans, write
$C_1\oplus\cdots\oplus C_r$ for their sum. Independence of the
orthogonal Gaussian components gives
\begin{equation}
\label{eq:prelim-orthogonal-angle-product}
\alpha(C_1\oplus\cdots\oplus C_r)
=\prod_{i=1}^r\alpha(C_i).
\end{equation}
For a face $F$ of a polyhedral cone $C$, the quantities $\alpha(F)$
and $\alpha(N(F,C))$ will be called its \emph{face angle} and
\emph{external angle}, respectively. They are measured in $\lin F$
and $(\lin F)^\perp$, respectively, by
\eqref{eq:prelim-normal-cone-dimension}.

\subsection{Conic intrinsic volumes and Grassmann angles}
\label{subsec:conic-intrinsic-volumes-grassmann-angles}

Throughout this subsection, $E$ is a Euclidean space of dimension
$d\ge1$, and cones are nonempty, closed, and convex.
For background on conic intrinsic volumes and Grassmann angles, see
Schneider~\cite[Chapters~2--4]{Schneider2022} and
Amelunxen--Lotz--McCoy--Tropp~\cite[Section~5]{AmelunxenLotzMcCoyTropp2014}.
For the polyhedral theory, see
Amelunxen--Lotz~\cite{AmelunxenLotz2017}.
A treatment in the language of spherical geometry is given in
Schneider--Weil~\cite[Section~6.5]{SchneiderWeil2008}.
We first recall the definition of conic intrinsic volumes for
polyhedral cones and then their extension to arbitrary closed convex
cones by polyhedral approximation; see
McCoy--Tropp~\cite[Section~2]{McCoyTropp2014}.

For a cone $C\subseteq E$, let $\Pi_C(x)$ be the unique point of $C$
nearest to $x\in E$, and let $G$ be standard Gaussian in $E$.
If $C$ is polyhedral, its $k$th \emph{conic intrinsic volume} is
\begin{equation}
\label{eq:prelim-polyhedral-conic-intrinsic-volumes}
\upsilon_k(C)
:=\Pp\left(\Pi_C(G)\in
       \bigcup_{F\in\mathcal F_k(C)}\relint F\right),
\qquad 0\le k\le d;
\end{equation}
see \cite[Definition~2.1]{McCoyTropp2014} and
\cite[Section~2.2]{AmelunxenLotz2017}.
Equivalently, the intrinsic volumes are sums of products of face
and external angles:
\begin{equation}
\label{eq:prelim-conic-face-angle-formula}
\upsilon_k(C)
=\sum_{F\in\mathcal F_k(C)}
  \alpha(F)\alpha\bigl(N(F,C)\bigr),
\qquad 0\le k\le d.
\end{equation}
Here the angles are measured in the respective linear spans, as in
\eqref{eq:prelim-solid-angle}; see
\cite[Sections~2.2--2.3]{AmelunxenLotz2017}.
The facial descriptions
\eqref{eq:prelim-polyhedral-conic-intrinsic-volumes} and
\eqref{eq:prelim-conic-face-angle-formula} are asserted only for
polyhedral cones.

For general cones, we use continuous extension. Let $B_E$ be the
closed unit ball of $E$. We write $C_n\to C$ if
$C_n\cap B_E$ converges to $C\cap B_E$ in Hausdorff distance.
This convergence is equivalent to convergence in the conic
(angular) Hausdorff metric used in \cite{McCoyTropp2014}.
Polyhedral cones are dense in this topology.
For any sequence of polyhedral cones $C_n\to C$, define
\begin{equation}
\label{eq:prelim-conic-intrinsic-volumes-extension}
\upsilon_k(C):=\lim_{n\to\infty}\upsilon_k(C_n),
\qquad 0\le k\le d.
\end{equation}
The limits exist, are independent of the approximating sequence,
and define continuous functions of $C$; see
\cite[Definition~2.2, Section~7.3, and Proposition~8.2]{McCoyTropp2014}.
In particular, continuity holds for convergent sequences of arbitrary
closed convex cones, not only polyhedral ones.
We set $\upsilon_k(C)=0$ for indices outside $\{0,\ldots,d\}$.

The intrinsic volumes are nonnegative and satisfy
\begin{equation}
\label{eq:prelim-conic-intrinsic-volume-normalization}
\sum_{k=0}^d\upsilon_k(C)=1.
\end{equation}
They are invariant under linear isometries, including isometric
embeddings into larger Euclidean spaces. Thus they may equivalently
be computed in $\lin C$, and $\upsilon_k(C)=0$ for $k>\dim C$.
For a linear subspace $L\subseteq E$ of dimension $r$,
\begin{equation}
\label{eq:prelim-conic-intrinsic-volumes-subspaces}
\upsilon_k(L)=\ind_{\{k=r\}};
\end{equation}
in particular, $\upsilon_0(\{0\})=1$ and $\upsilon_d(E)=1$.
Polarity, taken in $E$, reverses the indices:
$\upsilon_k(C^\circ)=\upsilon_{d-k}(C)$ for $0\le k\le d$.
Let $G$ be a standard Gaussian vector in the ambient Euclidean space
$E$.  Then the two endpoint intrinsic volumes admit the representations
\begin{equation}
\label{eq:prelim-conic-intrinsic-volume-endpoints}
\upsilon_0(C)=\Pp(G\in C^\circ),
\qquad
\upsilon_d(C)=\Pp(G\in C).
\end{equation}
Consequently, $\upsilon_d(C)=\alpha(C)$ when $C$ is full-dimensional
in $E$; more generally, $\upsilon_{\dim C}(C)=\alpha(C)$.
These properties hold for all closed convex cones; see
\cite[Section~2]{McCoyTropp2014}.
Our indexing is the conic one: the $k$th conic intrinsic volume,
for $k\ge1$, corresponds to the $(k-1)$st spherical intrinsic
volume of $C\cap \mathbb S^{d-1}$ in the convention of
\cite[Section~6.5]{SchneiderWeil2008}; compare
\cite[Section~1, footnote~1]{McCoyTropp2014}.

We next introduce \emph{Grassmann angles}; for their historical
origin, see Gr\"unbaum~\cite{Grunbaum1968}.
We use the following nontrivial-intersection probability convention.
For $0\le j\le d$, let $L_{d-j}$ be a uniformly distributed
$(d-j)$-dimensional linear subspace of $E$, where uniformity refers
to the orthogonally invariant probability measure on the
Grassmannian. Define
\begin{equation}
\label{eq:prelim-finite-dimensional-grassmann-angles}
\gamma_j(C):=\Pp\bigl(C\cap L_{d-j}\ne\{0\}\bigr),
\qquad 0\le j\le d,
\end{equation}
and set $\gamma_j(C)=0$ for $j>d$.
These quantities are invariant under linear isometries and
isometric embeddings, with the codimension $j$ kept fixed.
They are monotone under inclusion of cones, and
$\gamma_j(C)\ge\gamma_{j+1}(C)$; the latter follows by choosing
uniformly rotated nested subspaces.
Moreover, $\gamma_j(C)=0$ for $j\ge\dim C$ and
$\gamma_0(C)=1$ for $C\ne\{0\}$, whereas all Grassmann angles of
$\{0\}$ vanish.
For a fixed $r$-dimensional linear subspace $L$,
\begin{equation}
\label{eq:prelim-grassmann-angles-subspaces}
\gamma_j(L)=\ind_{\{j<r\}}.
\end{equation}

The following conic Crofton formula holds for arbitrary closed
convex cones under the stated non-subspace hypothesis; see
\cite[Fact~5.10, equation~(5.10)]{AmelunxenLotzMcCoyTropp2014}
and \cite[Chapter~4]{Schneider2022}.
For a proof in the polyhedral setting, see
\cite[Corollary~5.2]{AmelunxenLotz2017}; that statement and its
proof concern polyhedral cones.

\begin{theorem}[Conic Crofton formula]
\label{thm:prelim-conic-crofton}
Let $C\subseteq E$ be a closed convex cone which is not a linear
subspace. For every $0\le j\le d$,
\begin{equation}
\label{eq:finite-dimensional-conic-crofton}
\gamma_j(C)
=\Pp\bigl(C\cap L_{d-j}\ne\{0\}\bigr)
=2\sum_{\ell=0}^\infty\upsilon_{j+1+2\ell}(C).
\end{equation}
The sum is finite under the convention that intrinsic volumes
above dimension $d$ vanish.
\end{theorem}

Taking $j=0$ and using
\eqref{eq:prelim-conic-intrinsic-volume-normalization} gives the
conic Gauss--Bonnet identity
\begin{equation}
\label{eq:prelim-conic-gauss-bonnet}
\sum_{\ell=0}^\infty\upsilon_{2\ell}(C)
=\sum_{\ell=0}^\infty\upsilon_{2\ell+1}(C)=\frac12,
\end{equation}
again when $C$ is not a linear subspace.
Subtracting the sums in \eqref{eq:finite-dimensional-conic-crofton}
at indices $k-1$ and $k+1$ gives the positive-index identities below;
the zeroth follows from normalization:
\begin{equation}
\label{eq:prelim-intrinsic-volumes-from-grassmann-angles}
\upsilon_0(C)=1-\frac{\gamma_0(C)+\gamma_1(C)}2,
\qquad
\upsilon_k(C)=\frac{\gamma_{k-1}(C)-\gamma_{k+1}(C)}2,
\quad 1\le k\le d.
\end{equation}
These inversion formulas are used only for cones which are not
linear subspaces; for linear subspaces we use
\eqref{eq:prelim-conic-intrinsic-volumes-subspaces} and
\eqref{eq:prelim-grassmann-angles-subspaces} instead.

\begin{remark}[The full-space case]
\label{rem:prelim-conic-full-space-correction}
A full-dimensional closed convex cone $C\subseteq E$ is either
not a linear subspace or is $E$ itself. To cover both possibilities,
the top intrinsic volume satisfies, for every $d\ge1$,
\begin{equation}
\label{eq:prelim-top-intrinsic-volume-full-space}
\upsilon_d(C)
=\frac{\gamma_{d-1}(C)+\ind_{\{C=E\}}}{2}.
\end{equation}
For $d\ge2$, the next coefficient is
\begin{equation}
\label{eq:prelim-next-to-top-intrinsic-volume-full-space}
\upsilon_{d-1}(C)
=\frac{\gamma_{d-2}(C)-\ind_{\{C=E\}}}{2}.
\end{equation}
For such $d$, the zeroth-coefficient formula in
\eqref{eq:prelim-intrinsic-volumes-from-grassmann-angles}, and the
positive-index formulas for $1\le k\le d-2$, hold also when $C=E$.
Indeed, all assertions follow from
\eqref{eq:prelim-intrinsic-volumes-from-grassmann-angles} when
$C\ne E$, and from the subspace formulas when $C=E$.
In dimension one, the full-dimensional cones are rays and the
whole line, with intrinsic-volume sequences $(1/2,1/2)$ and $(0,1)$,
respectively. The distinction is needed when applying Crofton
identities to random cones that can equal the whole ambient space.
\end{remark}

\section{Finite-dimensional cones}\label{sec:finite-dim-cones}

In this section, we introduce two families of polyhedral cones, $\bW_n^B(\rho)$ and $\bW_n^A(\rho)$, associated with random walks and random bridges, and compute their conic intrinsic volumes. The computation is based on explicit descriptions of their faces, face angles, polar and normal cones, and external angles. The probabilistic significance of these conic intrinsic volumes will become clear later, when Gaussian-point absorption probabilities for random-walk and random-bridge convex hulls are expressed in terms of them; see Subsection~\ref{subsec:finite-gaussian-hulls-absorption-formula}.

\subsection{Random-walk and random-bridge bodies and their cones}
\label{subsec:type-B-bodies-and-cones}

Let $n\ge1$, and let $e_0,e_1,\ldots,e_n$ be the standard
orthonormal basis of $\R e_0\oplus\R^n$. We identify $\R^n$
with the linear subspace spanned by $e_1,\ldots,e_n$. Set $s_0:=0$
and $s_i:=e_1+\cdots+e_i$ for $i=1,\ldots,n$.

The
\emph{random-walk body} and the \emph{random-bridge body} are,
respectively,
\begin{align}
\bK_n^B
&:=\conv\{s_i:i=0,\ldots,n\}\subseteq\R^n,
\label{eq:type-B-body-definition}
\\
\bK_n^A
&:=\conv\left\{s_i-\frac{i}{n}s_n:i=0,\ldots,n-1\right\}
\subseteq s_n^\perp\subseteq \R^n.
\label{eq:type-A-body-definition}
\end{align}
Here the orthogonal complement $s_n^\perp$ is taken in $\R^n$. These bodies, in a different notation, already appeared
in the work of Gao--Vitale~\cite{GaoVitale2001} and Gao~\cite{Gao2003},
who computed their Euclidean intrinsic volumes.
The two bodies have the coordinate descriptions
\begin{align}
\bK_n^B
&=\left\{\sum_{j=1}^n t_je_j:
1\ge t_1\ge\cdots\ge t_n\ge0\right\},
\label{eq:type-B-body-coordinates}
\\
\bK_n^A
&=\left\{\sum_{j=1}^n t_je_j:
\sum_{j=1}^n t_j=0,\quad
 t_1\ge\cdots\ge t_n,\quad t_1-t_n\le1\right\}.
\label{eq:type-A-body-coordinates}
\end{align}

The vectors $s_1,\ldots,s_n$ form a basis of $\R^n$, so the points $0,s_1,\ldots, s_n$ are affinely independent and
hence $\bK_n^B$ is an $n$-dimensional simplex. Likewise,
$s_i-(i/n)s_n$, $i=1,\ldots,n-1$, form a basis of
$s_n^\perp$.  Hence $\bK_n^A$ is an
$(n-1)$-dimensional simplex in $s_n^\perp$. This includes $n=1$, when
$\bK_1^A=\{0\}$.

\medskip
For $\rho>0$, define the \emph{random-walk cone} and the \emph{random-bridge cone} by
\begin{align}
\bW_n^B(\rho)
&:=\pos\bigl(\rho e_0+\bK_n^B\bigr)
=\pos\{\rho e_0+s_i:i=0,\ldots,n\}\subseteq \R e_0 \oplus \R^{n},
\label{eq:type-B-cone-definition}
\\
\bW_n^A(\rho)
&:=\pos\bigl(\rho e_0+\bK_n^A\bigr)
=\pos\left\{\rho e_0+s_i-\frac{i}{n}s_n:
 i=0,\ldots,n-1\right\}\subseteq \R e_0 \oplus  s_n^\perp.
\label{eq:type-A-cone-definition}
\end{align}
In particular, $\rho e_0$ is a generator in both cases. We regard the two cones as
cones in $\R e_0\oplus\R^n$ and $\R e_0\oplus s_n^\perp$,
respectively; their polars and normal cones are taken in these
respective ambient spaces.
Their coordinate descriptions are
\begin{align}
\bW_n^B(\rho)
&=\left\{\rho t_0e_0+\sum_{j=1}^n t_je_j:
 t_0\ge t_1\ge\cdots\ge t_n\ge0\right\},
\label{eq:type-B-cone-coordinates}
\\
\bW_n^A(\rho)
&=\left\{\rho t_0e_0+\sum_{j=1}^n t_je_j:
\sum_{j=1}^n t_j=0,\quad
 t_1\ge\cdots\ge t_n,\quad t_0\ge t_1-t_n\right\}.
\label{eq:type-A-cone-coordinates}
\end{align}

In either case, affine independence of the vertices of the
body $\bK_n^{\bullet}$ with $\bullet \in \{B,A\}$ implies linear independence of the generators of the corresponding cone $\bW_n^{\bullet}(\rho)$:
the $e_0$ coordinate of a linear relation forces the sum of its
coefficients to vanish, and affine independence then forces
each coefficient to vanish. Thus $\bW_n^B(\rho)$ and $\bW_n^A(\rho)$ are simplicial cones of
dimensions $n+1$ and $n$, respectively, and hence are pointed.
Indeed, the linear map from the corresponding Euclidean
space to the ambient space of the cone that sends the standard
basis to its generators is a linear isomorphism, and it maps the
nonnegative orthant onto the cone.

Consequently, the faces of $\bW_n^B(\rho)$ and $\bW_n^A(\rho)$ are precisely
\begin{align}
F_I^B
&:=\pos\{\rho e_0+s_i:i\in I\},
\qquad I\subseteq\{0,\ldots,n\},
\label{eq:type-B-faces}
\\
F_I^A
&:=\pos\left\{\rho e_0+s_i-\frac{i}{n}s_n:i\in I\right\},
\qquad I\subseteq\{0,\ldots,n-1\},
\label{eq:type-A-faces}
\end{align}
with $F_\varnothing^B=F_\varnothing^A:=\{0\}$ and
$\dim F_I^\bullet=|I|$ for each $\bullet\in\{B,A\}$ and
its admissible index sets. These assertions follow from the corresponding
face description of the nonnegative orthant.

Our aim is to determine the conic intrinsic volumes of both
cones by evaluating the face and external angles in the
face-angle formula
\begin{equation}
\label{eq:finite-BA-face-angle-sum}
\upsilon_k\bigl(\bW_n^\bullet(\rho)\bigr)
=
\sum_{|I|=k}
\alpha(F_I^\bullet)
\alpha\bigl(N(F_I^\bullet,\bW_n^\bullet(\rho))\bigr),
\qquad \bullet\in\{B,A\}.
\end{equation}
The sum is over the admissible index sets in
\eqref{eq:type-B-faces} or \eqref{eq:type-A-faces}, respectively.

\subsection{Face angles}
\label{subsec:type-B-internal-angles}
In this subsection we show that the face angle $\alpha(F_I^\bullet)$
can be expressed as the probability that the components of a suitable
Gaussian random vector are nonincreasing. Fix $n\ge1$ and $\rho>0$.

\begin{lemma}[Face angles of $\bW_n^B(\rho)$]
\label{lem:type-B-internal-angle}
Let $I=\{i_1<\cdots<i_k\}\subseteq\{0,\ldots,n\}$ with
$k\ge1$. Then
\begin{equation}
\label{eq:type-B-internal-angle}
\alpha(F_I^B)
=
\Ord_k\left(\rho^2+i_1,
i_2-i_1,\ldots,i_k-i_{k-1}\right),
\end{equation}
where, for $a_1,\ldots,a_k>0$,
\begin{equation}
\label{eq:type-B-Gaussian-ordering-probability}
\Ord_k(a_1,\ldots,a_k)
:=\Pp\left(
\frac{Z_1}{\sqrt{a_1}}\ge\cdots\ge
\frac{Z_k}{\sqrt{a_k}}\ge0\right),
\end{equation}
and $Z_1,\ldots,Z_k$ are independent standard normal random
variables. In particular, $\Ord_1(a_1)=1/2$.
\end{lemma}

\begin{proof}
\emph{Step 1: Construct a standard Gaussian vector in $\lin F_I^B$.}
The vectors $\rho e_0+s_{i_1}$ and
$s_{i_r}-s_{i_{r-1}}$, $r=2,\ldots,k$, have disjoint
coordinate supports, so they are mutually orthogonal.
Their squared norms are $\rho^2+i_1$ and $i_r-i_{r-1}$,
respectively. They span $\lin F_I^B$, as is also seen from
\begin{equation}
\label{eq:type-B-face-block-expansion}
\sum_{j=1}^k\lambda_j(\rho e_0+s_{i_j})
=
\left(\sum_{j=1}^k\lambda_j\right)(\rho e_0+s_{i_1})
+
\sum_{r=2}^k
\left(\sum_{j=r}^k\lambda_j\right)
(s_{i_r}-s_{i_{r-1}}).
\end{equation}
Consequently, a standard Gaussian vector in $\lin F_I^B$
can be written as
\begin{equation}
\label{eq:type-B-face-Gaussian}
G=
\frac{Z_1}{\sqrt{\rho^2+i_1}}(\rho e_0+s_{i_1})
+
\sum_{r=2}^k\frac{Z_r}{\sqrt{i_r-i_{r-1}}}
(s_{i_r}-s_{i_{r-1}}).
\end{equation}

\medskip
\emph{Step 2: Translate cone membership into an ordering of coordinates.}
The left-hand side of~\eqref{eq:type-B-face-block-expansion} belongs to $F_I^B$ if and only if the coefficients $\lambda_1,\ldots, \lambda_k$ are nonnegative. The coefficients $\lambda_1,\ldots, \lambda_k$  are nonnegative if
and only if their successive tail sums on the right-hand side of~\eqref{eq:type-B-face-block-expansion} are nonincreasing
and nonnegative. Therefore the Gaussian vector $G$ in
\eqref{eq:type-B-face-Gaussian} belongs to $F_I^B$ precisely
when
$$
\frac{Z_1}{\sqrt{\rho^2+i_1}} \geq \frac{Z_2}{\sqrt{i_2-i_{1}}} \geq \ldots \geq \frac{Z_k}{\sqrt{i_k-i_{k-1}}}\geq 0.
$$
Taking probabilities
proves the claim since  $\alpha(F_I^B) = \Pp(G\in F_I^B)$.
\end{proof}

A similar description for the angle of $F_I^A$ is more involved. This time,  the two endpoints of the Gaussian vector form a correlated Gaussian pair,
independent of the coordinates between them.  If
$I=\{i_1<\cdots<i_k\}\subseteq\{0,\ldots,n-1\}$ is nonempty,
put
\begin{equation}
\label{eq:type-A-rooted-face-block-lengths}
a:=n-i_k,\qquad b:=i_1,\qquad h:=a+b,\qquad
d_j:=i_{j+1}-i_j\quad(1\le j<k).
\end{equation}
Thus $a\ge1$, $b\ge0$, $d_j\ge1$, and
$h+d_1+\cdots+d_{k-1}=n$.

\begin{lemma}[Face angles of $\bW_n^A(\rho)$]
\label{lem:type-A-internal-angle}
Let $I=\{i_1<\cdots<i_k\}\subseteq\{0,\ldots,n-1\}$ with
$k\ge1$, and let $a,b,d_1,\ldots,d_{k-1}$ be as in
\eqref{eq:type-A-rooted-face-block-lengths}. Let $(X,Z)$ be a
centered Gaussian random vector with density
\begin{equation}
\label{eq:type-A-endpoint-Gaussian-density}
p_{\rho;a,b}(x,z)
:=\frac{\sqrt{ab+\rho^2(a+b)}}{2\pi}
\eee^{-\frac12\left[bx^2+az^2+\rho^2(x-z)^2\right]},
\qquad (x,z)\in\R^2.
\end{equation}
Let $Z_1,\ldots,Z_{k-1}$ be independent centered Gaussian
random variables, independent of $(X,Z)$, with
$\operatorname{Var}(Z_j)=1/d_j$. Then
\begin{equation}
\label{eq:type-A-internal-angle}
\alpha(F_I^A)
=\Pp\bigl(X\ge Z_1\ge\cdots\ge Z_{k-1}\ge Z\bigr).
\end{equation}
Equivalently,
\begin{equation}
\label{eq:type-A-internal-angle-integral}
\alpha(F_I^A)
=
\int_{z<x}p_{\rho;a,b}(x,z)
\Pp\bigl(x\ge Z_1\ge\cdots\ge Z_{k-1}\ge z\bigr)
\,\dint x\,\dint z.
\end{equation}
For $k=1$, the event in \eqref{eq:type-A-internal-angle}
is $\{X\ge Z\}$, and the inner probability in
\eqref{eq:type-A-internal-angle-integral} is understood to
be $1$ on $z<x$. In particular, every one-dimensional face
has angle $1/2$.
\end{lemma}

\begin{proof}
\emph{Step 1: Construct a basis in $\lin F_I^A\oplus \R s_n$.}
Introduce the vectors
\begin{align*}
v_0:=\rho e_0+s_{i_1},
\quad
v_j:=s_{i_{j+1}}-s_{i_j}
\quad
(1\le j<k),
\quad
v_k:=-\rho e_0+s_n-s_{i_k}.
\end{align*}
The vectors $v_1,\ldots,v_{k-1}$ are mutually orthogonal,
have squared norms $d_1,\ldots,d_{k-1}$, and are orthogonal
to both $v_0$ and $v_k$. The Gram matrix of the latter pair is
\begin{equation}
\label{eq:type-A-endpoint-block-Gram-matrix}
\begin{pmatrix}
\rho^2+b&-\rho^2\\
-\rho^2&\rho^2+a
\end{pmatrix},
\end{equation}
whose determinant is $ab+\rho^2(a+b)>0$. Hence all
$k+1$  vectors are linearly independent.

We need an orthogonal projection mapping $\bW_n^B(\rho)$  to $\bW_n^A(\rho)$. Let $P$ be orthogonal projection from $\R e_0\oplus\R^n$ onto
$\R e_0\oplus s_n^\perp$;  explicitly,
\begin{equation}
\label{eq:type-A-centering-projection}
P(ae_0+x)
=ae_0+x-\frac{\langle x,s_n\rangle}{n}s_n,
\qquad a\in\R,\quad x\in\R^n.
\end{equation}
Then $Ps_i=s_i-(i/n)s_n$ for $i=0,\ldots, n$. Hence $P\bK_n^B=\bK_n^A$ and $P\bW_n^B(\rho)=\bW_n^A(\rho)$.

Let $V_I:=\lin\{v_0,\ldots,v_k\}$. The sum of $v_0,\ldots, v_k$ is $s_n$,
and the projection $P$  satisfies
\[
P\left(v_0+\sum_{\ell=1}^{j-1}v_\ell\right)
=\rho e_0+s_{i_j}-\frac{i_j}{n}s_n,
\qquad j=1,\ldots,k.
\]
The kernel of $P$ on $V_I$ is $\R s_n$, so $PV_I$ has
dimension $k$. The displayed vectors span $\lin F_I^A$,
which also has dimension $k$. It follows that
\begin{equation}
\label{eq:type-A-lifted-face-span}
V_I=\lin F_I^A\oplus\R s_n,
\end{equation}
where the sum is orthogonal.

\medskip
\emph{Step 2: Identify the standard Gaussian law in $\lin F_I^A$.}
For $x,z\in\R$, the Gram matrix~\eqref{eq:type-A-endpoint-block-Gram-matrix} yields
\[
\|xv_0+zv_k\|^2=bx^2+az^2+\rho^2(x-z)^2.
\]
The change from the coefficients $(x,z)$ to Euclidean area
measure in $\lin\{v_0,v_k\}$ has the constant Jacobian
$\sqrt{ab+\rho^2(a+b)}$, the square root of the Gram
determinant in \eqref{eq:type-A-endpoint-block-Gram-matrix}.
Thus the coefficients of a standard Gaussian vector in this
two-dimensional space have precisely the density
\eqref{eq:type-A-endpoint-Gaussian-density}.
In particular, this formula is a nondegenerate
Gaussian probability density, also when $b=0$.

Since $v_1,\ldots,v_{k-1}$ are mutually orthogonal and also orthogonal to $\lin \{v_0, v_k\}$,  the random vector
\begin{equation}
\label{eq:type-A-lifted-face-Gaussian}
\widetilde G:=Xv_0+\sum_{j=1}^{k-1}Z_jv_j+Zv_k
\end{equation}
is standard Gaussian in $V_I$. By \eqref{eq:type-A-lifted-face-span},
$G:=P\widetilde G$ is therefore standard Gaussian in
$\lin F_I^A$.

\medskip
\emph{Step 3: Express $G$ in the face generators.}
Set $Z_0:=X$ and $Z_k:=Z$. Telescoping~\eqref{eq:type-A-lifted-face-Gaussian}
and using $Ps_n=0$ gives
\begin{equation}
\label{eq:type-A-face-Gaussian-coefficients}
G
=\sum_{j=1}^k(Z_{j-1}-Z_j)
\left(\rho e_0+s_{i_j}-\frac{i_j}{n}s_n\right).
\end{equation}

The generators of $F_I^A$, that is the vectors $\rho e_0+s_{i_j}-\frac{i_j}{n}s_n$, $j=1,\ldots, k$,  are linearly independent. Consequently,
$G\in F_I^A$ if and only if all $Z_{j-1}-Z_j$ are nonnegative,
which is exactly the ordering in
\eqref{eq:type-A-internal-angle}. Taking probabilities proves
that formula. Conditioning on $(X,Z)$ gives
\eqref{eq:type-A-internal-angle-integral}. For $k=1$, $X-Z$ is a nondegenerate
centered Gaussian variable, so the probability is $1/2$.
\end{proof}

\subsection{Polar and normal cones}
\label{subsec:finite-BA-polar-and-normal-cones}

We continue to regard $\bW_n^B(\rho)$ and $\bW_n^A(\rho)$ as
cones in $\R e_0\oplus\R^n$ and $\R e_0\oplus s_n^\perp$,
respectively. All polars and orthogonal complements below are
taken in these respective ambient spaces. We first describe
the polar cones of $\bW_n^\bullet(\rho)$  and then obtain the normal cone at a face by
turning the corresponding polar inequalities into equalities.

\begin{proposition}[Polar cones of $\bW_n^\bullet(\rho)$]
\label{prop:finite-BA-polar-cones}
Let $n\ge1$ and $\rho>0$. The polar cones have the descriptions
\begin{align}
\bigl(\bW_n^B(\rho)\bigr)^\circ
&=
\left\{
\sum_{j=0}^n y_je_j:
\rho y_0+\sum_{\ell=1}^j y_\ell\le0,
\quad j=0,\ldots,n
\right\},
\label{eq:type-B-polar-inequalities}
\\
\bigl(\bW_n^A(\rho)\bigr)^\circ
&=
\left\{
\sum_{j=0}^n y_je_j:
\begin{array}{l}
\sum_{\ell=1}^n y_\ell=0,\\
\rho y_0+\sum_{\ell=1}^j y_\ell\le0,
\quad j=0,\ldots,n-1
\end{array}
\right\}.
\label{eq:type-A-polar-inequalities}
\end{align}
Equivalently,
\begin{align}
\bigl(\bW_n^B(\rho)\bigr)^\circ
&=
\pos\left(
\{-\rho^{-1}e_0+e_1,-e_n\}
\cup\{-e_j+e_{j+1}:1\le j<n\}
\right),
\label{eq:type-B-polar-generators}
\\
\bigl(\bW_n^A(\rho)\bigr)^\circ
&=
\pos\left(
\{-\rho^{-1}e_0+e_1-e_n\}
\cup\{-e_j+e_{j+1}:1\le j<n\}
\right).
\label{eq:type-A-polar-generators}
\end{align}
The displayed generators form bases of the respective ambient
spaces. In particular, both polar cones are simplicial and
pointed, of dimensions $n+1$ and $n$, respectively. Moreover,
\begin{equation}
\label{eq:type-A-polar-as-section}
\bigl(\bW_n^A(\rho)\bigr)^\circ
=
\bigl(\bW_n^B(\rho)\bigr)^\circ
\cap\bigl(\R e_0\oplus s_n^\perp\bigr).
\end{equation}
\end{proposition}

\begin{proof}
\emph{Step 1: Turn the defining generators into polar inequalities.}
Recall that $\bW_n^B(\rho)$ is the positive hull of the vectors
$\rho e_0+s_i$, $i=0,\ldots,n$. Hence
$y=\sum_{j=0}^n y_je_j$ belongs to
$\bigl(\bW_n^B(\rho)\bigr)^\circ$ if and only if
$\langle y,\rho e_0+s_i\rangle\le0$ for every $i=0,\ldots,n$.
Since
$\langle y,\rho e_0+s_i\rangle
=\rho y_0+\sum_{\ell=1}^i y_\ell$,
this proves~\eqref{eq:type-B-polar-inequalities}.

Next recall that $\bW_n^A(\rho)$ is the positive hull of $\rho e_0+s_i-(i/n)s_n$, $i=0,\ldots,n-1$.  Membership of $y$ in the ambient space
$\R e_0\oplus s_n^\perp$ means that
$\sum_{\ell=1}^n y_\ell=0$. Consequently,
for $i=0,\ldots,n-1$ we have
$\langle y,\rho e_0+s_i-(i/n)s_n\rangle
=\rho y_0+\sum_{\ell=1}^i y_\ell$.
This proves~\eqref{eq:type-A-polar-inequalities}.
Moreover, on $\R e_0\oplus s_n^\perp$, the inequality
$\rho y_0+\sum_{\ell=1}^n y_\ell\le0$ coincides with the
inequality $\rho y_0\le0$. Hence
\eqref{eq:type-A-polar-as-section} follows.

\medskip
\emph{Step 2: Identify the generators of the polar cones.}
For type $B$, the vectors
$-\rho^{-1}e_0+e_1$, $-e_j+e_{j+1}$, $1\le j<n$, and $-e_n$
have scalar product $-1$ with, respectively,
$\rho e_0+s_0$, $\rho e_0+s_j$, $1\le j<n$, and
$\rho e_0+s_n$, and scalar product zero with all the remaining
vectors $\rho e_0+s_i$, $i=0,\ldots,n$.
Since the latter vectors form a basis of
$\R e_0\oplus\R^n$, the former vectors constitute its negative
dual basis. A vector belongs to
$\bigl(\bW_n^B(\rho)\bigr)^\circ$ if and only if its coordinates
in this negative dual basis are nonnegative. This proves
\eqref{eq:type-B-polar-generators}.

For type $A$, the vector
$-\rho^{-1}e_0+e_1-e_n$ has scalar product $-1$ with
$\rho e_0+s_0=\rho e_0$ and scalar product zero with
$\rho e_0+s_i-(i/n)s_n$, $i=1,\ldots,n-1$.
For every $1\le j<n$, the vector $-e_j+e_{j+1}$ has scalar
product $-1$ with
$\rho e_0+s_j-(j/n)s_n$ and scalar product zero with all the
remaining vectors
$\rho e_0+s_i-(i/n)s_n$, $i=0,\ldots,n-1$.
Thus the vectors displayed in
\eqref{eq:type-A-polar-generators} form the negative dual basis
to the basis
$\rho e_0+s_i-(i/n)s_n$, $i=0,\ldots,n-1$, of
$\R e_0\oplus s_n^\perp$. This proves
\eqref{eq:type-A-polar-generators}.

The asserted simpliciality, pointedness, and dimensions now follow
from the fact that the displayed generators form bases of the
respective ambient spaces. When $n=1$, the type-$A$ polar has the
single generator $-\rho^{-1}e_0$, in agreement with the same
argument.
\end{proof}

Next we characterize the normal cones of $\bW_n^\bullet(\rho)$.
For a face $F$ of a polyhedral cone $C$ in a $d$-dimensional
Euclidean space, recall from
\eqref{eq:prelim-normal-cone-polar-intersection} that
\begin{equation}
\label{eq:finite-BA-normal-cone-polar-intersection}
N(F,C)=C^\circ\cap(\lin F)^\perp.
\end{equation}
By polar face duality and \eqref{eq:prelim-normal-cone-dimension},
$N(F,C)$ is a face of $C^\circ$ of dimension $d-\dim F$, and
$\lin N(F,C)=(\lin F)^\perp$.
If $C^\circ$ is simplicial and pointed, then so is $N(F,C)$,
since these properties pass to faces. We apply these facts in
$\R e_0\oplus\R^n$ for type $B$ and in
$\R e_0\oplus s_n^\perp$ for type $A$; all orthogonal
complements below are taken in these respective ambient spaces.

\begin{lemma}[Normal cones of $\bW_n^\bullet(\rho)$]
\label{lem:finite-BA-normal-cones}
Let $n\ge1$ and $\rho>0$. For $I\subseteq\{0,\ldots,n\}$,
\begin{equation}
\label{eq:type-B-normal-cone}
N\bigl(F_I^B,\bW_n^B(\rho)\bigr)
=
\left\{
\sum_{j=0}^n y_je_j:
\begin{array}{ll}
\rho y_0+\sum_{\ell=1}^j y_\ell\le0,
&j=0,\ldots,n,\\
\rho y_0+\sum_{\ell=1}^i y_\ell=0,
&i\in I
\end{array}
\right\},
\end{equation}
and, for $I\subseteq\{0,\ldots,n-1\}$,
\begin{equation}
\label{eq:type-A-normal-cone}
N\bigl(F_I^A,\bW_n^A(\rho)\bigr)
=
\left\{
\sum_{j=0}^n y_je_j:
\begin{array}{ll}
\sum_{\ell=1}^n y_\ell=0,&\\
\rho y_0+\sum_{\ell=1}^j y_\ell\le0,
&j=0,\ldots,n-1,\\[1mm]
\rho y_0+\sum_{\ell=1}^i y_\ell=0,
&i\in I
\end{array}
\right\}.
\end{equation}
These normal cones are pointed and simplicial, their linear
spans are $(\lin F_I^\bullet)^\perp$, and their dimensions are
\begin{align}
\dim N\bigl(F_I^B,\bW_n^B(\rho)\bigr)&=n+1-|I|,
\label{eq:type-B-normal-cone-dimension}
\\
\dim N\bigl(F_I^A,\bW_n^A(\rho)\bigr)&=n-|I|.
\label{eq:type-A-normal-cone-dimension}
\end{align}
For $I=\varnothing$ the normal cone is the corresponding
polar cone, and for $F_I^\bullet=\bW_n^\bullet(\rho)$ it is $\{0\}$.

For a nonempty admissible set $I=\{i_1<\cdots<i_k\}$, the
normal cones admit the following orthogonal decompositions.
For each pair of consecutive indices $(i_j, i_{j+1})$ with $1\le j<k$, define the cone
\begin{equation}
\label{eq:finite-BA-consecutive-coordinate-cone}
C_j:=
\left\{
\sum_{\ell=i_j+1}^{i_{j+1}} y_\ell e_\ell:
\begin{array}{l}
\sum_{\ell=i_j+1}^{i_{j+1}}y_\ell=0,\\
\sum_{\ell=i_j+1}^{r}y_\ell\le0,
\quad i_j<r<i_{j+1}
\end{array}
\right\},
\quad 1\le j<k.
\end{equation}

\emph{Type $B$.} Define the cones
\begin{align}
C_0^B
&:=\left\{
\sum_{\ell=0}^{i_1}y_\ell e_\ell:
\begin{array}{l}
\rho y_0+\sum_{\ell=1}^{i_1}y_\ell=0,\\
\rho y_0+\sum_{\ell=1}^{r}y_\ell\le0,
\quad 0\le r<i_1
\end{array}
\right\},
\label{eq:type-B-initial-coordinate-cone}
\\
C_k^B
&:=\left\{
\sum_{\ell=i_k+1}^{n}y_\ell e_\ell:
\sum_{\ell=i_k+1}^{r}y_\ell\le0,
\quad i_k<r\le n
\right\}.
\label{eq:type-B-terminal-coordinate-cone}
\end{align}
Then
\begin{equation}
\label{eq:type-B-normal-cone-orthogonal-decomposition}
N\bigl(F_I^B,\bW_n^B(\rho)\bigr)
=C_0^B\oplus\bigoplus_{j=1}^{k-1}C_j\oplus C_k^B.
\end{equation}

\emph{Type $A$.} Put $a:=n-i_k$, $b:=i_1$, and $h:=a+b$, as in
\eqref{eq:type-A-rooted-face-block-lengths}. For notational convenience,
relabel the coordinate vectors $e_{i_k+1},\ldots,e_n,e_1,\ldots,e_{i_1}$ as $f_1,\ldots,f_h$, respectively.
Define
\begin{equation}
\label{eq:type-A-cyclic-coordinate-cone}
C_0^A:=
\left\{
y_0e_0+\sum_{\ell=1}^{h}z_\ell f_\ell:
\begin{array}{l}
\sum_{\ell=1}^{h}z_\ell=0,\qquad
\rho y_0=\sum_{\ell=1}^{a}z_\ell,\\[1mm]
\sum_{\ell=1}^{r}z_\ell\le0,
\quad 1\le r<h
\end{array}
\right\}.
\end{equation}
Then
\begin{equation}
\label{eq:type-A-normal-cone-orthogonal-decomposition}
N\bigl(F_I^A,\bW_n^A(\rho)\bigr)
=C_0^A\oplus\bigoplus_{j=1}^{k-1}C_j.
\end{equation}
Here $\oplus$ denotes a sum of cones with mutually orthogonal
linear spans. In each summand, its linear span is obtained
by retaining the defining equalities and omitting the inequalities.

Empty sums and families of inequalities are understood in the
usual sense. In particular, $C_0^B=\{0\}$ if $i_1=0$,
$C_k^B=\{0\}$ if $i_k=n$, and $C_j=\{0\}$ if
$i_{j+1}=i_j+1$. If $b=0$ in type $A$, the two equalities
in \eqref{eq:type-A-cyclic-coordinate-cone} reduce to
$\sum_{\ell=1}^h z_\ell=0$ and $y_0=0$.
For $k=1$, no cones $C_j$ occur in either decomposition.
\end{lemma}

\begin{proof}
The descriptions \eqref{eq:type-B-normal-cone} and
\eqref{eq:type-A-normal-cone} follow from the formula $N(F_I^\bullet,\bW_n^\bullet(\rho))=(\bW_n^\bullet(\rho))^\circ\cap(\lin F_I^\bullet)^\perp$, the
inequalities in~\eqref{eq:type-B-polar-inequalities} and~\eqref{eq:type-A-polar-inequalities} defining $(\bW_n^\bullet(\rho))^\circ$, and the generators of $F_I^\bullet$.
The assertions about pointedness, simpliciality, spans, and
dimensions follow from the polar face duality recalled above
and Proposition~\ref{prop:finite-BA-polar-cones}.

Next we prove~\eqref{eq:type-B-normal-cone-orthogonal-decomposition}.
Let $I=\{i_1<\cdots<i_k\}$ and consider, for
$y=\sum_{\ell=0}^n y_\ell e_\ell$, the partial sums
$p_r:=\rho y_0+\sum_{\ell=1}^r y_\ell$, $r=0,\ldots,n$.
By~\eqref{eq:type-B-normal-cone}, membership of $y$ in
$N(F_I^B,\bW_n^B(\rho))$ is equivalent to $p_r\le0$ for every
$r=0,\ldots,n$, together with $p_{i_j}=0$ for
$j=1,\ldots,k$.
Thus the partial sums of $\rho y_0,y_1,\ldots,y_n$ form a
nonpositive walk which returns to zero at the indices
$i_1,\ldots,i_k$. Splitting the walk at these return times yields
an initial nonpositive bridge, $k-1$ further nonpositive bridges,
and a final nonpositive walk.

More precisely, on the first block, consisting of the coordinates
$y_0,\ldots,y_{i_1}$, the conditions say that the partial sums
of $\rho y_0,y_1,\ldots,y_{i_1}$ are nonpositive and that the
last one is zero. These are precisely the defining conditions
of $C_0^B$.
For $1\le j<k$, since $p_{i_j}=0$, the conditions on the
coordinates $y_{i_j+1},\ldots,y_{i_{j+1}}$ become
$\sum_{\ell=i_j+1}^r y_\ell\le0$ for
$i_j<r<i_{j+1}$, while $p_{i_{j+1}}=0$ becomes
$\sum_{\ell=i_j+1}^{i_{j+1}}y_\ell=0$.
Hence this block belongs precisely to $C_j$.
In probabilistic terminology, each of these first $k$ blocks
describes a nonpositive bridge.
Finally, since $p_{i_k}=0$, the remaining conditions are
$\sum_{\ell=i_k+1}^r y_\ell\le0$ for $i_k<r\le n$,
with no condition on the terminal value. These are exactly the
defining conditions of $C_k^B$, corresponding to a nonpositive
walk after the last return to zero.

Consequently, the conditions defining
$N(F_I^B,\bW_n^B(\rho))$ separate into the conditions defining the cones
$C_0^B,C_1,\ldots,C_{k-1},C_k^B$. These cones involve the
pairwise disjoint coordinate sets
$\{0,\ldots,i_1\}$,
$\{i_j+1,\ldots,i_{j+1}\}$, $1\le j<k$, and
$\{i_k+1,\ldots,n\}$, respectively. Their linear spans are
therefore mutually orthogonal, which proves
\eqref{eq:type-B-normal-cone-orthogonal-decomposition}.

Next we prove~\eqref{eq:type-A-normal-cone-orthogonal-decomposition}.
With the same notation
$p_r:=\rho y_0+\sum_{\ell=1}^r y_\ell$, the conditions in
\eqref{eq:type-A-normal-cone} say that the partial sums form a
nonpositive walk which is zero at the indices
$i_1,\ldots,i_k$ and, additionally, $\sum_{\ell=1}^n y_\ell = 0$.  The portions between consecutive indices
$i_j$ and $i_{j+1}$ therefore give the nonpositive bridges
described by $C_j$, $1\le j<k$.

It remains to consider the cyclic block of indices
$i_k+1,\ldots,n,1,\ldots,i_1$, joining $i_k$ back to $i_1$.
Relabel the corresponding basis vectors
$e_{i_k+1},\ldots,e_n,e_1,\ldots,e_{i_1}$ as
$f_1,\ldots,f_h$, and denote the coordinates of $y$ in these
directions by $z_1,\ldots,z_h$. The inequalities $p_r\le0$,
together with $p_{i_k}=p_{i_1}=0$ and
$\sum_{\ell=1}^n y_\ell=0$, show that the partial sums of
$z_1,\ldots,z_h$ are nonpositive and return to zero at time $h$,
while the sum of the first $a=n-i_k$ increments equals
$\rho y_0$. These are precisely the defining conditions of
$C_0^A$.

Thus the normal-cone conditions split into those defining
$C_0^A,C_1,\ldots,C_{k-1}$. Since these cones involve disjoint
sets of coordinates, their linear spans are mutually orthogonal,
and \eqref{eq:type-A-normal-cone-orthogonal-decomposition}
follows.
\end{proof}

\subsection{Gaussian walk and bridge persistence probabilities}
\label{subsec:finite-BA-gaussian-walk-bridge}

To evaluate the angles of the cones in the preceding orthogonal
decompositions~\eqref{eq:type-B-normal-cone-orthogonal-decomposition} and~\eqref{eq:type-A-normal-cone-orthogonal-decomposition}, we use the following classical identities; see \cite{SparreAndersen1949,SparreAndersen1953}.

\begin{lemma}[Persistence probabilities for Gaussian walks and bridges]
\label{lem:type-B-ordinary-block-probabilities}
Let $\xi_1,\xi_2,\ldots$ be independent standard normal random
variables, and put $G_0:=0$ and $G_m:=\xi_1+\cdots+\xi_m$
for $m\ge1$. Then
\begin{align}
q_m:=\Pp(G_1\le0,\ldots,G_m\le0)
&=\frac1{4^m}\binom{2m}{m},
\qquad m\ge0,
\label{eq:type-B-walk-persistence}
\\
\Pp(G_1\le0,\ldots,G_{m-1}\le0\mid G_m=0)
&=\frac1m,
\qquad m\ge1.
\label{eq:type-B-ordinary-bridge-probability}
\end{align}
Empty families of inequalities impose no restriction; in particular,
$q_0=1$. Conditioning on $G_m=0$ refers to the conditional law in the sense of multivariate Gaussian distributions.
\end{lemma}

\subsection{External angles for cones of type \texorpdfstring{\(B\)}{B}}
\label{subsec:type-B-external-angles}
We now combine the orthogonal decomposition of the normal
cones of $\bW_n^B(\rho)$ with the Gaussian walk and bridge probabilities from the
preceding subsection to compute the external angles of $\bW_n^B(\rho)$.
\begin{proposition}[External angles of $\bW_n^B(\rho)$]
\label{prop:type-B-external-angle}
Let $n\ge1$, $\rho>0$, and
$I=\{i_1<\cdots<i_k\}\subseteq\{0,\ldots,n\}$ with $k\ge1$.
Then
\begin{equation}
\label{eq:type-B-external-angle}
\alpha\bigl(N(F_I^B,\bW_n^B(\rho))\bigr)
=
\chi_{i_1}^B(\rho)
\left(\prod_{j=1}^{k-1}\frac1{i_{j+1}-i_j}\right)
q_{n-i_k},
\end{equation}
where $q_m$ is defined in \eqref{eq:type-B-walk-persistence},
$\chi_0^B(\rho):=1$, and, for $m\ge1$,
\begin{equation}
\label{eq:type-B-first-block-bridge}
\chi_m^B(\rho)
:=\Pp\left(
Y+G_j\le0\text{ for }j=0,\ldots,m-1
\,\middle|\,Y+G_m=0
\right).
\end{equation}
Here $(G_j)_{j=0}^m$ is the standard Gaussian walk from
Lemma~\ref{lem:type-B-ordinary-block-probabilities}, and
$Y\sim N(0,\rho^2)$ is independent of this walk.
\end{proposition}

\begin{proof}
By the orthogonal decomposition
\eqref{eq:type-B-normal-cone-orthogonal-decomposition}, the
external angle is the product of the angles of $C_0^B$,
$C_1,\ldots,C_{k-1}$, and $C_k^B$. We evaluate these angles
separately.

We use the standard conditioning formula for multivariate Gaussian
vectors. If $Z$ is standard Gaussian in $\R^N$ and $u\in\R^N$
is nonzero, then the Gaussian distribution of $Z$ conditioned on
$\langle Z,u\rangle=0$ is centered with covariance matrix
$I_N-uu^{\mathsf T}/\|u\|^2$. This is the matrix of the orthogonal
projection of $\R^N$ onto $u^\perp$. Hence the conditioned vector
is standard Gaussian in the Euclidean subspace $u^\perp$.

For $m=i_1\ge1$, the vector
$(Y/\rho)e_0+\sum_{j=1}^m\xi_je_j$ is standard Gaussian in
$\lin\{e_0,\ldots,e_m\}$. The equation $Y+G_m=0$ defines precisely
the hyperplane $\lin C_0^B$. Hence the Gaussian vector conditioned
on $Y+G_m=0$ is standard Gaussian in $\lin C_0^B$.
Its membership in $C_0^B$ is exactly the event in
\eqref{eq:type-B-first-block-bridge}, and therefore
$\alpha(C_0^B)=\chi_{i_1}^B(\rho)$. This also holds when $i_1=0$,
since $C_0^B=\{0\}$.

For $1\le j<k$, a vector with independent standard normal
coordinates on the coordinate block
$e_{i_j+1},\ldots,e_{i_{j+1}}$, conditioned on the sum of these
coordinates being zero, is standard Gaussian in $\lin C_j$.
Hence Lemma~\ref{lem:type-B-ordinary-block-probabilities} gives
$\alpha(C_j)=1/(i_{j+1}-i_j)$.

Finally, $\lin C_k^B$ is the coordinate space spanned by
$e_{i_k+1},\ldots,e_n$. Thus a standard Gaussian vector in this
space has independent standard normal coordinates, and membership
in $C_k^B$ is exactly the event that their partial sums remain
nonpositive. By  Lemma~\ref{lem:type-B-ordinary-block-probabilities}, $\alpha(C_k^B)=q_{n-i_k}$.

Multiplying these angles proves
\eqref{eq:type-B-external-angle}.
\end{proof}

For the later asymptotic analysis, we record an alternative
representation of $\chi_m^B(\rho)$ and a bound following from it. Write
$\varphi(x):=(2\pi)^{-1/2}\eee^{-x^2/2}$ and
$\Phi(x):=\int_{-\infty}^x\varphi(u)\,\dint u$ for the
standard normal density and distribution function.

\begin{lemma}
\label{lem:type-B-first-block-representations}
For every $\rho>0$ and integer $m\ge0$,
\begin{equation}
\label{eq:type-B-first-block-expectation}
\chi_m^B(\rho)
=
\sqrt{1+\frac m{\rho^2}}\,
\Ee\left[
\eee^{-G_m^2/(2\rho^2)};
G_1\ge0,\ldots,G_m\ge0
\right],
\end{equation}
where $\Ee[U;E]:=\Ee[U\ind_E]$.
In particular,
\begin{equation}
\label{eq:type-B-first-block-bound}
0\le\chi_m^B(\rho)
\le\sqrt{1+\frac m{\rho^2}}\,q_m.
\end{equation}
\end{lemma}

\begin{proof}
For $m\ge1$, evaluating the density of $Y$ at $-G_m$ in
\eqref{eq:type-B-first-block-bridge} gives
\begin{equation}
\label{eq:type-B-first-block-density-ratio}
\chi_m^B(\rho)
=
\frac{
\Ee\left[
\rho^{-1}\varphi(G_m/\rho);
G_m-G_j\ge0\text{ for }j=0,\ldots,m-1
\right]
}{
(\rho^2+m)^{-1/2}\varphi(0)
}.
\end{equation}
The denominator is the density of $Y+G_m$ at zero.
Reversing the order of $\xi_1,\ldots,\xi_m$ preserves their
joint law and total sum, and turns the event in the numerator
into $\{G_1\ge0,\ldots,G_m\ge0\}$.
Substituting the Gaussian densities proves
\eqref{eq:type-B-first-block-expectation}.
For $m=0$, its right-hand side equals $1=\chi_0^B(\rho)$.
Finally, the exponential is at most one, and
$\Pp(G_1\ge0,\ldots,G_m\ge0)=q_m$, proving
\eqref{eq:type-B-first-block-bound}.
\end{proof}

\subsection{External angles for cones of type \texorpdfstring{\(A\)}{A}}
\label{subsec:type-A-external-angles}

We now compute the external angles of the cone $\bW_n^A(\rho)$ using the orthogonal
decomposition \eqref{eq:type-A-normal-cone-orthogonal-decomposition}.

\begin{proposition}[External angles of $\bW_n^A(\rho)$]
\label{prop:type-A-external-angle}
Let $n\ge1$, $\rho>0$, and
$I=\{i_1<\cdots<i_k\}\subseteq\{0,\ldots,n-1\}$ with $k\ge1$.
Then
\begin{equation}
\label{eq:type-A-external-angle}
\alpha\bigl(N(F_I^A,\bW_n^A(\rho))\bigr)
=
\chi_{n-i_k,i_1}^A(\rho)
\prod_{j=1}^{k-1}\frac1{i_{j+1}-i_j},
\end{equation}
where, for integers $a\ge1$, $b\ge0$, and $h:=a+b$,
\begin{equation}
\label{eq:type-A-cyclic-factor-definition}
\chi_{a,b}^A(\rho)
:=\Pp\left(
G_1\le0,\ldots,G_{h-1}\le0
\,\middle|\,G_h=0,\ Y-G_a=0
\right).
\end{equation}
Here $(G_j)_{j=0}^h$ is the standard Gaussian walk from
Lemma~\ref{lem:type-B-ordinary-block-probabilities}, and
$Y\sim N(0,\rho^2)$ is independent of this walk.
The conditioning refers to the joint Gaussian law under the two
linear constraints displayed in
\eqref{eq:type-A-cyclic-factor-definition}.
\end{proposition}

\begin{proof}
Put $a:=n-i_k$, $b:=i_1$, and $h:=a+b$. By the orthogonal
decomposition~\eqref{eq:type-A-normal-cone-orthogonal-decomposition},
the external angle is the product of the angles of
$C_0^A,C_1,\ldots,C_{k-1}$.

For $1\le j<k$, the cones $C_j$ are the same nonpositive-bridge
cones as in type $B$, and therefore
Lemma~\ref{lem:type-B-ordinary-block-probabilities} gives
$\alpha(C_j)=1/(i_{j+1}-i_j)$.

Next we show that $\alpha(C_0^A)=\chi_{a,b}^A(\rho)$.  Relabel
$e_{i_k+1},\ldots,e_n,e_1,\ldots,e_{i_1}$ as
$f_1,\ldots,f_h$. Then
$(Y/\rho)e_0+\sum_{\ell=1}^h\xi_\ell f_\ell$ is standard Gaussian
in $\lin\{e_0,f_1,\ldots,f_h\}$. Applying the Gaussian conditioning observation from the type-$B$
case successively to the two independent linear constraints
$G_h=0$ and $Y-G_a=0$ shows that the resulting Gaussian vector
is standard Gaussian in their common kernel. By
\eqref{eq:type-A-cyclic-coordinate-cone}, this common kernel is
precisely $\lin C_0^A$. Moreover, membership in $C_0^A$ is exactly
the event $G_1\le0,\ldots,G_{h-1}\le0$. Hence
$\alpha(C_0^A)=\chi_{a,b}^A(\rho)$.
Multiplying the factors proves
\eqref{eq:type-A-external-angle}.
\end{proof}

We record an expectation representation of $\chi_{a,b}^A(\rho)$ and a bound
for the subsequent asymptotic analysis.

\begin{lemma}
\label{lem:type-A-cyclic-factor-representations}
For every $\rho>0$ and integers $a\ge1$, $b\ge0$, put $h:=a+b$.
Then
\begin{equation}
\label{eq:type-A-cyclic-factor-expectation}
\chi_{a,b}^A(\rho)
=
\sqrt{1+\frac{ab}{h\rho^2}}\,
\Ee\left[
\eee^{-G_a^2/(2\rho^2)};
G_1\le0,\ldots,G_{h-1}\le0
\,\middle|\,G_h=0
\right].
\end{equation}
In particular,
\begin{align}
0\le\chi_{a,b}^A(\rho)
&\le\frac1h\sqrt{1+\frac{ab}{h\rho^2}},
\label{eq:type-A-cyclic-factor-bound}
\\
\chi_{a,0}^A(\rho)&=\frac1a.
\label{eq:type-A-cyclic-factor-endpoint}
\end{align}
\end{lemma}

\begin{proof}
Under the conditional law given $G_h=0$, the variable $G_a$ is
centered Gaussian with variance $a-a^2/h=ab/h$, and $Y$ remains
independent of the conditioned walk. Thus $Y-G_a$ has conditional variance
$\rho^2+ab/h>0$. Evaluating the density of $Y$ at $G_a$ gives
\begin{equation}
\label{eq:type-A-cyclic-factor-density-ratio}
\chi_{a,b}^A(\rho)
=
\frac{
\Ee\left[
\rho^{-1}\varphi(G_a/\rho);
G_1\le0,\ldots,G_{h-1}\le0
\,\middle|\,G_h=0
\right]
}{
(\rho^2+ab/h)^{-1/2}\varphi(0)
}.
\end{equation}
Substituting the Gaussian densities proves
\eqref{eq:type-A-cyclic-factor-expectation}.
Bounding its exponential by one and using
\eqref{eq:type-B-ordinary-bridge-probability} proves
\eqref{eq:type-A-cyclic-factor-bound}.
If $b=0$, then $a=h$ and $G_a=0$ under the conditioning, so the
prefactor and exponential in
\eqref{eq:type-A-cyclic-factor-expectation} both equal one.
The same bridge identity \eqref{eq:type-B-ordinary-bridge-probability} yields
\eqref{eq:type-A-cyclic-factor-endpoint}, including $a=1$.
\end{proof}

\subsection{Conic intrinsic volumes: the unsymmetrized formulas}
\label{subsec:finite-BA-intrinsic-volumes-unsymmetrized}

We now combine the face and external angles computed above
in the face-angle formula \eqref{eq:finite-BA-face-angle-sum}.

\begin{proposition}[Conic intrinsic volumes of $\bW_n^\bullet(\rho)$]
\label{prop:finite-BA-intrinsic-volumes-unsymmetrized}
Let $n\ge1$ and $\rho>0$. In each summand below, write
$d_j:=i_{j+1}-i_j$ for $1\le j<k$.
For $1\le k\le n+1$,
\begin{equation}
\label{eq:type-B-intrinsic-volumes-unsymmetrized}
\upsilon_k\bigl(\bW_n^B(\rho)\bigr)
=
\sum_{0\le i_1<\cdots<i_k\le n}
\chi_{i_1}^B(\rho)\,q_{n-i_k}
\left(\prod_{j=1}^{k-1}\frac1{d_j}\right)
\Ord_k\bigl(\rho^2+i_1,d_1,\ldots,d_{k-1}\bigr).
\end{equation}
For $1\le k\le n$,
\begin{multline}
\label{eq:type-A-intrinsic-volumes-unsymmetrized}
\upsilon_k\bigl(\bW_n^A(\rho)\bigr)
=
\sum_{0\le i_1<\cdots<i_k\le n-1}
\chi_{n-i_k,i_1}^A(\rho)
\left(\prod_{j=1}^{k-1}\frac1{d_j}\right)
\int_{z<x}p_{\rho;n-i_k,i_1}(x,z)
\\
\times
\Pp\left(
 x\ge\frac{Z_1}{\sqrt{d_1}}\ge\cdots\ge
 \frac{Z_{k-1}}{\sqrt{d_{k-1}}}\ge z
\right)\,\dint x\,\dint z,
\end{multline}
where $Z_1,\ldots,Z_{k-1}$ are independent standard normal
variables, and $\Ord_k$ and $p_{\rho;a,b}$ are defined in
\eqref{eq:type-B-Gaussian-ordering-probability} and
\eqref{eq:type-A-endpoint-Gaussian-density}, respectively.
For $k=1$,  the products are one,
and the probability under the integral in
\eqref{eq:type-A-intrinsic-volumes-unsymmetrized} is understood
to be one on $z<x$.

The zeroth intrinsic volumes are
\begin{align}
\upsilon_0\bigl(\bW_n^B(\rho)\bigr)
&=
\Pp\left(\rho Z_0+G_j\le0\text{ for }j=0,\ldots,n\right),
\label{eq:type-B-zeroth-intrinsic-volume}
\\
\upsilon_0\bigl(\bW_n^A(\rho)\bigr)
&=
\Pp\left(\rho Z_0+G_j-\frac jnG_n\le0
\text{ for }j=0,\ldots,n-1\right),
\label{eq:type-A-zeroth-intrinsic-volume}
\end{align}
where $Z_0$ is standard normal and independent of the standard
Gaussian walk $(G_j)_{j=0}^n$. The intrinsic volumes vanish for $k>n+1$
in type $B$ and for $k>n$ in type $A$.
\end{proposition}

\begin{proof}
For $k\ge1$, substitute the face-angle formulas from
Lemmas~\ref{lem:type-B-internal-angle} and
\ref{lem:type-A-internal-angle} and the external-angle formulas
from Propositions~\ref{prop:type-B-external-angle} and
\ref{prop:type-A-external-angle} into
\eqref{eq:finite-BA-face-angle-sum}.
For $k=0$, since both cones are pointed, their zero face is $\{0\}$,
and the same face-angle formula gives
$\upsilon_0(C)=\alpha(C^\circ)$.
Apply the polar descriptions
\eqref{eq:type-B-polar-inequalities} and
\eqref{eq:type-A-polar-inequalities} to the standard Gaussian
vector $Z_0e_0+\sum_{j=1}^n\xi_je_j$ and, respectively, to its
orthogonal projection onto $\R e_0\oplus s_n^\perp$, which is
standard Gaussian in that subspace.
In the latter case, the spatial coordinates are
$\xi_j-G_n/n$, whose partial sums are $G_j-(j/n)G_n$.
\end{proof}

\begin{remark}[Weyl chambers]
\label{rem:type-B-unit-height-Weyl-chamber}
For $\rho=1$, the coordinate description
\eqref{eq:type-B-cone-coordinates} identifies $\bW_n^B(1)= \{t_0\ge t_1\ge\cdots\ge t_n\ge0\}$
with the fundamental Weyl chamber of type $B_{n+1}$. Its
conic intrinsic volumes are
\begin{equation}
\label{eq:type-B-unit-height-intrinsic-volumes}
\upsilon_k\bigl(\bW_n^B(1)\bigr)
=\frac{B[n+1,k]}{2^{n+1}(n+1)!},
\qquad 0\le k\le n+1,
\end{equation}
where $B[N,k]$ is the coefficient of $t^k$ in
$\prod_{j=1}^{N}(t+2j-1)$, the type-$B$ analogue of an unsigned
Stirling number of the first kind.
A direct proof based on face and external angles, similar
in spirit to the present approach, is given in~\cite[Theorem~3.3 and its proof]{GodlandKabluchko2022}.
Here $\chi_m^B(1)=1/(m+1)$ for $m\ge0$; for $m\ge1$ this
follows from \eqref{eq:type-B-ordinary-bridge-probability}
applied to the $m+1$ independent standard normal increments
$Y,\xi_1,\ldots,\xi_m$ in
\eqref{eq:type-B-first-block-bridge}.
\end{remark}

\subsection{Conic intrinsic volumes: the symmetrized formulas}
\label{subsec:finite-BA-intrinsic-volumes-symmetrized}

We now average the preceding sums over permutations of the gaps
$d_j=i_{j+1}-i_j$. This removes the ordering constraints among the
corresponding Gaussian variables and gives product formulas suited
to the scaling analysis.

Let $\mathfrak S_r$ denote the permutation group of $\{1,\ldots,r\}$.
If $Z_1,\ldots,Z_r$ are independent with $Z_j\sim N(0,1/d_j)$,
where $d_j>0$, then for $z<x$,
\begin{equation}
\label{eq:finite-BA-ordering-symmetrization}
\sum_{\pi\in\mathfrak S_r}
\Pp\bigl(x\ge Z_{\pi(1)}\ge\cdots\ge Z_{\pi(r)}\ge z\bigr)
=
\prod_{j=1}^r
\bigl(\Phi(\sqrt{d_j}\,x)-\Phi(\sqrt{d_j}\,z)\bigr).
\end{equation}
Indeed, the ordering events partition, up to null sets, the event
that all $Z_j$ lie in $[z,x]$, whose probability is the product on
the right by independence. For $r=0$, we take $\mathfrak S_0$ to
consist of the empty permutation and interpret both sides as one.

\begin{proposition}[Conic intrinsic volumes of $\bW_n^\bullet(\rho)$: the symmetrized formula]
\label{prop:finite-BA-intrinsic-volumes-symmetrized}
Let $n\ge1$ and $\rho>0$. For $1\le k\le n+1$,
\begin{multline}
\label{eq:type-B-intrinsic-volumes-symmetrized}
\upsilon_k\bigl(\bW_n^B(\rho)\bigr)
=
\frac1{(k-1)!}
\sum_{\substack{
 m\ge0,\ d_1,\ldots,d_{k-1}\ge1\\
 m+d_1+\cdots+d_{k-1}\le n
}}
\chi_m^B(\rho)\,q_{n-m-d_1-\cdots-d_{k-1}}
\\
\times \int_0^\infty\sqrt{\rho^2+m}\,
\varphi\bigl(\sqrt{\rho^2+m}\,x\bigr)
\prod_{j=1}^{k-1}
\frac{\Phi(\sqrt{d_j}\,x)-1/2}{d_j}\,\dint x.
\end{multline}
For $1\le k\le n$,
\begin{equation}
\label{eq:type-A-intrinsic-volumes-symmetrized}
\upsilon_k\bigl(\bW_n^A(\rho)\bigr)
=
\frac1{(k-1)!}
\sum_{\substack{
 h,d_1,\ldots,d_{k-1}\ge1\\
 h+d_1+\cdots+d_{k-1}=n
}}
\int_{z<x}R_{h,\rho}(x,z)
\prod_{j=1}^{k-1}
\frac{\Phi(\sqrt{d_j}\,x)-\Phi(\sqrt{d_j}\,z)}{d_j}
\,\dint x\,\dint z,
\end{equation}
where,  for every integer $h\ge1$ and $\rho>0$,
\begin{equation}
\label{eq:type-A-symmetrized-kernel}
R_{h,\rho}(x,z)
:=\sum_{a=1}^h
\chi_{a,h-a}^A(\rho)\,p_{\rho;a,h-a}(x,z),
\qquad (x,z)\in\R^2,
\end{equation}
and $p_{\rho;a,b}$ is the density in
\eqref{eq:type-A-endpoint-Gaussian-density}.
All summation variables are integers, and the integrals are finite.
The type-$A$ integral is over the whole half-plane
$\{(x,z)\in\R^2:z<x\}$.
\end{proposition}

\begin{proof}
\emph{Type $B$.}
In \eqref{eq:type-B-intrinsic-volumes-unsymmetrized}, set
$m:=i_1$ and $d_j:=i_{j+1}-i_j$ for $1\le j<k$.
This gives a bijection with the integer tuples in
\eqref{eq:type-B-intrinsic-volumes-symmetrized}, with inverse
$i_r=m+\sum_{j=1}^{r-1}d_j$.
For fixed $m$, both the admissible set of gap tuples and the factor
$\chi_m^B(\rho)q_{n-m-\sum_jd_j}\prod_jd_j^{-1}$ are invariant
under permutations of the gaps  $d_1,\ldots, d_{k-1}$. We may therefore replace
$\Ord_k(\rho^2+m,d_1,\ldots,d_{k-1})$ by its average over the
$(k-1)!$ permutations of $d_1,\ldots,d_{k-1}$.
Conditioning on the upper Gaussian coordinate, whose density at
$x$ is $\sqrt{\rho^2+m}\,\varphi(\sqrt{\rho^2+m}\,x)$, and
applying \eqref{eq:finite-BA-ordering-symmetrization} with
$r=k-1$ and $z=0$ gives
\eqref{eq:type-B-intrinsic-volumes-symmetrized}.

\medskip
\emph{Type $A$.}
In \eqref{eq:type-A-intrinsic-volumes-unsymmetrized}, set
$a:=n-i_k$, $b:=i_1$, and $d_j:=i_{j+1}-i_j$ for $1\le j<k$.
The selected indices correspond bijectively to integer tuples
satisfying $a\ge1$, $b\ge0$, $d_j\ge1$, and
$a+b+\sum_{j=1}^{k-1}d_j=n$, with inverse
$i_r=b+\sum_{j=1}^{r-1}d_j$.
For fixed $a,b$, the set of admissible gap tuples and the factor
$\chi_{a,b}^A(\rho)p_{\rho;a,b}(x,z)\prod_jd_j^{-1}$ are
invariant under permutations of the gaps $d_1,\ldots, d_{k-1}$.
Average the ordering probability in \eqref{eq:type-A-intrinsic-volumes-unsymmetrized} over these permutations and apply
\eqref{eq:finite-BA-ordering-symmetrization} with $r=k-1$.
Finally, put $h:=a+b$ and sum over $a=1,\ldots,h$, including
$a=h$ when $b=0$. By \eqref{eq:type-A-symmetrized-kernel},
this gives \eqref{eq:type-A-intrinsic-volumes-symmetrized}.

All sums are finite, and the integrals are finite because the
distribution-function differences are bounded and the remaining
factors are Gaussian densities or finite sums thereof.
Thus all exchanges of sums and integrals are justified.
\end{proof}

\section{Scaling limits of the conic intrinsic volumes}
\label{sec:finite-BA-scaling-limits}
We now study the limits of the conic intrinsic volumes of $\bW_n^\bullet(\rho_n)$ as $n\to\infty$ when the height $\rho=\rho_n$ grows on the scale $\sqrt n$. As we shall see in Section~\ref{sec:brownian-ou-absorption}, these intrinsic volumes determine Gaussian-point absorption probabilities for the convex hulls of Gaussian random walks and bridges; under the usual diffusive scaling, the latter converge to the corresponding Brownian-motion and Brownian-bridge convex hulls. Thus the scale $\sqrt n$ is precisely the one relevant for passing from the finite-dimensional cone formulas \eqref{eq:type-B-intrinsic-volumes-symmetrized} and \eqref{eq:type-A-intrinsic-volumes-symmetrized} to the Brownian absorption problem.

\subsection{Main theorem}
\label{subsec:finite-BA-scaling-normalization}

Multiplying a generating set of a cone by the same positive number
leaves the cone unchanged. Thus, for every positive sequence
$(\rho_n)_{n\ge1}$,
\begin{equation}
\label{eq:finite-BA-scaling-normalization}
\bW_n^\bullet(\rho_n)
=
\pos\left(
\rho_n e_0 + \bK_n^\bullet
\right)
=
\pos\left(
\frac{\rho_n}{\sqrt n}e_0+\frac1{\sqrt n}\bK_n^\bullet
\right),
\qquad \bullet\in\{B,A\}.
\end{equation}
The following theorem determines the limits of the conic intrinsic volumes of \(\bW_n^\bullet(\rho_n)\) when \(\rho_n/\sqrt n\to c\in(0,\infty)\), and proves convergence of the entire intrinsic-volume sequence in \(\ell^1\).
\begin{theorem}[Scaling limits of the conic intrinsic volumes]
\label{thm:finite-BA-scaling-limit}
Let $(\rho_n)_{n\ge1}$ be a sequence of positive numbers such that
$\rho_n/\sqrt n\to c\in(0,\infty)$. For type $B$, define
\begin{align}
L_0^B(c)
&:=\frac{\arctan c}{\pi},
\label{eq:type-B-limiting-zeroth-intrinsic-volume}
\\
L_k^B(c)
&:=\frac1{\pi(k-1)!}
\int_0^{\arsinh(1/c)}
\frac{x^{k-1}\,\dint x}{\sqrt{c^{-2}-\sinh^2x}},
\qquad k\ge1.
\label{eq:type-B-limiting-intrinsic-volumes}
\end{align}
For type $A$, put $\theta_c:=2\arsinh(1/(2c))$ and define
\begin{align}
L_0^A(c)
&:=\frac12\left(1-\frac1{\sqrt{1+4c^2}}\right),
\label{eq:type-A-limiting-zeroth-intrinsic-volume}
\\
L_k^A(c)
&:=\frac{\theta_c^k}{2k!\sinh\theta_c},
\qquad k\ge1.
\label{eq:type-A-limiting-intrinsic-volumes}
\end{align}
For each $\bullet\in\{B,A\}$ and every fixed integer $k\ge0$,
\begin{equation}
\label{eq:finite-BA-intrinsic-volume-scaling-limit}
\lim_{n\to\infty}
\upsilon_k\bigl(\bW_n^\bullet(\rho_n)\bigr)
=L_k^\bullet(c).
\end{equation}
The limiting coefficients $L_k^\bullet(c)$ satisfy
\begin{equation}
\label{eq:finite-BA-limiting-parity-normalization}
\sum_{j=0}^\infty L_{2j}^\bullet(c)
=
\sum_{j=0}^\infty L_{2j+1}^\bullet(c)
=\frac12,
\qquad \bullet\in\{B,A\}.
\end{equation}
In particular, $\sum_{k=0}^\infty L_k^\bullet(c)=1$, and the
convergence also holds in $\ell^1$:
\begin{equation}
\label{eq:finite-BA-intrinsic-volume-l1-convergence}
\lim_{n\to\infty}
\sum_{k=0}^\infty
\left|
\upsilon_k\bigl(\bW_n^\bullet(\rho_n)\bigr)-L_k^\bullet(c)
\right|
=0,
\qquad \bullet\in\{B,A\}.
\end{equation}
\end{theorem}

The proof is given in the following subsections. We first assume
$\rho_n=c\sqrt n$. Our starting points are the symmetrized formulas \eqref{eq:type-B-intrinsic-volumes-symmetrized} and \eqref{eq:type-A-intrinsic-volumes-symmetrized}. For $k\ge1$, we pass to the limit in
\eqref{eq:type-B-intrinsic-volumes-symmetrized} and
\eqref{eq:type-A-intrinsic-volumes-symmetrized}, controlling the
contributions from gaps that are small compared with $n$, and
evaluate the resulting integrals by Laplace transforms.
For $k=0$, we take limits in the probabilities appearing in
\eqref{eq:type-B-zeroth-intrinsic-volume} and
\eqref{eq:type-A-zeroth-intrinsic-volume}.
Finally, we verify \eqref{eq:finite-BA-limiting-parity-normalization}
and complete the proof for $\rho_n/\sqrt n\to c$, including
convergence in $\ell^1$.

\subsection{Asymptotics of the walk and bridge probabilities}
\label{subsec:finite-BA-probability-asymptotics}

We first obtain the estimates for $q_m$, $\chi_m^B(c\sqrt n)$, and
$\chi_{a,b}^A(c\sqrt n)$ needed to pass to the limit in
\eqref{eq:type-B-intrinsic-volumes-symmetrized} and
\eqref{eq:type-A-intrinsic-volumes-symmetrized}.
For the standard Gaussian random walk $(G_j)_{j=0}^\infty$, put
\[
E_m:=\{G_1\ge0,\ldots,G_m\ge0\},\qquad
D_h:=\{G_1\le0,\ldots,G_{h-1}\le0\},
\quad m,h\ge1.
\]

\begin{lemma}
\label{lem:finite-BA-conditional-exponential-limits}
For every $M\in[0,\infty)$,
\begin{equation}
\label{eq:gaussian-conditional-rayleigh-uniform-transform}
\sup_{0\le\lambda\le M}
\left|
\Ee\left[\eee^{-\lambda G_m^2/(2m)}\,\middle|\,E_m\right]
-\frac1{1+\lambda}
\right|
\longrightarrow0
\qquad\text{as }m\to\infty.
\end{equation}
Moreover,
\begin{equation}
\label{eq:gaussian-conditional-bridge-uniform-transform}
\sup_{\substack{a\in\{0,\ldots,h\}\\0\le\lambda\le M}}
\left|
\Ee\left[\eee^{-\lambda G_a^2/(2h)}\,\middle|\,G_h=0,D_h\right]
-\left(1+\lambda\frac ah\left(1-\frac ah\right)\right)^{-3/2}
\right|
\longrightarrow0
\qquad\text{as }h\to\infty.
\end{equation}
Conditioning on $G_h=0,D_h$ means first taking the Gaussian bridge
law given $G_h=0$ and then conditioning on $D_h$, whose probability
under that law is $1/h$ by
\eqref{eq:type-B-ordinary-bridge-probability}.
\end{lemma}

\begin{proof}
\emph{The walk.}
A classical conditional limit theorem of
Iglehart~\cite[Proposition~2.1]{Iglehart1974} states, in the present
Gaussian setting, that the law of $G_m/\sqrt m$ given $E_m$ converges
weakly to the Rayleigh distribution with density
$x\eee^{-x^2/2}\ind_{(0,\infty)}(x)$; see also~\cite[(1.3)]{Caravenna2005}.
Applying this convergence to the bounded continuous function
$x\mapsto\eee^{-\lambda x^2/2}$ gives
\[
\Ee\left[\eee^{-\lambda G_m^2/(2m)}\,\middle|\,E_m\right]
\longrightarrow
\int_0^\infty x\eee^{-(1+\lambda)x^2/2}\,\dint x
=\frac1{1+\lambda}.
\]
The expectations are nonincreasing in $\lambda$, and the limit is
continuous. A finite subdivision of $[0,M]$ therefore makes the
convergence uniform, proving
\eqref{eq:gaussian-conditional-rayleigh-uniform-transform}.

\medskip
\emph{The bridge.}
Under the law conditioned on $G_h=0,D_h$, let $\mathcal E_h$ be the
continuous function on $[0,1]$ satisfying
$\mathcal E_h(j/h)=-G_j/\sqrt h$ for $0\le j\le h$ and linear
between consecutive grid points. By symmetry and
\cite[Corollary~2.5]{CaravennaChaumont2013},
\begin{equation}
\label{eq:gaussian-conditioned-bridge-excursion-limit}
\mathcal E_h\Longrightarrow\mathfrak e
\quad\text{weakly in }C[0,1]\text{ with the uniform topology},
\end{equation}
where $\mathfrak e$ is a normalized Brownian excursion on $[0,1]$.
The hypotheses of that result hold because the increments have
mean zero, variance one, and a bounded density.
For $0<t<1$, the random variable $\mathfrak e(t)$ has the same
distribution as the Euclidean norm of a three-dimensional centered
Gaussian vector with covariance $t(1-t)I_3$, since a normalized
Brownian excursion is equal in law to a three-dimensional Bessel
bridge; see Revuz--Yor~\cite[Chapter~XII, Theorem~4.2]{RevuzYor1999}.
Consequently, $\mathfrak e(t)$ has density
\[
\sqrt{\frac2\pi}\,
\frac{x^2}{[t(1-t)]^{3/2}}
\eee^{-x^2/(2t(1-t))},\qquad x>0.
\]
Integration gives
\begin{equation}
\label{eq:Brownian-excursion-one-time-transform}
\Ee\left[\eee^{-\lambda\mathfrak e(t)^2/2}\right]
=(1+\lambda t(1-t))^{-3/2},
\qquad \lambda\ge0,\quad 0\le t\le1,
\end{equation}
including $t=0,1$ because the excursion vanishes at both endpoints.

To prove uniformity in
\eqref{eq:gaussian-conditional-bridge-uniform-transform}, suppose
that it fails. Along a subsequence $h\to\infty$, we can then choose
$a_h\in\{0,\ldots,h\}$ and $\lambda_h\in[0,M]$ for which the
absolute difference stays bounded away from zero. Passing to a
further subsequence, assume $a_h/h\to t\in[0,1]$ and
$\lambda_h\to\lambda\in[0,M]$.
The map $(f,t,\lambda)\mapsto\eee^{-\lambda f(t)^2/2}$ is bounded
and continuous on $C[0,1]\times[0,1]\times[0,M]$.
Thus \eqref{eq:gaussian-conditioned-bridge-excursion-limit} implies
that the conditional expectation tends to the right-hand side of
\eqref{eq:Brownian-excursion-one-time-transform}. The explicit
expression subtracted from it has the same limit, a contradiction.
This proves the assertion, including indices approaching either
endpoint.
\end{proof}

We now apply these limits to the expectation representations
\eqref{eq:type-B-first-block-expectation} and
\eqref{eq:type-A-cyclic-factor-expectation}.

\begin{lemma}
\label{lem:finite-BA-probability-asymptotics}
Let $c>0$. For every $0<\varepsilon<M<\infty$, as $n\to\infty$,
\begin{align}
\sup_{\varepsilon n\le m\le Mn}
\left|\sqrt n\,q_m-\frac1{\sqrt{\pi(m/n)}}\right|
&\longrightarrow0,
\label{eq:type-B-scaled-terminal-asymptotics}
\\
\sup_{\varepsilon n\le m\le Mn}
\left|
\sqrt n\,\chi_m^B(c\sqrt n)
-\frac{c}{\sqrt{\pi(m/n)(c^2+m/n)}}
\right|
&\longrightarrow0,
\label{eq:type-B-scaled-first-block-asymptotics}
\\
\sup_{\substack{\varepsilon n\le h\le Mn\\1\le a\le h}}
\left|
n\chi_{a,h-a}^A(c\sqrt n)
-\frac{c^2}{c^2h/n+(a/n)((h-a)/n)}
\right|
&\longrightarrow0.
\label{eq:type-A-scaled-cyclic-factor-asymptotics}
\end{align}
The indices $m,h,a$ in these suprema are integers. No restriction
is imposed on $a/h$ in the third formula; in particular, $a=h$
is included.

There is an absolute constant $C<\infty$ such that, for all integers
$n\ge1$, $m\ge0$, and $a\ge1$, $b\ge0$ with $h=a+b$,
\begin{align}
q_m&\le\frac{C}{\sqrt{m+1}},\qquad
\chi_m^B(c\sqrt n)
\le\frac{C}{\sqrt{m+1}}\sqrt{1+\frac{m}{c^2n}},
\label{eq:type-B-scaled-block-bounds}
\\
\chi_{a,b}^A(c\sqrt n)
&\le\frac1h\sqrt{1+\frac{ab}{hc^2n}}
\le\frac1h\sqrt{1+\frac{h}{4c^2n}}.
\label{eq:type-A-scaled-cyclic-factor-bound}
\end{align}
\end{lemma}

\begin{proof}
\emph{Type $B$.}
Stirling's formula in \eqref{eq:type-B-walk-persistence} gives
$\sqrt m\,q_m\to\pi^{-1/2}$ and
$q_m\le C/\sqrt{m+1}$. The convergence is uniform for
$m\ge\varepsilon n$, proving
\eqref{eq:type-B-scaled-terminal-asymptotics}.
By symmetry, $\Pp(E_m)=q_m$, so
\eqref{eq:type-B-first-block-expectation} becomes
\begin{equation}
\label{eq:type-B-first-block-scaled-expectation}
\chi_m^B(c\sqrt n)
=
\sqrt{1+\frac{m}{c^2n}}\,q_m\,
\Ee\left[\eee^{-G_m^2/(2c^2n)}\,\middle|\,E_m\right],
\qquad m\ge1.
\end{equation}
Apply \eqref{eq:gaussian-conditional-rayleigh-uniform-transform}
with $\lambda=m/(c^2n)$. Uniformly for
$\varepsilon n\le m\le Mn$, the expectation is
$(1+m/(c^2n))^{-1}+o(1)$.
Combining this with the preceding estimate for $q_m$ proves
\eqref{eq:type-B-scaled-first-block-asymptotics}.
The bound for $\chi_m^B$ follows from
\eqref{eq:type-B-first-block-bound}, also when $m=0$.

\medskip
\emph{Type $A$.}
Since $\Pp(D_h\mid G_h=0)=1/h$,
\eqref{eq:type-A-cyclic-factor-expectation} gives
\begin{equation}
\label{eq:type-A-cyclic-factor-scaled-expectation}
n\chi_{a,h-a}^A(c\sqrt n)
=
\frac nh\sqrt{1+\frac{a(h-a)}{hc^2n}}\,
\Ee\left[\eee^{-G_a^2/(2c^2n)}\,\middle|\,G_h=0,D_h\right].
\end{equation}
Apply \eqref{eq:gaussian-conditional-bridge-uniform-transform}
with $\lambda=h/(c^2n)$. Uniformly for
$\varepsilon n\le h\le Mn$ and $1\le a\le h$, the expectation is
$(1+a(h-a)/(hc^2n))^{-3/2}+o(1)$.
The prefactor in \eqref{eq:type-A-cyclic-factor-scaled-expectation}
is bounded on this range. Multiplication therefore gives, uniformly,
\[
n\chi_{a,h-a}^A(c\sqrt n)
=
\frac nh\left(1+\frac{a(h-a)}{hc^2n}\right)^{-1}+o(1)
=
\frac{c^2}{c^2h/n+(a/n)((h-a)/n)}+o(1),
\]
which proves \eqref{eq:type-A-scaled-cyclic-factor-asymptotics}.
Finally, \eqref{eq:type-A-scaled-cyclic-factor-bound} follows from
\eqref{eq:type-A-cyclic-factor-bound} and $ab\le h^2/4$.
\end{proof}

\subsection{Passage from discrete sums to multidimensional integrals}
\label{subsec:finite-BA-continuum-limits}

We now pass to the limit in
\eqref{eq:type-B-intrinsic-volumes-symmetrized} and
\eqref{eq:type-A-intrinsic-volumes-symmetrized} with
$\rho=c\sqrt n$, where $c>0$ is fixed. We first record estimates
for the sums and integrals near the boundary of their domains.
For $r\ge1$, write
\[
\Delta_r:=\{(u_1,\ldots,u_r)\in(0,\infty)^r:
 u_1+\cdots+u_r<1\}.
\]
An integral over $\Delta_0$ means evaluation at the empty tuple;
its coordinate sum is zero.

\begin{lemma}
\label{lem:finite-BA-simplex-estimates}
Let $p\ge1$, $\beta_1,\ldots,\beta_p>0$, and
$\beta:=\sum_{i=1}^p\beta_i$. There is a constant $C$, depending
only on these parameters, such that, for every integer $N\ge0$,
\begin{equation}
\label{eq:finite-BA-discrete-simplex-bound}
\sum_{\substack{j_1,\ldots,j_p\ge0\\j_1+\cdots+j_p=N}}
\prod_{i=1}^p(j_i+1)^{\beta_i-1}
\le C(N+1)^{\beta-1}.
\end{equation}
For $p\ge2$, any $i\in\{1,\ldots,p\}$, and integers
$0\le H\le N/2$,
\begin{equation}
\label{eq:finite-BA-discrete-simplex-boundary}
\sum_{\substack{j_1,\ldots,j_p\ge0\\j_1+\cdots+j_p=N,\ j_i\le H}}
\prod_{\ell=1}^p(j_\ell+1)^{\beta_\ell-1}
\le C(N+1)^{\beta-1}
\left(\frac{H+1}{N+1}\right)^{\beta_i}.
\end{equation}
All summation variables are integers. The corresponding integral,
with $u_p:=1-\sum_{i=1}^{p-1}u_i$, is
\begin{equation}
\label{eq:finite-BA-continuous-simplex-integral}
\int_{\Delta_{p-1}}\prod_{i=1}^p u_i^{\beta_i-1}
\,\dint u_1\cdots\dint u_{p-1}
=\frac{\prod_{i=1}^p\Gamma(\beta_i)}{\Gamma(\beta)}.
\end{equation}
For $0<\delta\le1/2$, its portion with $u_i\le\delta$ is at
most $C\delta^{\beta_i}$; for $p=1$ this portion is empty.
\end{lemma}

\begin{proof}
For $v,w>0$, we have
\begin{equation}\label{eq:beta_sum_inequality}
\sum_{j=0}^N(j+1)^{v-1}(N-j+1)^{w-1}
\le C_{v,w}(N+1)^{v+w-1}.
\end{equation}
Indeed, split the sum into $j\le N/2$ and $j>N/2$.
On the first range, $N-j+1$ is comparable to $N+1$, while on the
second range, $j+1$ is comparable to $N+1$.  Using $\sum_{j=0}^m (j+1)^{v-1}\le C_v(m+1)^v$,  $v>0$,
on the two ranges, respectively, gives the claimed inequality.

Induction on $p$ combined with~\eqref{eq:beta_sum_inequality} proves \eqref{eq:finite-BA-discrete-simplex-bound}.
If $j_i\le H\le N/2$, apply \eqref{eq:finite-BA-discrete-simplex-bound} to the other $p-1$
indices and sum over $j_i$; the result is at most
$C(N+1)^{\beta-\beta_i-1}(H+1)^{\beta_i}$, proving
\eqref{eq:finite-BA-discrete-simplex-boundary}.
Repeated beta integration gives
\eqref{eq:finite-BA-continuous-simplex-integral}.
For $p\ge2$, integrating first over the other coordinates reduces
its portion with $u_i\le\delta$ to a constant times
$\int_0^\delta v^{\beta_i-1}(1-v)^{\beta-\beta_i-1}\,\dint v$,
which is bounded as claimed.
\end{proof}

In the next result, we express the limit of
$\upsilon_k\bigl(\bW_n^\bullet(c\sqrt n)\bigr)$ as $n\to\infty$
in the form of a multidimensional integral. The resulting integrals
will be evaluated in
Subsection~\ref{subsec:finite-BA-continuum-integrals}. For $c,t>0$ and $x\ge0$, define
\begin{align}
g_c(t,x)&:=\frac{c}{\sqrt{\pi t}}
 \varphi\bigl(\sqrt{c^2+t}\,x\bigr),
\label{eq:type-B-limiting-initial-kernel}
\\
h(t)&:=\frac1{\sqrt{\pi t}},
\label{eq:type-B-limiting-terminal-kernel}
\\
f_x(t)&:=\frac{\Phi(\sqrt t\,x)-1/2}{t}.
\label{eq:type-B-limiting-ordinary-kernel}
\end{align}
For $s,u>0$ and $z<x$, put
\begin{align}
R_c(s;x,z)
&:=\frac{c^2}{2\pi}\int_0^s
\frac{\eee^{-\frac12[(s-\alpha)x^2+\alpha z^2+c^2(x-z)^2]}}
{\sqrt{\alpha(s-\alpha)+c^2s}}\,\dint\alpha,
\label{eq:type-A-limiting-kernel}
\\
b_u(x,z)&:=\frac{\Phi(\sqrt u\,x)-\Phi(\sqrt u\,z)}u.
\label{eq:type-A-limiting-ordinary-kernel}
\end{align}
The formula for $R_c$ also defines it for all $(x,z)\in\R^2$.

\begin{proposition}[Limits of the conic intrinsic volumes as multidimensional integrals]\label{prop:finite-BA-continuum-limits}
Let $c>0$ and let $k\ge1$ be an integer. Define
\begin{equation}
\label{eq:type-B-unit-time-continuum-integral}
\begin{aligned}
I_k^B(c):={}&\frac1{(k-1)!}
\int_0^\infty\int_{\Delta_k}
g_c(t,x)\,
h\left(1-t-\sum_{j=1}^{k-1}u_j\right)
\prod_{j=1}^{k-1}f_x(u_j)
\,\dint t\,\dint u_1\cdots\dint u_{k-1}\,\dint x,
\end{aligned}
\end{equation}
where the coordinates of $\Delta_k$ are
$(t,u_1,\ldots,u_{k-1})$, and
\begin{equation}
\label{eq:type-A-unit-time-continuum-integral}
\begin{aligned}
I_k^A(c):={}&\frac1{(k-1)!}
\int_{\Delta_{k-1}}\int_{z<x}
R_c\left(1-\sum_{j=1}^{k-1}u_j;x,z\right)
\prod_{j=1}^{k-1}b_{u_j}(x,z)
\,\dint x\,\dint z\,\dint u_1\cdots\dint u_{k-1}.
\end{aligned}
\end{equation}
For $k=1$, the $u_j$ are absent and the products are one;
in particular, $I_1^A(c)=\int_{z<x}R_c(1;x,z)\,\dint x\,\dint z$.
Both integrals are finite, and
\begin{equation}
\label{eq:finite-BA-unit-time-continuum-limit}
\lim_{n\to\infty}\upsilon_k\bigl(\bW_n^\bullet(c\sqrt n)\bigr)
=I_k^\bullet(c),\qquad \bullet\in\{B,A\}.
\end{equation}
Moreover, $I_k^B$ and $I_k^A$ are continuous on $(0,\infty)$
and take values in $[0,1]$.
\end{proposition}

\begin{proof}
Constants in the estimates below may depend on $c$ and $k$, but
not on $n$ or the summation indices.

\smallskip
\emph{Step 1: Pass to the limit for type $B$.}
For integers $m,\ell\ge0$, set
\[
g_{n,c}(m,x):=\sqrt n\,\chi_m^B(c\sqrt n)
 \sqrt{c^2+m/n}\,\varphi\bigl(\sqrt{c^2+m/n}\,x\bigr),
\qquad h_n(\ell):=\sqrt n\,q_\ell.
\]
In \eqref{eq:type-B-intrinsic-volumes-symmetrized}, replace the
integration variable by $x/\sqrt n$ and put
$\ell:=n-m-\sum_jd_j$. This gives the exact identity
\begin{equation}
\label{eq:type-B-unit-time-scaled-sum}
\upsilon_k\bigl(\bW_n^B(c\sqrt n)\bigr)
=\frac1{(k-1)!}\int_0^\infty\frac1{n^k}
\sum_{\substack{m,\ell\ge0,\ d_1,\ldots,d_{k-1}\ge1\\
 m+\ell+d_1+\cdots+d_{k-1}=n}}
g_{n,c}(m,x)h_n(\ell)\prod_{j=1}^{k-1}f_x(d_j/n)\,\dint x.
\end{equation}

For each fixed $x\ge0$, the sum in
\eqref{eq:type-B-unit-time-scaled-sum} will be treated as a
$k$-dimensional Riemann sum. More precisely, fix $0<\delta<1/(k+1)$.  After restricting to
$m,\ell,d_1,\ldots,d_{k-1}>\delta n$, we shall prove that
\begin{multline}
\label{eq:type-B-truncated-riemann-sum-limit}
\frac1{n^k}
\sum_{\substack{
m+\ell+d_1+\cdots+d_{k-1}=n\\
m,\ell,d_1,\ldots,d_{k-1}>\delta n}}
g_{n,c}(m,x)h_n(\ell)
\prod_{j=1}^{k-1}f_x(d_j/n)
\\
\longrightarrow
\int_{\substack{(t,u_1,\ldots,u_{k-1})\in\Delta_k\\
t,\ u_1,\ldots,u_{k-1},\
1-t-\sum_j u_j>\delta}}
g_c(t,x)\,
h\left(1-t-\sum_{j=1}^{k-1}u_j\right)
\prod_{j=1}^{k-1}f_x(u_j)
\,\dint t\,\dint u_1\cdots\dint u_{k-1}.
\end{multline}
We shall then use a bound integrable in $x$ to pass this truncated
limit through the outer $x$-integral. Finally, we shall show that
the contributions from the omitted boundary regions vanish as
$\delta\downarrow0$.

By Lemma~\ref{lem:finite-BA-probability-asymptotics}, for each fixed
$x\ge0$ the differences $g_{n,c}(m,x)-g_c(m/n,x)$ and
$h_n(\ell)-h(\ell/n)$ tend to zero uniformly on the truncated
index set. The limiting function is continuous on the corresponding
compact truncated simplex. Hence the usual $k$-dimensional
Riemann-sum theorem proves
\eqref{eq:type-B-truncated-riemann-sum-limit}.

We next justify integration in $x$, still retaining the truncation at $\delta n$, respectively $\delta$. For $m,\ell\le n$ and $d\ge1$, the bounds in
\eqref{eq:type-B-scaled-block-bounds} imply
\begin{equation}
\label{eq:type-B-unit-time-kernel-bounds}
\begin{aligned}
g_{n,c}(m,x)&\le C_c\varphi(cx)
 \left(\frac{m+1}{n}\right)^{-1/2},
&h_n(\ell)&\le C\left(\frac{\ell+1}{n}\right)^{-1/2},
\\
f_x(d/n)&\le Cx\left(\frac{d+1}{n}\right)^{-1/2}.
\end{aligned}
\end{equation}
For the first bound use
$\sqrt{1+m/(c^2n)}\sqrt{c^2+m/n}=(c^2+m/n)/c$;
for the last use $\Phi(v)-1/2\le v/\sqrt{2\pi}$ for $v\ge0$.

Applying the bounds in \eqref{eq:type-B-unit-time-kernel-bounds}, we majorize the sum in \eqref{eq:type-B-unit-time-scaled-sum} by the corresponding sum in which the variables \(d_j\) are allowed to take the value zero. Lemma~\ref{lem:finite-BA-simplex-estimates}, with \(p=k+1\) and all \(\beta_i=1/2\), then shows that the integrand after summation is bounded by \(C_{c,k}x^{k-1}\varphi(cx)\).
This bound  permits
integration in $x$ by dominated convergence. Thus,
\begin{multline}
\label{eq:type-B-truncated-integrated-limit}
\int_0^\infty \frac1{n^k}
\sum_{\substack{
m+\ell+d_1+\cdots+d_{k-1}=n\\
m,\ell,d_1,\ldots,d_{k-1}>\delta n}}
g_{n,c}(m,x)h_n(\ell)
\prod_{j=1}^{k-1}f_x(d_j/n)\,\dint x
\\
\longrightarrow
\int_0^\infty
\int_{\substack{(t,u_1,\ldots,u_{k-1})\in\Delta_k\\
t,\ u_1,\ldots,u_{k-1},\
1-t-\sum_j u_j>\delta}}
g_c(t,x)\,
h\left(1-t-\sum_{j=1}^{k-1}u_j\right)
\prod_{j=1}^{k-1}f_x(u_j)
\,\dint t\,\dint u_1\cdots\dint u_{k-1}\,\dint x .
\end{multline}

It remains to remove the truncation in the discrete and continuous
domains. By Lemma~\ref{lem:finite-BA-simplex-estimates}, the
contribution to the discrete sum from the region where at least one of
$m,\ell,d_1,\ldots,d_{k-1}$ is at most $\delta n$ is bounded, for
each $x\ge0$, by
\begin{equation}
\label{eq:type-B-unit-time-short-index-bound}
C_{c,k}x^{k-1}\varphi(cx)\sqrt{\delta+1/n}.
\end{equation}
After integration in $x$, this contribution is therefore
$O_{c,k}(\sqrt{\delta+1/n})$. In particular, the terms with
$m=0$ or $\ell=0$ contribute $O_{c,k}(n^{-1/2})$.

On the continuous side, the integrand is bounded by a constant
multiple of
\[
x^{k-1}\varphi(cx)
\left[
t\left(1-t-\sum_j u_j\right)\prod_j u_j
\right]^{-1/2}.
\]
Lemma~\ref{lem:finite-BA-simplex-estimates} shows that this majorant
is integrable and that its integral over the region where at least one
of the time coordinates is at most $\delta$ is
$O_{c,k}(\sqrt\delta)$. Combining these estimates with
\eqref{eq:type-B-truncated-integrated-limit}, letting first
$n\to\infty$ for fixed $\delta$ and then $\delta\downarrow0$,
proves \eqref{eq:finite-BA-unit-time-continuum-limit} for type $B$.

\medskip
\emph{Step 2: Obtain the limit of $R_{h,c\sqrt n}$.}
For $1\le h\le n$, define
\begin{equation}
\label{eq:type-A-scaled-kernel-definition}
R_{n,c}(h;x,z):=\frac1n
 R_{h,c\sqrt n}(x/\sqrt n,z/\sqrt n).
\end{equation}
For $s=h/n$ and $\alpha=a/n$, the density
\eqref{eq:type-A-endpoint-Gaussian-density} satisfies
\begin{equation}
\label{eq:type-A-scaled-density}
\frac1n p_{c\sqrt n;a,h-a}(x/\sqrt n,z/\sqrt n)
=\frac{\sqrt{\alpha(s-\alpha)+c^2s}}{2\pi}
 \eee^{-\frac12[(s-\alpha)x^2+\alpha z^2+c^2(x-z)^2]}.
\end{equation}
Using \eqref{eq:type-A-scaled-cyclic-factor-asymptotics} in
\eqref{eq:type-A-symmetrized-kernel}, the sum defining
$R_{n,c}$ becomes a Riemann sum in $a/n$ with mesh $1/n$.
It follows that, for every $0<\varepsilon<1$ and $M>0$,
\begin{equation}
\label{eq:type-A-scaled-kernel-convergence}
\sup_{\substack{\varepsilon n\le h\le n\\|x|,|z|\le M}}
\left|R_{n,c}(h;x,z)-R_c(h/n;x,z)\right|\longrightarrow0.
\end{equation}
Indeed, the limiting summand is the integrand in
\eqref{eq:type-A-limiting-kernel}, and its denominator is at least
$c\sqrt\varepsilon$ on this range. The limiting integrand is uniformly continuous on this compact
parameter set, including at $\alpha=0$ and $\alpha=s$, so the Riemann sums converge uniformly.
The term $a=h$, corresponding to $b=0$, is included in these sums.

\medskip
\emph{Step 3: Control small time coordinates and the integrals over $x,z$.}
We retain the Gaussian decay in each distribution-function
difference, so that a single integrable bound controls the whole
half-plane $z<x$. For $u>0$,
\begin{equation}
\label{eq:type-A-ordinary-factor-spatial-bound}
0\le b_u(x,z)
=\frac1{\sqrt{2\pi u}}\int_z^x\eee^{-u y^2/2}\,\dint y
\le\frac{x-z}{\sqrt{2\pi u}}
 \eee^{-\frac u2\dist(0,[z,x])^2}.
\end{equation}
For $s=h/n\le1$, equations
\eqref{eq:type-A-scaled-cyclic-factor-bound} and
\eqref{eq:type-A-scaled-density} give
\begin{equation}
\label{eq:type-A-kernel-spatial-bound}
R_{n,c}(h;x,z)
\le C_c\sqrt s\,
 \eee^{-\frac12[s\dist(0,[z,x])^2+c^2(x-z)^2]}.
\end{equation}
To see this, use
$\alpha(s-\alpha)+c^2s\le s(c^2+1/4)$,
$(s-\alpha)x^2+\alpha z^2\ge s\dist(0,[z,x])^2$, and
$\sum_{a=1}^h\chi_{a,h-a}^A(c\sqrt n)\le\sqrt{1+1/(4c^2)}$.
The same bound holds for $R_c(s;x,z)$: in its defining integral,
the denominator is at least $c\sqrt s$ and the integration
interval has length $s$.

Set, for $z<x$,
\[
\mathcal M_{c,k}(x,z):=(x-z)^{k-1}
 \eee^{-\frac12[\dist(0,[z,x])^2+c^2(x-z)^2]}.
\]
Combining \eqref{eq:type-A-ordinary-factor-spatial-bound} and
\eqref{eq:type-A-kernel-spatial-bound}, whenever
$s+u_1+\cdots+u_{k-1}=1$, yields
\begin{equation}
\label{eq:type-A-continuum-product-bound}
R_c(s;x,z)\prod_{j=1}^{k-1}b_{u_j}(x,z)
\le C_{c,k}\,s^{1/2}\prod_{j=1}^{k-1}u_j^{-1/2}
 \mathcal M_{c,k}(x,z).
\end{equation}
The identical bound holds with $R_c(s;x,z)$ replaced by
$R_{n,c}(h;x,z)$ when $s=h/n$ and $u_j=d_j/n$.
Moreover,
\begin{equation}
\label{eq:type-A-spatial-majorant-integrable}
\int_{z<x}\mathcal M_{c,k}(x,z)\,\dint x\,\dint z
=\int_0^\infty w^{k-1}(w+\sqrt{2\pi})
 \eee^{-c^2w^2/2}\,\dint w<\infty.
\end{equation}
Here $w=x-z$; integration in $z$ over $(-\infty,-w)$,
$[-w,0]$, and $(0,\infty)$ gives the factor $w+\sqrt{2\pi}$.

Apply \eqref{eq:finite-BA-continuous-simplex-integral} with
parameter $3/2$ for $s$ and parameter $1/2$ for each $u_j$.
Together with \eqref{eq:type-A-continuum-product-bound}, this
proves finiteness of \eqref{eq:type-A-unit-time-continuum-integral}.
For $k\ge2$, its portions with $s\le\delta$ or a specified
$u_j\le\delta$ are respectively $O_{c,k}(\delta^{3/2})$ and
$O_{c,k}(\sqrt\delta)$.

\medskip
\emph{Step 4: Pass to the limit for type $A$.}
The substitutions $x\mapsto x/\sqrt n$ and $z\mapsto z/\sqrt n$
in \eqref{eq:type-A-intrinsic-volumes-symmetrized} give
\begin{equation}
\label{eq:type-A-unit-time-scaled-sum}
\upsilon_k\bigl(\bW_n^A(c\sqrt n)\bigr)
=\frac1{(k-1)!}\int_{z<x}\frac1{n^{k-1}}
\sum_{\substack{h,d_1,\ldots,d_{k-1}\ge1\\
 h+d_1+\cdots+d_{k-1}=n}}
R_{n,c}(h;x,z)\prod_{j=1}^{k-1}b_{d_j/n}(x,z)
\,\dint x\,\dint z.
\end{equation}
The spatial Jacobian $1/n$ is absorbed by
\eqref{eq:type-A-scaled-kernel-definition}; each of the $k-1$
remaining factors contributes $1/n$.

Use the discrete bound from Step~3 and
\eqref{eq:finite-BA-discrete-simplex-bound}, now with $p=k$,
parameter $3/2$ for $h$ and $1/2$ for each $d_j$.
Since $h^{1/2}\le(h+1)^{1/2}$ and
$d^{-1/2}\le\sqrt2(d+1)^{-1/2}$ for $d\ge1$, the integrand
after summation in \eqref{eq:type-A-unit-time-scaled-sum} is
bounded by $C_{c,k}\mathcal M_{c,k}(x,z)$.
For $k\ge2$, equation
\eqref{eq:finite-BA-discrete-simplex-boundary} gives, after spatial
integration, bounds $C_{c,k}(\delta+1/n)^{3/2}$ for $h\le\delta n$
and $C_{c,k}\sqrt{\delta+1/n}$ for any specified $d_j\le\delta n$.

Restrict first to $h,d_j>\delta n$.
For each fixed $(x,z)$ with $z<x$,
\eqref{eq:type-A-scaled-kernel-convergence} and continuity of
$b_u(x,z)$ for $u>0$ give the $(k-1)$-dimensional Riemann-sum
limit on this restricted domain. Equation
\eqref{eq:type-A-spatial-majorant-integrable} permits integration
in $x,z$ by dominated convergence. Letting first $n\to\infty$
and then $\delta\downarrow0$, and using the discrete and
continuous bounds above, proves
\eqref{eq:finite-BA-unit-time-continuum-limit} for type $A$.
When $k=1$, $h=n$ and there are no gap variables: the claim follows
directly from \eqref{eq:type-A-scaled-kernel-convergence} and
the same integrable bound, with no boundary truncation.

\medskip
\emph{Step 5: Verify continuity and the range of the limits.}
All the factors in the defining integrals are continuous
in $c>0$ at interior points of their integration domains.
For $c$ in a compact interval $[c_-,c_+]\subset(0,\infty)$,
the bounds in Steps~1 and~3 hold with a common constant and with
$\varphi(c_-x)$ and $\mathcal M_{c_-,k}(x,z)$ in place of their
$c$-dependent counterparts. The accompanying time factors are
integrable by \eqref{eq:finite-BA-continuous-simplex-integral}.
Dominated convergence therefore proves continuity of $I_k^B$
and $I_k^A$. Finally, each integral is the limit of conic intrinsic
volumes, which belong to $[0,1]$, proving the last assertion.
\end{proof}

\subsection{Evaluation of the multidimensional integrals}
\label{subsec:finite-BA-continuum-integrals}

We now evaluate the integrals in
\eqref{eq:type-B-unit-time-continuum-integral} and
\eqref{eq:type-A-unit-time-continuum-integral}. To this end, we
introduce $T>0$ and take the Laplace transforms of
$T\mapsto I_k^\bullet(T^{-1/2})$, for $\bullet\in\{B,A\}$.
By Proposition~\ref{prop:finite-BA-continuum-limits}, these functions
are continuous and take values in $[0,1]$, so their Laplace
transforms are finite for every positive transform parameter.

For nonnegative functions on $(0,\infty)$, write
$(a*b)(T):=\int_0^T a(t)b(T-t)\,\dint t$, and use $a^{*r}$
for the $r$-fold convolution. Dilating the time coordinates in the
defining integrals \eqref{eq:type-B-unit-time-continuum-integral} and
\eqref{eq:type-A-unit-time-continuum-integral} gives
\begin{align}
I_k^B(T^{-1/2})
&=\frac1{(k-1)!}\int_0^\infty
\bigl(g_1(\,\cdot\,,x)*f_x^{*(k-1)}*h\bigr)(T)\,\dint x,
\label{eq:type-B-rescaled-continuum-convolution}
\\
I_k^A(T^{-1/2})
&=\frac1{(k-1)!}\int_{z<x}
\bigl(R_1(\,\cdot\,;x,z)*b_{\cdot}(x,z)^{*(k-1)}\bigr)(T)
\,\dint x\,\dint z.
\label{eq:type-A-rescaled-continuum-convolution}
\end{align}
Here $b_{\cdot}(x,z)$ denotes the function $u\mapsto b_u(x,z)$.
For $k=1$, the additional convolution factors are omitted, leaving
$g_1(\,\cdot\,,x)*h$ and $R_1(\,\cdot\,;x,z)$, respectively.
To verify these identities, substitute $c=T^{-1/2}$ in
\eqref{eq:type-B-unit-time-continuum-integral} and
\eqref{eq:type-A-unit-time-continuum-integral}, and rescale the
simplex variables from total time $1$ to total time $T$. Thus, in
type $B$, set $t=s/T$, $u_j=v_j/T$, $x=\sqrt T\,y$, so that the remaining time coordinate becomes $1-t-\sum_j u_j
=(T-s-\sum_jv_j)/T$.
In type $A$, set $u_j=v_j/T$, $x=\sqrt T\,y$, $z=\sqrt T\,w$, so that $1-\sum_j u_j= (T-\sum_jv_j)/T$.
The defining formulas for $g_c,h,f_x,R_c$, and $b_u$, together
with the corresponding Jacobians, then give
\eqref{eq:type-B-rescaled-continuum-convolution} and
\eqref{eq:type-A-rescaled-continuum-convolution}; all powers of
$T$ cancel.

For $\lambda>0$, the elementary transforms needed below are
\begin{align}
\int_0^\infty\eee^{-\lambda t}h(t)\,\dint t
&=\lambda^{-1/2},
\label{eq:type-B-terminal-kernel-transform}
\\
\int_0^\infty\eee^{-\lambda t}g_1(t,x)\,\dint t
&=\frac{\varphi(x)}{\sqrt{\lambda+x^2/2}},
\qquad x\ge0,
\label{eq:type-B-initial-kernel-transform}
\\
\int_0^\infty\eee^{-\lambda t}f_x(t)\,\dint t
&=\arsinh\frac{x}{\sqrt{2\lambda}},
\qquad x\ge0,
\label{eq:type-B-ordinary-kernel-transform}
\\
\int_0^\infty\eee^{-\lambda u}b_u(x,z)\,\dint u
&=\arsinh\frac{x}{\sqrt{2\lambda}}
 -\arsinh\frac{z}{\sqrt{2\lambda}}
 =:H_\lambda(x,z),
\qquad z<x.
\label{eq:type-A-ordinary-kernel-transform}
\end{align}
The first two follow from the gamma integral. For the last,
Tonelli's theorem and the representation
$b_u(x,z)=(2\pi u)^{-1/2}\int_z^x\eee^{-u y^2/2}\,\dint y$
give the integral $\int_z^x(2\lambda+y^2)^{-1/2}\,\dint y$.
Taking $z=0$ also gives \eqref{eq:type-B-ordinary-kernel-transform}.

For type $A$, we need one further identity. Write
$\overline\Phi:=1-\Phi$ and set
\begin{equation}
\label{eq:type-A-kernel-Laplace-transform-definition}
\widehat R_{c,\lambda}(x,z)
:=\int_0^\infty\eee^{-\lambda s}R_c(s;x,z)\,\dint s,
\qquad c,\lambda>0.
\end{equation}

\begin{lemma}
\label{lem:type-A-hyperbolic-integral-identity}
For $c,\lambda>0$ and every nonnegative measurable function
$\psi:(0,\infty)\to[0,\infty]$,
\begin{equation}
\label{eq:type-A-hyperbolic-integral-identity}
\int_{z<x}\psi\bigl(H_\lambda(x,z)\bigr)
 \widehat R_{c,\lambda}(x,z)\,\dint x\,\dint z
\\=c^2\int_0^\infty\psi(w)
 \left(\int_w^\infty
 \eee^{-4c^2\lambda\sinh^2(t/2)}\,\dint t\right)\dint w.
\end{equation}
Both sides are understood as nonnegative integrals and may be infinite.
\end{lemma}

\begin{proof}
\emph{Step 1: Evaluate the transform of $R_c$.}
For fixed $x,z\in\R$, put
$p:=\sqrt{2\lambda+x^2}$ and $q:=\sqrt{2\lambda+z^2}$.
The substitutions $a=\alpha$ and $b=s-\alpha$ in
\eqref{eq:type-A-limiting-kernel} give
\begin{equation}
\label{eq:type-A-kernel-transform-double-integral}
\widehat R_{c,\lambda}(x,z)
=\frac{c^2}{2\pi}\eee^{-c^2(x-z)^2/2}
\int_0^\infty\int_0^\infty
\frac{\eee^{-p^2b/2-q^2a/2}}{\sqrt{ab+c^2(a+b)}}
\,\dint a\,\dint b.
\end{equation}
We separate the two integrations using the Gaussian identity
\[
\frac1{\sqrt{ab+c^2(a+b)}}
=\frac1{c\sqrt{\pi ab}}\int_0^\infty
\eee^{-r^2/(4c^2)-r^2/(4a)-r^2/(4b)}\,\dint r
\]
and the elementary formula
\begin{equation}
\label{eq:finite-BA-Gaussian-integral-with-reciprocal-term}
\int_0^\infty t^{-1/2}\eee^{-At-r^2/(4t)}\,\dint t
=\sqrt{\frac\pi A}\,\eee^{-r\sqrt A},
\qquad A>0,\quad r\ge0.
\end{equation}
For a direct verification of the latter, denote its left-hand side
by $F_A(r)$. Differentiating for $r>0$ and substituting
$t=r^2/(4Au)$ in the resulting integral gives
$F_A'(r)=-\sqrt A\,F_A(r)$. Dominated convergence at zero gives
$F_A(0)=\sqrt{\pi/A}$, proving the formula.

Using these identities and Tonelli's theorem in
\eqref{eq:type-A-kernel-transform-double-integral} gives
\[
\widehat R_{c,\lambda}(x,z)
=\frac{c\,\eee^{-c^2(x-z)^2/2}}{\sqrt\pi\,pq}
\int_0^\infty
\eee^{-r^2/(4c^2)-r(p+q)/\sqrt2}\,\dint r.
\]
Completing the square yields the explicit expression
\begin{equation}
\label{eq:type-A-kernel-transform-normal-tail}
\widehat R_{c,\lambda}(x,z)
=\frac{2c^2}{pq}
\eee^{\frac{c^2}{2}[(p+q)^2-(x-z)^2]}
\overline\Phi\bigl(c(p+q)\bigr).
\end{equation}
In particular, this transform is finite for every $x,z\in\R$.

\medskip
\emph{Step 2: Use $H_\lambda$ as an integration variable.}
The change of variables
\begin{equation}
\label{eq:type-A-hyperbolic-change-of-variables}
x=\sqrt{2\lambda}\sinh(v+w/2),\qquad
z=\sqrt{2\lambda}\sinh(v-w/2),
\qquad v\in\R,\quad w>0,
\end{equation}
maps $\R\times(0,\infty)$ bijectively onto $\{z<x\}$ and gives
$H_\lambda(x,z)=w$. Moreover,
\[
p+q=2\sqrt{2\lambda}\cosh v\cosh(w/2),\qquad
x-z=2\sqrt{2\lambda}\cosh v\sinh(w/2),
\]
and the absolute Jacobian is $pq$.
Put $\eta:=2c\sqrt{2\lambda}$. Substitution into
\eqref{eq:type-A-kernel-transform-normal-tail} shows that the
left-hand side of \eqref{eq:type-A-hyperbolic-integral-identity} is
\begin{equation}
\label{eq:type-A-transformed-normal-tail-integral}
2c^2\int_0^\infty\psi(w)\int_{\R}
\eee^{\eta^2\cosh^2v/2}
\overline\Phi\bigl(\eta\cosh v\cosh(w/2)\bigr)
\,\dint v\,\dint w.
\end{equation}

\medskip
\emph{Step 3: Integrate the normal tail.}
For $w>0$, substituting $y=\eta\cosh v\cosh(t/2)$ in the
normal-tail integral and then using Tonelli's theorem gives
\begin{align*}
&2\int_{\R}\eee^{\eta^2\cosh^2v/2}
 \overline\Phi\bigl(\eta\cosh v\cosh(w/2)\bigr)\,\dint v
\\
&\quad\quad\quad=\frac{\eta}{\sqrt{2\pi}}\int_w^\infty\sinh(t/2)
 \int_{\R}\cosh v\,
 \eee^{-(\eta^2/2)\sinh^2(t/2)\cosh^2v}
 \,\dint v\,\dint t
\\
&\quad\quad\quad=\int_w^\infty\eee^{-(\eta^2/2)\sinh^2(t/2)}\,\dint t.
\end{align*}
For the last equality, put $r=\sinh v$ in the inner integral;
its value is
$\sqrt{2\pi}\,\eee^{-(\eta^2/2)\sinh^2(t/2)}/(\eta\sinh(t/2))$.
Substituting this result into
\eqref{eq:type-A-transformed-normal-tail-integral} and using
$\eta^2/2=4c^2\lambda$ proves
\eqref{eq:type-A-hyperbolic-integral-identity}.
All rearrangements above involve nonnegative integrands.
\end{proof}

We can now evaluate both integrals explicitly.

\begin{proposition}
\label{prop:finite-BA-continuum-integrals-evaluation}
For every $c>0$ and integer $k\ge1$,
\begin{equation}
\label{eq:finite-BA-continuum-integrals-evaluation}
I_k^B(c)=L_k^B(c),\qquad I_k^A(c)=L_k^A(c),
\end{equation}
where the right-hand sides are given by
\eqref{eq:type-B-limiting-intrinsic-volumes} and
\eqref{eq:type-A-limiting-intrinsic-volumes}, respectively.
\end{proposition}

\begin{proof}
\emph{Type $B$.}
Taking the Laplace transform in
\eqref{eq:type-B-rescaled-continuum-convolution} and multiplying
\eqref{eq:type-B-terminal-kernel-transform}--\eqref{eq:type-B-ordinary-kernel-transform}
gives
\begin{align}
\int_0^\infty\eee^{-\lambda T}I_k^B(T^{-1/2})\,\dint T
&=\frac1{(k-1)!\sqrt\lambda}\int_0^\infty
 \frac{\varphi(x)}{\sqrt{\lambda+x^2/2}}
 \left(\arsinh\frac{x}{\sqrt{2\lambda}}\right)^{k-1}\,\dint x
\notag\\
&=\frac1{\sqrt{\pi\lambda}(k-1)!}\int_0^\infty
 u^{k-1}\eee^{-\lambda\sinh^2u}\,\dint u.
\label{eq:type-B-unit-height-Laplace-transform}
\end{align}
The last equality uses $x=\sqrt{2\lambda}\sinh u$.
All exchanges are justified by Tonelli's theorem.
The gamma-integral identity
\[
\int_a^\infty\frac{\eee^{-\lambda T}}{\sqrt{\pi(T-a)}}\,\dint T
=\frac{\eee^{-\lambda a}}{\sqrt\lambda},
\qquad a\ge0,
\]
shows that the last expression in
\eqref{eq:type-B-unit-height-Laplace-transform} is the Laplace
transform of
\begin{equation}
\label{eq:type-B-unit-height-inverse-transform}
\frac1{\pi(k-1)!}\int_0^{\arsinh\sqrt T}
\frac{u^{k-1}}{\sqrt{T-\sinh^2u}}\,\dint u.
\end{equation}
This function is continuous for $T>0$: the substitution
$v=\sinh u/\sqrt T$ rewrites it as
\[
\frac1{\pi(k-1)!}\int_0^1
\frac{(\arsinh(\sqrt T\,v))^{k-1}}
{\sqrt{1-v^2}\sqrt{1+Tv^2}}\,\dint v,
\]
whose integrand has an integrable majorant locally uniformly in $T$.
Uniqueness of Laplace transforms and continuity therefore identify
\eqref{eq:type-B-unit-height-inverse-transform} with
$I_k^B(T^{-1/2})$ for every $T>0$. Setting $T=c^{-2}$ gives
$I_k^B(c)=L_k^B(c)$.

\medskip
\emph{Type $A$.}
Taking the Laplace transform in
\eqref{eq:type-A-rescaled-continuum-convolution}, using
\eqref{eq:type-A-ordinary-kernel-transform}, and applying
Lemma~\ref{lem:type-A-hyperbolic-integral-identity} with $c=1$
and $\psi(w)=w^{k-1}$ gives
\begin{align}
\int_0^\infty\eee^{-\lambda T}I_k^A(T^{-1/2})\,\dint T
&=\frac1{(k-1)!}\int_{z<x}
 H_\lambda(x,z)^{k-1}\widehat R_{1,\lambda}(x,z)
 \,\dint x\,\dint z
\\
&=\frac1{(k-1)!}\int_0^\infty w^{k-1}
 \int_w^\infty\eee^{-4\lambda\sinh^2(t/2)}\,\dint t\,\dint w
\\
&=\frac1{k!}\int_0^\infty
 t^k\eee^{-4\lambda\sinh^2(t/2)}\,\dint t.
\label{eq:type-A-unit-height-Laplace-transform}
\end{align}
Again, all exchanges follow from Tonelli's theorem, and the final
integral is finite. To invert it, set $T=4\sinh^2(t/2)$,
so that $\dint T=2\sinh t\,\dint t$. The last integral in
\eqref{eq:type-A-unit-height-Laplace-transform} is thus the
Laplace transform of the continuous function
\begin{equation}
\label{eq:type-A-unit-height-inverse-transform}
\frac{\bigl(2\arsinh(\sqrt T/2)\bigr)^k}
{2k!\,\sinh\bigl(2\arsinh(\sqrt T/2)\bigr)},
\qquad T>0.
\end{equation}
Uniqueness and continuity identify this expression with
$I_k^A(T^{-1/2})$ for every $T>0$. Setting $T=c^{-2}$ gives
$I_k^A(c)=\theta_c^k/(2k!\sinh\theta_c)=L_k^A(c)$.
The conventions for empty convolution factors include $k=1$
in both calculations.
\end{proof}

Together with \eqref{eq:finite-BA-unit-time-continuum-limit}, this
proves \eqref{eq:finite-BA-intrinsic-volume-scaling-limit} for
$k\ge1$ when $\rho_n=c\sqrt n$.

\subsection{The zeroth intrinsic volumes and completion of the proof}
\label{subsec:finite-BA-scaling-completion}

Propositions~\ref{prop:finite-BA-continuum-limits} and
\ref{prop:finite-BA-continuum-integrals-evaluation} establish
\eqref{eq:finite-BA-intrinsic-volume-scaling-limit} for $k\ge1$
when $\rho_n=c\sqrt n$. We now treat $k=0$ and prove the remaining
assertions of Theorem~\ref{thm:finite-BA-scaling-limit}.

\begin{proof}[Completion of the proof of
Theorem~\ref{thm:finite-BA-scaling-limit}]
Put $c_n:=\rho_n/\sqrt n$, so that $c_n\to c \in (0,\infty)$.

\smallskip
\emph{Step 1: Determine the zeroth intrinsic volumes.}
Let $W$ be a one-dimensional standard Brownian motion, set
$W_t^B:=W_t$ and $W_t^A:=W_t-tW_1$ for $0\le t\le1$, and let
$Z_0$ be standard normal and independent of $W$. Put
$M_n^\bullet:=\max_{0\le j\le n}W_{j/n}^\bullet$ and
$M^\bullet:=\max_{0\le t\le1}W_t^\bullet$, for
$\bullet\in\{B,A\}$. The Gaussian vectors $(G_j/\sqrt n)_{j=0}^n$
and $(W_{j/n})_{j=0}^n$ have the same law; centering gives the
corresponding equality in law for the bridges. Hence
\eqref{eq:type-B-zeroth-intrinsic-volume} and
\eqref{eq:type-A-zeroth-intrinsic-volume} imply
\begin{equation}
\label{eq:finite-BA-zeroth-grid-limit}
\upsilon_0\bigl(\bW_n^\bullet(\rho_n)\bigr)
=\Ee\left[\Phi\left(-\frac{M_n^\bullet}{c_n}\right)\right]
\longrightarrow
\Ee\left[\Phi\left(-\frac{M^\bullet}{c}\right)\right]
=\Pp(cZ_0+M^\bullet\le0).
\end{equation}
Including the index $n$ does not change the maximum in type $A$,
since $W_1^A=W_0^A=0$. The convergence follows from path continuity and then by dominated (bounded) convergence, since \(0\le \Phi\le1\).

By the reflection principle, $M^B$ has the law of $|Z_1|$, where
$Z_1$ is standard normal and independent of $Z_0$. The same
principle, applied to the endpoint density of Brownian motion,
gives $\Pp(M^A\le x)=1-\eee^{-2x^2}$ for $x\ge0$: for $x>0$,
the reflected density at endpoint zero is $\varphi(2x)$, and its
ratio to $\varphi(0)$ is $\eee^{-2x^2}$. Consequently,
\begin{align}
\Pp(cZ_0+M^B\le0)
&=\Pp(cZ_0+|Z_1|\le0)
=\frac{\arctan c}{\pi}=L_0^B(c),
\label{eq:type-B-zeroth-limit-evaluation}
\\
\Pp(cZ_0+M^A\le0)
&=\int_0^\infty\varphi(x)\bigl(1-\eee^{-2c^2x^2}\bigr)\,\dint x
=\frac12\left(1-\frac1{\sqrt{1+4c^2}}\right)=L_0^A(c).
\label{eq:type-A-zeroth-limit-evaluation}
\end{align}
For the first evaluation, the event in the $(Z_0,Z_1)$-plane is
a sector of angle $2\arctan c$; the standard Gaussian angular
coordinate is uniform. This proves
\eqref{eq:finite-BA-intrinsic-volume-scaling-limit} for $k=0$.

\medskip
\emph{Step 2: Verify \eqref{eq:finite-BA-limiting-parity-normalization}.}
For type $B$, Tonelli's theorem and
\eqref{eq:type-B-limiting-intrinsic-volumes} give
\begin{align}
\sum_{j=0}^\infty L_{2j+1}^B(c)
&=\frac1\pi\int_0^{\arsinh(1/c)}
\frac{\cosh u}{\sqrt{c^{-2}-\sinh^2u}}\,\dint u
=\frac12,
\label{eq:type-B-limiting-odd-mass}
\\
\sum_{j=1}^\infty L_{2j}^B(c)
&=\frac1\pi\int_0^{\arsinh(1/c)}
\frac{\sinh u}{\sqrt{c^{-2}-\sinh^2u}}\,\dint u
=\frac12-L_0^B(c).
\label{eq:type-B-limiting-positive-even-mass}
\end{align}
The substitutions $v=c\sinh u$ and $v=\cosh u$, respectively,
give the last equalities. For type $A$, recall
$\theta_c=2\arsinh(1/(2c))$ and note that
$\tanh(\theta_c/2)=(1+4c^2)^{-1/2}$. Thus, \eqref{eq:type-A-limiting-intrinsic-volumes} gives
\begin{align*}
\sum_{j=0}^\infty L_{2j+1}^A(c)
&=\frac{\sinh\theta_c}{2\sinh\theta_c}=\frac12,
\\
L_0^A(c)+\sum_{j=1}^\infty L_{2j}^A(c)
&=L_0^A(c)+\frac{\cosh\theta_c-1}{2\sinh\theta_c}=\frac12.
\end{align*}
This proves \eqref{eq:finite-BA-limiting-parity-normalization} for
both types, and in particular $\sum_{k=0}^\infty L_k^\bullet(c)=1$.

\medskip
\emph{Step 3: Allow varying $c_n= \rho_n/\sqrt n$.} The case \(k=0\) was already proved in Step~1 for arbitrary \(c_n\to c\). We now extend the convergence for \(k\ge1\), established above for fixed normalized height, to the general case \(c_n\to c\).
Fix $\bullet\in\{B,A\}$. We use monotonicity of Grassmann angles
with respect to cone inclusion. Since $\bK_n^\bullet$ is convex
and contains the origin, for $0<c_1\le c_2$ we have
\begin{equation}
\label{eq:finite-BA-height-cone-monotonicity}
\bW_n^\bullet(c_2\sqrt n)\subseteq\bW_n^\bullet(c_1\sqrt n).
\end{equation}
Hence
$c'\mapsto\gamma_j(\bW_n^\bullet(c'\sqrt n))$ is nonincreasing
for every integer $j\ge0$.

The cones are nonzero and pointed, so
\eqref{eq:finite-dimensional-conic-crofton} and
\eqref{eq:prelim-conic-gauss-bonnet} give a finite-sum expression
for each Grassmann angle. For every fixed $c'>0$ and integer
$j\ge0$, the convergence already proved for fixed normalized
height, including the zeroth coefficient in Step~1, therefore gives
\begin{equation}
\label{eq:finite-BA-fixed-height-grassmann-limit}
\gamma_j\bigl(\bW_n^\bullet(c'\sqrt n)\bigr)
=1-2\sum_{\ell=0}^{\lfloor(j-1)/2\rfloor}
\upsilon_{j-1-2\ell}\bigl(\bW_n^\bullet(c'\sqrt n)\bigr)
\longrightarrow
1-2\sum_{\ell=0}^{\lfloor(j-1)/2\rfloor}
L_{j-1-2\ell}^\bullet(c')
=: \Gamma_j^\bullet(c').
\end{equation}
For $j=0$, both sums are empty and $\Gamma_0^\bullet=1$.
Each $\Gamma_j^\bullet$ is continuous on $(0,\infty)$, since
the functions $L_r^\bullet$ are continuous and the sum is finite.

Choose $0<c_-<c<c_+$. Since $c_n\to c$, for all sufficiently
large $n$ monotonicity gives
\[
\gamma_j\bigl(\bW_n^\bullet(c_+\sqrt n)\bigr)
\le\gamma_j\bigl(\bW_n^\bullet(c_n\sqrt n)\bigr)
\le\gamma_j\bigl(\bW_n^\bullet(c_-\sqrt n)\bigr).
\]
Using \eqref{eq:finite-BA-fixed-height-grassmann-limit} yields
\[
\Gamma_j^\bullet(c_+)
\le\liminf_{n\to\infty}
\gamma_j\bigl(\bW_n^\bullet(c_n\sqrt n)\bigr)
\le\limsup_{n\to\infty}
\gamma_j\bigl(\bW_n^\bullet(c_n\sqrt n)\bigr)
\le\Gamma_j^\bullet(c_-).
\]
Letting $c_-\uparrow c$ and $c_+\downarrow c$, continuity gives
\begin{equation}
\label{eq:finite-BA-varying-height-grassmann-limit}
\gamma_j\bigl(\bW_n^\bullet(c_n\sqrt n)\bigr)
\longrightarrow\Gamma_j^\bullet(c),\qquad j\ge0.
\end{equation}
Finally, for every fixed $k\ge1$ and all sufficiently large $n$,
\eqref{eq:prelim-intrinsic-volumes-from-grassmann-angles} and
$\rho_n=c_n\sqrt n$ give
\[
\upsilon_k\bigl(\bW_n^\bullet(\rho_n)\bigr)
=\frac12\left[
\gamma_{k-1}\bigl(\bW_n^\bullet(c_n\sqrt n)\bigr)
-\gamma_{k+1}\bigl(\bW_n^\bullet(c_n\sqrt n)\bigr)
\right]
\longrightarrow
\frac12\bigl(\Gamma_{k-1}^\bullet(c)-\Gamma_{k+1}^\bullet(c)\bigr)
=L_k^\bullet(c),
\]
where the last equality follows directly from the finite-sum
formula defining $\Gamma_j^\bullet$. This proves
\eqref{eq:finite-BA-intrinsic-volume-scaling-limit} for every
fixed $k\ge1$ in the stated regime.

\medskip
\emph{Step 4: Prove convergence in $\ell^1$.}
For either type, the intrinsic volumes
$\upsilon_k(\bW_n^\bullet(\rho_n))$, $k\ge0$, are nonnegative
and sum to one. By
\eqref{eq:finite-BA-limiting-parity-normalization}, the limiting
coefficients $L_k^\bullet(c)$ are likewise nonnegative and sum to
one. Since
\eqref{eq:finite-BA-intrinsic-volume-scaling-limit} gives
pointwise convergence for every $k\ge0$, Scheff\'e's lemma,
applied to counting measure on $\N_0$, yields
$
\sum_{k=0}^\infty
\left|
\upsilon_k\bigl(\bW_n^\bullet(\rho_n)\bigr)
-L_k^\bullet(c)
\right|
\to 0
$.
This proves
\eqref{eq:finite-BA-intrinsic-volume-l1-convergence} and completes
the proof.
\end{proof}

\section{Absorption probabilities for Brownian and Ornstein--Uhlenbeck processes}
\label{sec:brownian-ou-absorption}
We first compute the probability that the convex hull of a Brownian
motion or a Brownian bridge in $\R^d$ contains an independent Gaussian
point, proving Theorem~\ref{thm:intro-brownian-gaussian-absorption}.
We also derive the large-scale consequences stated in
Corollary~\ref{cor:intro-brownian-gaussian-large-scale-and-volumes}.
We then use deterministic rescalings and time changes to obtain
the origin-absorption probabilities for stationary
Ornstein--Uhlenbeck paths and circular Ornstein--Uhlenbeck loops
from Theorem~\ref{thm:intro-ou-origin-absorption}.
Next, we establish the long-time and high-dimensional asymptotic
results stated in Theorems~\ref{thm:intro-ou-nonabsorption-fixed-d}
and~\ref{thm:intro-ou-high-dimensional-transition} and
Corollary~\ref{cor:intro-ou-linear-time-threshold}.
Finally, we determine the expected conic intrinsic volumes of
the random OU cones, providing one of the ingredients of
Proposition~\ref{prop:intro-random-ou-expected-intrinsic-volumes}.

\subsection{Gaussian-point absorption for Brownian convex hulls}
\label{subsec:brownian-gaussian-absorption}

Fix an integer $d\ge1$. Let $(W_t)_{0\le t\le1}$ be a standard
$d$-dimensional Brownian motion starting at $0$, and let
$W_t^{\mathrm{br}}:=W_t-tW_1$, where  $0\leq t \leq 1$,  be the associated standard Brownian
bridge from $0$ to $0$. Define the Brownian convex hulls
\begin{align}
\mathcal K_d^B
&:=\conv\{W_t:0\le t\le1\},
\label{eq:type-B-brownian-hull}
\\
\mathcal K_d^A
&:=\conv\{W_t^{\mathrm{br}}:0\le t\le1\}.
\label{eq:type-A-brownian-hull}
\end{align}
These are almost surely compact convex subsets of $\R^d$.
Let $Z\sim N(0,I_d)$ be independent of the Brownian motion $(W_t)_{0\le t\le1}$, and put
\begin{equation}
\label{eq:brownian-gaussian-absorption-functions}
q_d^\bullet(\sigma)
:=\Pp(\sigma Z\in\mathcal K_d^\bullet),
\qquad \sigma\ge0,\quad \bullet\in\{B,A\}.
\end{equation}
In particular, $q_d^B(0)=q_d^A(0)=1$, since both paths start at $0$.

The next theorem gives a formula for $q_d^\bullet(\sigma)$. To state it, for every integer $m\ge0$, we define the entire function
\begin{equation}
\label{eq:brownian-hyperbolic-remainders}
\Psi_m(x):=\sum_{j=0}^\infty\frac{x^{m+2j}}{(m+2j)!},
\qquad x\in\mathbb C.
\end{equation}
For $m\ge0$, these functions can be written as
\begin{align*}
\Psi_{2m}(x)
&=\cosh x-\sum_{j=0}^{m-1}\frac{x^{2j}}{(2j)!},
\\
\Psi_{2m+1}(x)
&=\sinh x-\sum_{j=0}^{m-1}\frac{x^{2j+1}}{(2j+1)!}.
\end{align*}
In particular, $\Psi_0(x)=\cosh x$
and $\Psi_1(x)=\sinh x$. Each $\Psi_m$ is nonnegative on $[0,\infty)$.

\begin{theorem}[Gaussian-point absorption by Brownian convex hulls]
\label{thm:brownian-gaussian-absorption}
For every integer $d\ge1$ and every $\sigma>0$, the probabilities $q_d^\bullet(\sigma)
=\Pp(\sigma Z\in\mathcal K_d^\bullet)$ are given by
\begin{align}
q_d^B(\sigma)
&=2\sum_{j=0}^\infty L_{d+1+2j}^B(\sigma)
=\frac2\pi\int_0^{\arsinh(1/\sigma)}
\frac{\Psi_d(x)\,\dint x}{\sqrt{\sigma^{-2}-\sinh^2x}},
\label{eq:type-B-brownian-gaussian-absorption}
\\
q_d^A(\sigma)
&=2\sum_{j=0}^\infty L_{d+1+2j}^A(\sigma)
=\frac{\Psi_{d+1}(\theta_\sigma)}{\sinh\theta_\sigma},
\label{eq:type-A-brownian-gaussian-absorption}
\end{align}
where $\theta_\sigma:=2\arsinh(1/(2\sigma))$ and $L_k^B(\sigma),L_k^A(\sigma)$
are the limits of the conic intrinsic volumes  from Theorem~\ref{thm:finite-BA-scaling-limit}.
Both series and the integral converge. Moreover,
$\Pp(\sigma Z\in\bd\mathcal K_d^\bullet)=0$ for
$\bullet\in\{B,A\}$, so the absorption probabilities are unchanged
if the hulls are replaced by their interiors.
\end{theorem}

The proof is given in Subsections~\ref{subsec:finite-gaussian-hulls-absorption-formula} and~\ref{subsec:finite-gaussian-hulls-absorption-proof-short}. In the remainder of the present subsection, we discuss some corollaries of Theorem~\ref{thm:brownian-gaussian-absorption}.

\begin{corollary}[Complementary formulas and low dimensions]
\label{cor:brownian-gaussian-absorption-consequences}
For every $d\ge1$, $\sigma>0$, and $\bullet\in\{B,A\}$,
\begin{align}
1-q_d^\bullet(\sigma)
&=2\sum_{j=0}^{\lfloor(d-1)/2\rfloor}
L_{d-1-2j}^\bullet(\sigma),
\label{eq:brownian-gaussian-nonabsorption-finite-sum}
\\
q_{d+2}^\bullet(\sigma)
&=q_d^\bullet(\sigma)-2L_{d+1}^\bullet(\sigma).
\label{eq:brownian-gaussian-absorption-dimension-recurrence}
\end{align}
For $d=1,2$ the Gaussian-point absorption probabilities are given by
\begin{align}
q_1^B(\sigma)
&=\frac2\pi\arctan\frac1\sigma,
\label{eq:type-B-gaussian-absorption-dimension-one}
\\
q_2^B(\sigma)
&=1-\frac2\pi\int_0^{\pi/2}
\frac{\dint\theta}{\sqrt{1+\sigma^{-2}\sin^2\theta}},
\label{eq:type-B-gaussian-absorption-dimension-two}
\\
q_1^A(\sigma)
&=\frac1{\sqrt{1+4\sigma^2}},
\label{eq:type-A-gaussian-absorption-dimension-one}
\\
q_2^A(\sigma)
&=1-\frac{\theta_\sigma}{\sinh\theta_\sigma}.
\label{eq:type-A-gaussian-absorption-dimension-two}
\end{align}
\end{corollary}

\begin{proof}
The indices $d+1,d+3,\ldots$ in the series~\eqref{eq:type-B-brownian-gaussian-absorption} and~\eqref{eq:type-A-brownian-gaussian-absorption}  have the
same parity as $d-1$. Their complement among the nonnegative
indices of this parity is $d-1,d-3,\ldots$, down to $0$ or $1$.
Since the sum over all indices of either parity is $1/2$ by
\eqref{eq:finite-BA-limiting-parity-normalization}, subtracting
the series~\eqref{eq:type-B-brownian-gaussian-absorption} and~\eqref{eq:type-A-brownian-gaussian-absorption}  from $1$ gives
\eqref{eq:brownian-gaussian-nonabsorption-finite-sum}.
Removing the first term of these series gives
\eqref{eq:brownian-gaussian-absorption-dimension-recurrence}.

For $d=1$ and $d=2$, the formula~\eqref{eq:brownian-gaussian-nonabsorption-finite-sum} reads
$q_1^\bullet(\sigma)=1-2L_0^\bullet(\sigma)$ and
$q_2^\bullet(\sigma)=1-2L_1^\bullet(\sigma)$, respectively.
Insert the explicit values  from
Theorem~\ref{thm:finite-BA-scaling-limit}.
For $q_1^B$, use
$\arctan\sigma+\arctan(1/\sigma)=\pi/2$; for $q_2^B$, use the
substitution $\sigma\sinh x=\sin\theta$ in the integral for
$L_1^B(\sigma)$. The type-$A$ formulas follow directly.
\end{proof}

In the next corollary, we let $\sigma\to\infty$ in Theorem~\ref{thm:brownian-gaussian-absorption}, which allows us to determine the
expected volumes of the Brownian convex hulls $\mathcal K_d^\bullet$. Write $\operatorname{vol}_d$ for $d$-dimensional
Lebesgue measure and $\kappa_m:=\pi^{m/2}/\Gamma(m/2+1)$ for the volume
of the Euclidean unit ball in $\R^m$, for $m\ge1$.

\begin{corollary}[Large-$\sigma$ limit and expected hull volumes]
\label{cor:brownian-gaussian-large-scale-and-volumes}
For every fixed integer $d\ge1$, as $\sigma\to\infty$,
\begin{align}
q_d^B(\sigma)
&\sim\frac{\kappa_d^2}{(2\pi)^d}\,\sigma^{-d},
\label{eq:type-B-gaussian-absorption-large-scale}
\\
q_d^A(\sigma)
&\sim\frac{1}{(d+1)!}\,\sigma^{-d}.
\label{eq:type-A-gaussian-absorption-large-scale}
\end{align}
Moreover, for $\bullet\in\{B,A\}$,
\begin{equation}
\label{eq:brownian-gaussian-large-scale-volume-limit}
\lim_{\sigma\to\infty}\sigma^d q_d^\bullet(\sigma)
=\frac{\Ee\,\operatorname{vol}_d(\mathcal K_d^\bullet)}{(2\pi)^{d/2}}.
\end{equation}
Consequently,
\begin{align}
\Ee\,\operatorname{vol}_d(\mathcal K_d^B)
&=\frac{\kappa_d^2}{(2\pi)^{d/2}},
\label{eq:type-B-brownian-hull-expected-volume}
\\
\Ee\,\operatorname{vol}_d(\mathcal K_d^A)
&=\frac{(2\pi)^{d/2}}{(d+1)!}
=\frac{\kappa_d\kappa_{d+1}}{2(2\pi)^{d/2}}.
\label{eq:type-A-brownian-hull-expected-volume}
\end{align}
\end{corollary}

\begin{proof}
For type $B$, the substitution $u=\sigma\sinh x$ in
\eqref{eq:type-B-brownian-gaussian-absorption} gives
\[
\sigma^d q_d^B(\sigma)
=\frac2\pi\int_0^1
\frac{\sigma^d\Psi_d(\arsinh(u/\sigma))}
     {\sqrt{1-u^2}\sqrt{1+u^2/\sigma^2}}\,\dint u.
\]
For each $u\in[0,1]$, the numerator converges to $u^d/d!$ as $\sigma \to\infty$.
For $\sigma\ge1$, the series defining $\Psi_d$ and the inequality
$\arsinh(u/\sigma)\le u/\sigma$ give
\[
0\le\sigma^d\Psi_d(\arsinh(u/\sigma))
\le\frac{u^d}{d!}\cosh 1.
\]
Dominated convergence therefore yields
\begin{align*}
\lim_{\sigma\to\infty}\sigma^d q_d^B(\sigma)
=\frac2{\pi d!}\int_0^1\frac{u^d}{\sqrt{1-u^2}}\,\dint u
=\frac{\Gamma((d+1)/2)}{\sqrt\pi\,d!\,\Gamma(d/2+1)}
=\frac{\kappa_d^2}{(2\pi)^d}.
\end{align*}
The last equality follows from the duplication formula for the gamma
function. This proves \eqref{eq:type-B-gaussian-absorption-large-scale}.
For type $A$, use $\theta_\sigma\sim\sigma^{-1}$,
$\Psi_{d+1}(x)\sim x^{d+1}/(d+1)!$, and $\sinh x\sim x$ as
$x\downarrow0$ in \eqref{eq:type-A-brownian-gaussian-absorption}.
This proves \eqref{eq:type-A-gaussian-absorption-large-scale}.

To identify the limits with expected volumes, condition on the hull $\mathcal K_d^\bullet$.
Independence of $Z$ and $\mathcal K_d^\bullet$
implies, for every $\sigma>0$,
\begin{equation}
\label{eq:brownian-gaussian-absorption-density-integral}
\sigma^d q_d^\bullet(\sigma)
=\frac1{(2\pi)^{d/2}}\,
\Ee\int_{\mathcal K_d^\bullet}
\eee^{-\|x\|^2/(2\sigma^2)}\,\dint x,
\qquad \bullet\in\{B,A\}.
\end{equation}
As $\sigma\uparrow\infty$, the integrand increases to one.
Monotone convergence gives
\eqref{eq:brownian-gaussian-large-scale-volume-limit}, initially
allowing an infinite expected volume. The finite limits of $\sigma^d q_d^\bullet(\sigma)$ computed in~\eqref{eq:type-B-gaussian-absorption-large-scale} and~\eqref{eq:type-A-gaussian-absorption-large-scale}  show that both expected volumes are finite and give
\eqref{eq:type-B-brownian-hull-expected-volume} and the first equality
in \eqref{eq:type-A-brownian-hull-expected-volume}.
The second equality follows again from the gamma duplication formula.
\end{proof}

The expected-volume formulas~\eqref{eq:type-B-brownian-hull-expected-volume} and~\eqref{eq:type-A-brownian-hull-expected-volume} are known: for Brownian motion in
$d\ge2$, see Eldan~\cite[Theorem~1]{Eldan2014}; for both Brownian motion
and Brownian bridges, see~\cite[Section~1.5]{KabluchkoZaporozhets2016}.

\begin{remark}[Arbitrary time horizon]
The unit time horizon in the definitions of $\mathcal K_d^\bullet$ entails no loss of generality.
For a standard Brownian motion $(W_t)_{t\ge0}$ in $\R^d$, define
\[
\mathcal K_{d,T}^B:=\conv\{W_t:0\le t\le T\},
\qquad
\mathcal K_{d,T}^A:=\conv\{W_t-(t/T)W_T:0\le t\le T\}.
\]
Brownian scaling, with $Z\sim N(0,I_d)$ independent of $W$, gives
\begin{equation}
\label{eq:brownian-gaussian-absorption-arbitrary-horizon}
\Pp(\sigma Z\in\mathcal K_{d,T}^\bullet)
=q_d^\bullet\left(\frac\sigma{\sqrt T}\right),
\qquad T>0,\quad\sigma\ge0,\quad\bullet\in\{B,A\}.
\end{equation}
\end{remark}

\subsection{Gaussian random walk and bridge hulls}
\label{subsec:finite-gaussian-hulls-absorption-formula}
We  compute Gaussian-point absorption probabilities for convex hulls of discrete-time Gaussian random walks and bridges. To this end, we apply the conic Crofton formula to the cones  $\bW_n^\bullet(\sigma\sqrt n)$ from Subsection~\ref{subsec:type-B-bodies-and-cones}. In the next subsection  we approximate the Brownian convex hulls by these discrete-time hulls.

Let $X_1,X_2,\ldots$ be independent $N(0,I_d)$ random vectors,
independent of $Z\sim N(0,I_d)$, and put $S_0:=0$ and $S_i:=X_1+\cdots+X_i$. Thus $(S_i)_{i=0}^\infty$ is a random walk in $\R^d$ with standard Gaussian increments.
For $n\ge1$, define
\begin{align}
\mathcal Q_{d,n}^B
&:=\conv\left\{\frac{S_i}{\sqrt n}:i=0,\ldots,n\right\},
\label{eq:type-B-normalized-gaussian-walk-hull}
\\
\mathcal Q_{d,n}^A
&:=\conv\left\{\frac{S_i-(i/n)S_n}{\sqrt n}:
 i=0,\ldots,n-1\right\}.
\label{eq:type-A-normalized-gaussian-bridge-hull}
\end{align}
Including $i=n$ in the second definition does not change the hull,
since the bridge values at $i=0$ and $i=n$ coincide.
The next proposition determines $\Pp(\sigma Z \in \mathcal Q_{d,n}^\bullet)$.

\begin{proposition}[Gaussian-point absorption for convex hulls of Gaussian random walks and bridges]
\label{prop:finite-BA-gaussian-absorption}
For every $n,d\ge1$, $\sigma>0$, and $\bullet\in\{B,A\}$,
\begin{equation}
\label{eq:finite-BA-gaussian-absorption}
\Pp(\sigma Z\in\mathcal Q_{d,n}^\bullet)
=2\sum_{j=0}^\infty
\upsilon_{d+1+2j}\bigl(\bW_n^\bullet(\sigma\sqrt n)\bigr).
\end{equation}
The probabilities vanish for $n<d$ in type $B$ and for $n\le d$
in type $A$. Explicit formulas for the conic intrinsic volumes on the right-hand side are given in Proposition~\ref{prop:finite-BA-intrinsic-volumes-symmetrized}.
\end{proposition}

\begin{proof}
The idea is to  represent the random polytopes $\mathcal Q_{d,n}^B$ and $\mathcal Q_{d,n}^A$ as random Gaussian images of the deterministic bodies $\bK_n^B$ and $\bK_n^A$ defined in Subsection~\ref{subsec:type-B-bodies-and-cones}.
Write $E_n^B:=\R e_0\oplus\R^n$ and
$E_n^A:=\R e_0\oplus s_n^\perp$ for the natural ambient spaces of
the cones $\bW_n^B(\sigma\sqrt n)$ and $\bW_n^A(\sigma\sqrt n)$ defined in Subsection~\ref{subsec:type-B-bodies-and-cones}, and let $D_n^B:=n+1$ and $D_n^A:=n$ be their dimensions.
Define the  Gaussian linear map $\mathcal A_n:E_n^B\to\R^d$ by
$\mathcal A_ne_0=-Z$ and $\mathcal A_ne_j=X_j$ for $1\le j\le n$,
and let $\mathcal A_n^\bullet$ denote its restriction to
$E_n^\bullet$, for $\bullet\in\{B,A\}$.
The definitions give
\begin{equation}
\label{eq:finite-gaussian-hulls-as-images}
\mathcal Q_{d,n}^\bullet
=\frac1{\sqrt n}\mathcal A_n^\bullet(\bK_n^\bullet),
\qquad \bullet\in\{B,A\}.
\end{equation}
We claim that for every realization of the Gaussian vectors $Z, X_1,\ldots, X_n$, one has
\begin{equation}
\label{eq:finite-BA-absorption-kernel-equivalence}
\{\sigma Z\in\mathcal Q_{d,n}^\bullet\}
=\{\ker\mathcal A_n^\bullet\cap
       \bW_n^\bullet(\sigma\sqrt n)\ne\{0\}\}.
\end{equation}
Indeed, since $\bK_n^\bullet$ is convex, every nonzero vector of
$\bW_n^\bullet(\sigma\sqrt n)$ has the form
$\lambda(\sigma\sqrt n\,e_0+x)$ with $\lambda>0$ and
$x\in\bK_n^\bullet$. Conversely, every such vector belongs to
the cone and is nonzero, because its $e_0$ coordinate is positive.
Its image under $\mathcal A_n^\bullet$ vanishes precisely when
$\mathcal A_n^\bullet x=\sigma\sqrt n\,Z$.
Together with \eqref{eq:finite-gaussian-hulls-as-images}, this
proves \eqref{eq:finite-BA-absorption-kernel-equivalence}.

Next we apply the conic Crofton formula
\eqref{eq:finite-dimensional-conic-crofton} to determine the probability of the event on the right-hand side of~\eqref{eq:finite-BA-absorption-kernel-equivalence}.  In any fixed orthonormal bases, $\mathcal A_n^\bullet$ is represented
by a matrix with independent standard normal entries.
Thus $\mathcal A_n^\bullet$ has rank $\min(d,D_n^\bullet)$ almost surely.
If $d\le D_n^\bullet$, orthogonal invariance shows that its kernel is
uniformly distributed among the subspaces of codimension $d$ in
$E_n^\bullet$.
The cone $\bW_n^\bullet(\sigma\sqrt n)$ is pointed and
full-dimensional in this space by
Subsection~\ref{subsec:type-B-bodies-and-cones}; in particular,
it is not a linear subspace. The conic Crofton formula
\eqref{eq:finite-dimensional-conic-crofton} therefore gives
\eqref{eq:finite-BA-gaussian-absorption}.
This includes $d=D_n^\bullet$, when the kernel is $\{0\}$.
If $d>D_n^\bullet$, the map $\mathcal A_n^\bullet$ is injective almost surely and the
left-hand side is zero by
\eqref{eq:finite-BA-absorption-kernel-equivalence}.
The right-hand side is also zero, since all its indices exceed
$D_n^\bullet$. The stated vanishing cases follow.
\end{proof}

\subsection{Proof of Theorem~\ref{thm:brownian-gaussian-absorption}}
\label{subsec:finite-gaussian-hulls-absorption-proof-short}
To pass from the discrete-time hulls to the continuous ones, we use the
following elementary lemma. Write $d_{\mathrm H}$ for Hausdorff
distance between nonempty compact subsets of $\R^d$.

\begin{lemma}
\label{lem:absorption-hausdorff-continuity}
Let $K,K_1,K_2,\ldots$ be random nonempty compact convex subsets
of $\R^d$. Suppose that $d_{\mathrm H}(K_n,K)\to0$ almost surely.
Let $Y$ be an $\R^d$-valued random vector with a Lebesgue density,
independent of $K,K_1,K_2,\ldots$. Then
$\Pp(Y\in\bd K)=0$ and
\begin{equation}\label{eq:lem:absorption-hausdorff-continuity}
\ind_{\{Y\in K_n\}}\longrightarrow\ind_{\{Y\in K\}}
\quad\text{almost surely}.
\end{equation}
In particular, $\Pp(Y\in K_n)\to\Pp(Y\in K)$.
\end{lemma}

\begin{proof}
First, the boundary $\bd M$ of any nonempty compact convex set
$M\subseteq\R^d$ has zero Lebesgue measure. Indeed, if $M$ has empty
interior, it lies in a proper affine subspace and the assertion follows. Otherwise, choose
$a\in\intr M$. For $0<r<1$, convexity gives
$a+r(M-a)\subseteq\intr M$, whence
$\operatorname{vol}_d(\bd M)\le(1-r^d)\operatorname{vol}_d(M)$.
Letting $r\uparrow1$ proves that $\bd M$ is a Lebesgue-zero set.

If $f_Y$ is the density of $Y$,  independence of $Y$ and $K$  yields
\[
\Pp(Y\in\bd K)
=\Ee\left[\int_{\bd K}f_Y(y)\,\dint y\right]=0.
\]

For deterministic compact convex sets $K_n$ converging to $K$ in Hausdorff
distance, membership stabilizes at every $x\notin\bd K$, that is $\ind_{\{x\in K_n\}} = \ind_{\{x\in K\}}$ for sufficiently large $n$.
Indeed, if $x\notin K$, then $\dist(x,K)>0$, so eventually
$x\notin K_n$. If $x\in\intr K$, choose $r>0$ such that the
closed ball of radius $r$ about $x$ lies in $K$. Suppose that
$d_{\mathrm H}(K_n,K)<r$ but $x\notin K_n$. Strict separation
provides a unit vector $u$ satisfying
$\sup_{z\in K_n}\langle u,z\rangle<\langle u,x\rangle$.
However, $x+ru\in K$ has distance less than $r$ from some
$z\in K_n$, which gives
$\langle u,z\rangle>\langle u,x\rangle$, a contradiction.

On the event of Hausdorff convergence, the deterministic argument
therefore applies to $x=Y$ almost surely, proving~\eqref{eq:lem:absorption-hausdorff-continuity}. Dominated (bounded) convergence
gives the convergence of probabilities.
\end{proof}

\begin{proof}[Proof of Theorem~\ref{thm:brownian-gaussian-absorption}]
\emph{Step 1: Approximate the continuous hulls by discrete grids.}
Write $W_t^B:=W_t$ and $W_t^A:=W_t^{\mathrm{br}}$, and define
\begin{equation}
\label{eq:finite-BA-brownian-grid-hulls}
\mathcal K_d^{\bullet,[n]}
:=\conv\{W_{i/n}^\bullet:i=0,\ldots,n\},
\qquad n\ge1,\quad \bullet\in\{B,A\}.
\end{equation}
For each $n$, the vectors
$\sqrt n\,(W_{j/n}-W_{(j-1)/n})$, $1\le j\le n$, are
independent $N(0,I_d)$ vectors, independent of $Z$.
Their partial sums are $\sqrt n\,W_{i/n}$.  Thus $(\mathcal Q_{d,n}^\bullet,Z)$ has the same law as
$(\mathcal K_d^{\bullet,[n]},Z)$.

For $\bullet\in\{B,A\}$, path continuity of the Brownian motion  gives almost surely
\begin{equation}
\label{eq:finite-BA-brownian-grid-hausdorff-convergence}
d_{\mathrm H}(\mathcal K_d^{\bullet,[n]},\mathcal K_d^\bullet)
\le\sup_{0\le t\le1}
\bigl\|W_t^\bullet-W_{\lfloor nt\rfloor/n}^\bullet\bigr\|
\longrightarrow0.
\end{equation}
Indeed, the discrete  hull $\mathcal K_d^{\bullet,[n]}$ is contained in the full hull $\mathcal K_d^\bullet$, and every
finite convex combination of path values is approximated by the
same combination of their grid values with at most the displayed
error. Since $\sigma>0$, the independent point $\sigma Z$ has a
Lebesgue density. Lemma~\ref{lem:absorption-hausdorff-continuity}
therefore proves both $\Pp(\sigma Z\in\bd\mathcal K_d^\bullet)=0$
and
\begin{equation}
\label{eq:finite-BA-grid-absorption-limit}
q_d^\bullet(\sigma)
=\lim_{n\to\infty}\Pp(\sigma Z\in\mathcal K_d^{\bullet,[n]})
=\lim_{n\to\infty}\Pp(\sigma Z\in\mathcal Q_{d,n}^\bullet).
\end{equation}

\medskip
\emph{Step 2: Use conic intrinsic volumes and pass to the limit.}
Substituting \eqref{eq:finite-BA-gaussian-absorption} into
\eqref{eq:finite-BA-grid-absorption-limit} gives
\[
q_d^\bullet(\sigma)
=\lim_{n\to\infty}2\sum_{j=0}^\infty
\upsilon_{d+1+2j}\bigl(\bW_n^\bullet(\sigma\sqrt n)\bigr).
\]
The series $2\sum_{j=0}^\infty L_{d+1+2j}^\bullet(\sigma)$ is
finite by \eqref{eq:finite-BA-limiting-parity-normalization}.
The absolute difference between the sum inside the limit and this
series is at most
\[
2\sum_{k=0}^\infty
\left|\upsilon_k\bigl(\bW_n^\bullet(\sigma\sqrt n)\bigr)
      -L_k^\bullet(\sigma)\right|,
\]
which tends to zero by
\eqref{eq:finite-BA-intrinsic-volume-l1-convergence}, with
$\rho_n=\sigma\sqrt n$ and $c=\sigma$. Hence
$
q_d^\bullet(\sigma)=2\sum_{j=0}^\infty
L_{d+1+2j}^\bullet(\sigma).
$

\medskip
\emph{Step 3: Evaluate the sums.}
For type $B$, insert
\eqref{eq:type-B-limiting-intrinsic-volumes} into the series $q_d^\bullet(\sigma)=2\sum_{j=0}^\infty L_{d+1+2j}^\bullet(\sigma)$.
All integrands are nonnegative, so Tonelli's theorem allows the
sum to be taken under the integral sign. Its numerator becomes
$\sum_{j=0}^\infty x^{d+2j}/(d+2j)!=\Psi_d(x)$, proving the
integral formula in
\eqref{eq:type-B-brownian-gaussian-absorption}.
The integral is finite because the series is finite.
For type $A$, inserting~\eqref{eq:type-A-limiting-intrinsic-volumes} into the series gives
\[
q_d^A(\sigma)
=
2\sum_{j=0}^\infty L_{d+1+2j}^A(\sigma)
=\frac1{\sinh\theta_\sigma}
  \sum_{j=0}^\infty
  \frac{\theta_\sigma^{d+1+2j}}{(d+1+2j)!}
=\frac{\Psi_{d+1}(\theta_\sigma)}{\sinh\theta_\sigma},
\]
which proves \eqref{eq:type-A-brownian-gaussian-absorption}.
\end{proof}

\subsection{Origin absorption for Ornstein--Uhlenbeck paths and circular loops}
\label{subsec:ou-origin-absorption}

We now relate the Gaussian-point absorption by Brownian convex hulls to origin absorption
for Ornstein--Uhlenbeck hulls. Throughout this subsection, the
ordinary Ornstein--Uhlenbeck process has unit variance and rate $1/2$.
Let $U=(U_u)_{u\in \mathbb R}$ be a centered stationary Gaussian process in
$\R^d$ with continuous paths and covariance
\begin{equation}
\label{eq:ordinary-ou-covariance}
\Ee[U_uU_v^{\mathsf T}]=\eee^{-|u-v|/2}I_d,
\qquad u,v\in \mathbb R.
\end{equation}
For each $T>0$, by a \emph{circular Ornstein--Uhlenbeck loop of
circumference $T$} we mean a centered Gaussian process
$V^{(T)}=(V_u^{(T)})_{0\le u\le T}$ in $\R^d$ with continuous paths
and covariance
\begin{equation}
\label{eq:circular-ou-covariance}
\Ee\bigl[V_u^{(T)}(V_v^{(T)})^{\mathsf T}\bigr]
=\frac{\cosh(T/4-|u-v|/2)}{\cosh(T/4)}I_d,
\qquad 0\le u,v\le T.
\end{equation}
In particular, $V_0^{(T)}=V_T^{(T)}$ almost surely. The covariance is
unchanged when the separation $h=|u-v|$ is replaced by $T-h$; hence,
after identifying the endpoints, the law is invariant under rotations
of the circle. The representation in Proposition~\ref{prop:ou-brownian-representations} below also establishes the existence
of a continuous process with this covariance.

Define the random convex hulls
\begin{equation}
\label{eq:ou-convex-hulls}
\mathcal H_{d,T}^B:=\conv\{U_u:0\le u\le T\}\subseteq \R^d,
\qquad
\mathcal H_{d,T}^A:=\conv\{V_u^{(T)}:0\le u\le T\}\subseteq \R^d.
\end{equation}
The superscript $A$ thus refers to the circular
loop, whereas $B$ refers to the ordinary stationary Ornstein--Uhlenbeck process. We are interested in evaluating the origin absorption probability
$$
a_d^\bullet(T):=\Pp(0\in\mathcal H_{d,T}^\bullet), \qquad \bullet\in\{B,A\}.
$$

The next proposition relates appropriately time-changed Ornstein--Uhlenbeck processes $U$ and $V^{(T)}$ to the Brownian motion and the Brownian bridge, respectively. First we introduce the notation for these time-changes.  For $\rho>0$, put
\begin{equation}
\label{eq:ou-time-lengths}
T_B(\rho):=\log(1+\rho^{-2}),
\qquad
T_A(\rho):=2\theta_\rho=4\arsinh\frac1{2\rho},
\end{equation}
where $\theta_\rho=2\arsinh(1/(2\rho))$ as before. For $\rho>0$, put $r:=\eee^{-T_A(\rho)/2}\in(0,1)$ and define
\begin{align}
u_\rho^B(t)&:=\log(1+t/\rho^2),
\qquad 0\le t\le1,
\label{eq:ordinary-ou-time-change}
\\
u_\rho^A(t)&:=
\log\frac{r+(1-r)t}{r(1-(1-r)t)},
\qquad 0\le t\le1.
\label{eq:circular-ou-time-change}
\end{align}

\begin{proposition}[Representations of Ornstein--Uhlenbeck processes]
\label{prop:ou-brownian-representations}
Let $W$, $W^{\mathrm{br}}$, and $Z$ be as in
Subsection~\ref{subsec:brownian-gaussian-absorption}, and let $\rho>0$.
The maps $u_\rho^B$ and $u_\rho^A$ are increasing bijections from
$[0,1]$ onto $[0,T_B(\rho)]$ and $[0,T_A(\rho)]$, respectively.
Define the processes
\begin{align}
\widetilde U_{u_\rho^B(t)}
&:=\frac{W_t-\rho Z}{\sqrt{\rho^2+t}},
\qquad 0\le t\le1,
\label{eq:ordinary-ou-brownian-representation}
\\
\widetilde V_{u_\rho^A(t)}
&:=\frac{W_t^{\mathrm{br}}-\rho Z}
{\sqrt{\rho^2+t(1-t)}},
\qquad 0\le t\le1.
\label{eq:circular-ou-bridge-representation}
\end{align}
Then $\widetilde U$ and $\widetilde V$ have continuous sample
paths and
\begin{align*}
(\widetilde U_u)_{0\le u\le T_B(\rho)}
\stackrel{\mathrm d}{=}
(U_u)_{0\le u\le T_B(\rho)},
\qquad
(\widetilde V_u)_{0\le u\le T_A(\rho)}
\stackrel{\mathrm d}{=}
(V_u^{(T_A(\rho))})_{0\le u\le T_A(\rho)}.
\end{align*}
Both equalities hold in the spaces of continuous
$\R^d$-valued functions on the respective index intervals.
\end{proposition}

\begin{proof}
Both constructed processes are centered Gaussian. We first
verify that their covariance functions agree with those of
the claimed Ornstein--Uhlenbeck processes.

\medskip
\emph{Ordinary OU process.}
The map $u_\rho^B$ is strictly increasing on $[0,1]$, with endpoint values $0$
and $T_B(\rho)$; its inverse is $t=\rho^2(\eee^u-1)$.
For $0\le s\le t\le1$, independence of $Z$ and $(W_t)_{t\geq 0}$ gives
\begin{equation}
\label{eq:ordinary-ou-brownian-covariance}
\Ee\bigl[\widetilde U_{u_\rho^B(s)}
         \widetilde U_{u_\rho^B(t)}^{\mathsf T}\bigr]
=\frac{\rho^2+s}{\sqrt{\rho^2+s}\sqrt{\rho^2+t}}I_d
=\eee^{-(u_\rho^B(t)-u_\rho^B(s))/2}I_d
=
\Ee\bigl[U_{u_\rho^B(s)} U_{u_\rho^B(t)}^{\mathsf T}\bigr],
\end{equation}
where the last step uses~\eqref{eq:ordinary-ou-covariance}.

\medskip
\emph{Circular loop.}
Put $T:=T_A(\rho)$ and recall that $r=\eee^{-T/2}$.
Then $\rho^2= r/(1-r)^2$.
Equation~\eqref{eq:circular-ou-time-change} shows that $t\mapsto u_\rho^A(t)$ is strictly increasing, with
$u_\rho^A(0)=0$ and $u_\rho^A(1)=-2\log r=T$.
For $0\le s\le t\le1$, independence of $Z$ and
$W^{\mathrm{br}}$, together with the Brownian-bridge
covariance function, gives
\begin{equation}
\label{eq:circular-ou-bridge-covariance}
\Ee\bigl[\widetilde V_{u_\rho^A(s)}
         \widetilde V_{u_\rho^A(t)}^{\mathsf T}\bigr]
=\frac{\rho^2+s(1-t)}
 {\sqrt{\rho^2+s(1-s)}\sqrt{\rho^2+t(1-t)}}I_d.
\end{equation}
On the other hand, the covariance kernel of the circular Ornstein--Uhlenbeck process satisfies
\[
\frac{\cosh(T/4-h/2)}{\cosh(T/4)}
=\frac{\eee^{-h/2}+r\eee^{h/2}}{1+r},
\qquad 0\le h\le T.
\]
By \eqref{eq:circular-ou-time-change},
\[
\eee^{-(u_\rho^A(t)-u_\rho^A(s))/2}
=\sqrt{\frac{(r+(1-r)s)(1-(1-r)t)}
             {(r+(1-r)t)(1-(1-r)s)}}.
\]
Substituting into the circular covariance kernel and simplifying,
we obtain
\begin{align*}
\Ee\bigl[
V_{u_\rho^A(s)}^{(T)}
(V_{u_\rho^A(t)}^{(T)})^{\mathsf T}
\bigr]
&=
\frac{r+(1-r)^2s(1-t)}
{\sqrt{r+(1-r)^2s(1-s)}
 \sqrt{r+(1-r)^2t(1-t)}}I_d
\\
&=
\frac{\rho^2+s(1-t)}
{\sqrt{\rho^2+s(1-s)}
 \sqrt{\rho^2+t(1-t)}}I_d.
\end{align*}
This agrees with \eqref{eq:circular-ou-bridge-covariance}.
By symmetry, the covariance functions agree for all pairs
of time parameters. Since $T_A$ maps $(0,\infty)$ onto
$(0,\infty)$, the construction also establishes the existence
of the circular loop for every circumference $T>0$.

\medskip
Finally, both constructed processes have continuous sample
paths, since the Brownian motion and the Brownian bridge are
continuous, the time changes are homeomorphisms, and the
denominators in
\eqref{eq:ordinary-ou-brownian-representation} and
\eqref{eq:circular-ou-bridge-representation}
are strictly positive.
The covariance identities therefore imply equality of
finite-dimensional distributions and hence equality in law
in the respective spaces of continuous functions.
\end{proof}

Recall from~\eqref{eq:ou-convex-hulls} the definitions of the convex hulls  $\mathcal H_{d,T}^B$ and $\mathcal H_{d,T}^A$.  The next theorem determines the origin absorption probabilities $a_d^\bullet(T)=\Pp(0\in\mathcal H_{d,T}^\bullet)$,  $\bullet\in\{B,A\}$.

\begin{theorem}[Origin absorption by Ornstein--Uhlenbeck processes]
\label{thm:ou-origin-absorption}
For $T>0$, put
\begin{equation}
\label{eq:ou-inverse-time-scales}
\rho_B(T):=\frac1{\sqrt{\eee^T-1}},
\qquad
\rho_A(T):=\frac1{2\sinh(T/4)}.
\end{equation}
For every integer $d\ge1$ and every $T>0$,
\begin{align}
a_d^B(T)
&=q_d^B(\rho_B(T))
=\frac2\pi\int_0^{\arsinh\sqrt{\eee^T-1}}
\frac{\Psi_d(x)\,\dint x}{\sqrt{\eee^T-1-\sinh^2x}},
\label{eq:ordinary-ou-origin-absorption}
\\
a_d^A(T)
&=q_d^A(\rho_A(T))
=\frac{\Psi_{d+1}(T/2)}{\sinh(T/2)},
\label{eq:circular-ou-origin-absorption}
\end{align}
where $\Psi_m$ is defined in \eqref{eq:brownian-hyperbolic-remainders}.
The integral is convergent. Moreover,
$\Pp(0\in\bd\mathcal H_{d,T}^\bullet)=0$ for $\bullet\in\{B,A\}$, and
\begin{equation}
\label{eq:ou-absorption-positive-spanning}
a_d^\bullet(T)
=\Pp(0\in\intr\mathcal H_{d,T}^\bullet)
=\Pp(\pos\mathcal H_{d,T}^\bullet=\R^d),
\qquad \bullet\in\{B,A\}.
\end{equation}
The last event says that the positive hull of the corresponding
process over $[0,T]$ is all of $\R^d$.
\end{theorem}

For the proof we  use the following elementary observation.
\begin{lemma}[Positive rescaling and origin absorption]
\label{lem:positive-rescaling-origin}
Let $E,F\subseteq\R^d$ be nonempty sets such that every point
of $E$ is a positive multiple of some point of $F$, and
conversely. Then $\pos E=\pos F$, and
\begin{equation}
\label{eq:lem:positive-rescaling-origin_1}
0\in\conv E\iff 0\in\conv F,
\qquad
0\in\intr\conv E\iff 0\in\intr\conv F.
\end{equation}
If both sets are compact, then also
\begin{equation}
\label{eq:lem:positive-rescaling-origin_2}
0\in\bd\conv E\iff 0\in\bd\conv F.
\end{equation}
Moreover, for every nonempty set $E\subseteq\R^d$,
\begin{equation}
\label{eq:origin-interior-positive-spanning}
0\in\intr\conv E
\quad\Longleftrightarrow\quad
\pos E=\R^d.
\end{equation}
\end{lemma}

\begin{proof}
By assumption, every point of either set belongs to the
positive hull of the other, so $\pos E=\pos F$.
Suppose that $0\in\conv E$. Then
$0=\sum_{i=1}^m\lambda_i x_i$ for some $x_i\in E$ and
$\lambda_i\ge0$ with $\sum_i\lambda_i=1$.
Write $x_i=c_i y_i$, where $c_i>0$ and $y_i\in F$.
Dividing the coefficients $\lambda_i c_i$ by their positive
sum gives a convex combination of the $y_i$ equal to zero.
The reverse implication follows in the same way, proving the first equivalence in~\eqref{eq:lem:positive-rescaling-origin_1}.

To prove \eqref{eq:origin-interior-positive-spanning},
observe first that a neighborhood of zero contained in
$\conv E$ implies $\pos E=\R^d$. Conversely, if
$\pos E=\R^d$, then each of the vectors
$\pm e_1,\ldots,\pm e_d$ can be expressed as a finite
nonnegative linear combination of points of $E$.
Normalizing the coefficients shows that $\conv E$
contains a positive multiple of each of these vectors.
Their convex hull contains a neighborhood of zero,
proving the converse implication of~\eqref{eq:origin-interior-positive-spanning}.

The equality $\pos E=\pos F$ now implies the second
equivalence in~\eqref{eq:lem:positive-rescaling-origin_1}.
Finally, if $E$ and $F$ are compact, then their convex
hulls are compact and hence closed. Thus zero belongs
to the boundary of either convex hull precisely when
it belongs to that convex hull but not to its interior.
In view of~\eqref{eq:lem:positive-rescaling-origin_1},
this implies~\eqref{eq:lem:positive-rescaling-origin_2}.
\end{proof}

\begin{proof}[Proof of Theorem~\ref{thm:ou-origin-absorption}]
Fix $T>0$ and write $\rho_B:=\rho_B(T)$ and
$\rho_A:=\rho_A(T)$. By \eqref{eq:ou-inverse-time-scales},
$T_B(\rho_B)=T_A(\rho_A)=T$.
Apply Proposition~\ref{prop:ou-brownian-representations} with
$\rho=\rho_B$ in the ordinary OU case and $\rho=\rho_A$ in the
circular OU case. Denote the convex hulls of the resulting processes
$\widetilde U$ and $\widetilde V$ by
\[
\widetilde{\mathcal H}_{d,T}^B
:=\conv\{\widetilde U_u:0\le u\le T\},
\qquad
\widetilde{\mathcal H}_{d,T}^A
:=\conv\{\widetilde V_u:0\le u\le T\}.
\]
By Proposition~\ref{prop:ou-brownian-representations},
$\widetilde{\mathcal H}_{d,T}^\bullet
\stackrel{\mathrm d}{=}\mathcal H_{d,T}^\bullet$
for $\bullet\in\{B,A\}$.

The time changes in Proposition~\ref{prop:ou-brownian-representations} are bijective, and the
normalizing factors in
\eqref{eq:ordinary-ou-brownian-representation} and
\eqref{eq:circular-ou-bridge-representation} are strictly positive.
Lemma~\ref{lem:positive-rescaling-origin} therefore gives the
following equality  of events:
\begin{align*}
\{0\in\widetilde{\mathcal H}_{d,T}^B\}
&=\bigl\{0\in\conv\{W_t-\rho_B Z:0\le t\le1\}\bigr\}
=\{\rho_B Z\in\mathcal K_d^B\},
\\
\{0\in\widetilde{\mathcal H}_{d,T}^A\}
&=\bigl\{0\in\conv\{W_t^{\mathrm{br}}-\rho_A Z:
0\le t\le1\}\bigr\}
=\{\rho_A Z\in\mathcal K_d^A\}.
\end{align*}
Taking probabilities and using
\eqref{eq:brownian-gaussian-absorption-functions}, we obtain
\[
a_d^\bullet(T)
=
\Pp(0 \in \mathcal H_{d,T}^\bullet)
=
\Pp(0 \in \widetilde{\mathcal H}_{d,T}^\bullet)
=\Pp(\rho_\bullet Z\in\mathcal K_d^\bullet)
=q_d^\bullet(\rho_\bullet),
\qquad \bullet\in\{B,A\}.
\]
The explicit formulas in
\eqref{eq:ordinary-ou-origin-absorption} and
\eqref{eq:circular-ou-origin-absorption}, including convergence
of the integral, now follow from
Theorem~\ref{thm:brownian-gaussian-absorption}, using
$\rho_B^{-2}=\eee^T-1$ and $\theta_{\rho_A}=T/2$.

Since the ranges of the processes are compact,
Lemma~\ref{lem:positive-rescaling-origin} also gives
\[
\{0\in\bd\widetilde{\mathcal H}_{d,T}^\bullet\}
=\{\rho_\bullet Z\in\bd\mathcal K_d^\bullet\},
\qquad \bullet\in\{B,A\}.
\]
The assertion $\Pp(\sigma Z\in\bd\mathcal K_d^\bullet)=0$ of
Theorem~\ref{thm:brownian-gaussian-absorption} therefore implies
$\Pp(0\in\bd\mathcal H_{d,T}^\bullet)=0$.
Finally, by~\eqref{eq:origin-interior-positive-spanning}, we have $\{0\in\intr\mathcal H_{d,T}^\bullet\}=\{\pos\mathcal H_{d,T}^\bullet=\R^d\}$,
which proves~\eqref{eq:ou-absorption-positive-spanning}.
\end{proof}

\begin{corollary}[Low-dimensional examples]
\label{cor:ou-origin-absorption-low-dimensions}
For $T>0$ and $d=1,2$, the origin absorption probabilities are
\begin{align}
a_1^B(T)&=\frac2\pi\arctan\sqrt{\eee^T-1},
\label{eq:ordinary-ou-origin-absorption-dimension-one}
\\
a_2^B(T)&=1-\frac2\pi\int_0^{\pi/2}
\frac{\dint\theta}{\sqrt{1+(\eee^T-1)\sin^2\theta}},
\label{eq:ordinary-ou-origin-absorption-dimension-two}
\\
a_1^A(T)&=\tanh(T/4),
\label{eq:circular-ou-origin-absorption-dimension-one}
\\
a_2^A(T)&=1-\frac{T/2}{\sinh(T/2)}.
\label{eq:circular-ou-origin-absorption-dimension-two}
\end{align}
\end{corollary}

\begin{proof}
Substitute \eqref{eq:ou-inverse-time-scales} into
Corollary~\ref{cor:brownian-gaussian-absorption-consequences}, using
$\theta_{\rho_A(T)}=T/2$ and
$(1+4\rho_A(T)^2)^{-1/2}=\tanh(T/4)$.
\end{proof}

\begin{remark}[Circular OU loops versus endpoint conditioning]
\label{rem:circular-ou-endpoint-conditioning}
The circular OU loop $V^{(T)}$ with covariance given by~\eqref{eq:circular-ou-covariance} is not the
stationary process $U$ conditioned to have equal values at times
$0$ and $T$. Indeed, let $X_u$ be one coordinate of $U_u$ and write
$r:=\eee^{-T/2}$. Gaussian conditioning on $X_T-X_0=0$ gives
\[
\operatorname{Var}(X_0\mid X_T-X_0=0)=\frac{1+r}{2},
\qquad
\operatorname{Var}(X_{T/2}\mid X_T-X_0=0)=1.
\]
To see this, the conditioning variable $X_T-X_0$ has variance $2(1-r)$ and
covariances $r-1$ with $X_0$ and $0$ with $X_{T/2}$.
The resulting conditioned process has nonconstant variance, unlike
the circular loop.
\end{remark}

\subsection{Asymptotics of the Ornstein--Uhlenbeck absorption probabilities}
\label{subsec:ou-absorption-asymptotics}

Recall that $a_d^B(T)$ is the probability that the convex hull of the
stationary Ornstein--Uhlenbeck process over $[0,T]$ in $\R^d$ contains
zero, whereas $a_d^A(T)$ is the corresponding probability for the
circular OU loop of circumference $T$. We retain the covariance
normalizations \eqref{eq:ordinary-ou-covariance} and
\eqref{eq:circular-ou-covariance}, respectively. We now investigate the asymptotic behavior of the origin absorption
probabilities $a_d^B(T)$ and $a_d^A(T)$. We first consider a fixed
dimension $d$ and a growing time horizon $T$, and then the joint growth of
$d$ and $T$.

\subsubsection{Fixed dimension and a long time horizon}
\label{subsubsec:ou-absorption-fixed-dimension}
The following theorem shows that, for each fixed dimension $d$,
the probability of origin non-absorption decays exponentially
as $T\to\infty$, and determines its precise asymptotic behavior
for both Ornstein--Uhlenbeck models.

\begin{theorem}[Long-time non-absorption in fixed dimension]
\label{thm:ou-nonabsorption-fixed-dimension}
For every fixed integer $d\ge1$, as $T\to\infty$,
\begin{align}
1-a_d^B(T)
&=\frac{2^{2-d}}{\pi(d-1)!}\,
  T^{d-1}\eee^{-T/2}\bigl(1+O(T^{-1})\bigr),
\label{eq:ordinary-ou-nonabsorption-long-time}
\\
1-a_d^A(T)
&=\frac{2^{2-d}}{(d-1)!}\,
  T^{d-1}\eee^{-T/2}\bigl(1+O(T^{-2})\bigr).
\label{eq:circular-ou-nonabsorption-long-time}
\end{align}
The constants implicit in the error terms may depend on $d$.
\end{theorem}
\begin{remark}
In particular, $(1-a_d^A(T))/(1-a_d^B(T))\to\pi$ as $T\to\infty$.
\end{remark}
\begin{proof}
By \eqref{eq:brownian-gaussian-nonabsorption-finite-sum} and
Theorem~\ref{thm:ou-origin-absorption},
\begin{equation}
\label{eq:ou-nonabsorption-finite-sum}
1-a_d^\bullet(T)
=2\sum_{j=0}^{\lfloor(d-1)/2\rfloor}
  L_{d-1-2j}^\bullet(\rho_\bullet(T)),
\qquad \bullet\in\{B,A\},
\end{equation}
where $\rho_B(T)$ and $\rho_A(T)$ are given in
\eqref{eq:ou-inverse-time-scales}. We estimate the coefficients $L_{k}^\bullet(\rho_\bullet(T))$ at
fixed indices $k= d-1-2j$ and show that the term with index $k = d-1$ dominates.

\medskip
\emph{Ordinary OU process.}
Put $b_T:=\arsinh\sqrt{\eee^T-1}$. Then
$\sinh b_T=\rho_B(T)^{-1}$, and
\begin{equation}
\label{eq:ordinary-ou-long-time-parameters}
b_T=\frac T2+\log2+O(\eee^{-T}),
\qquad
\rho_B(T)=\eee^{-T/2}\bigl(1+O(\eee^{-T})\bigr).
\end{equation}
For $k\ge1$, \eqref{eq:type-B-limiting-intrinsic-volumes} gives
\begin{equation}
\label{eq:ordinary-ou-fixed-index-rescaled-integral}
L_k^B(\rho_B(T))
=\frac{\rho_B(T)}{\pi(k-1)!}
 \int_0^{b_T}
 \frac{x^{k-1}}
      {\sqrt{1-(\sinh x/\sinh b_T)^2}}\,\dint x.
\end{equation}
For $b>0$ and $0\le u\le b$, the inequality
$\sinh(b-u)/\sinh b\le\eee^{-u}$ implies, after setting $u=b-x$,
\begin{align*}
0
\le \int_0^b
 \frac{x^{k-1}}{\sqrt{1-(\sinh x/\sinh b)^2}}\,\dint x
 -\frac{b^k}{k}
\le b^{k-1}\int_0^\infty
 \left(\frac1{\sqrt{1-\eee^{-2u}}}-1\right)\dint u
 =b^{k-1}\log2.
\end{align*}
The last integral equals $\int_0^1(1+v)^{-1}\,\dint v$ under the
substitution $v=\sqrt{1-\eee^{-2u}}$.
Consequently, the integral in
\eqref{eq:ordinary-ou-fixed-index-rescaled-integral} is
$b_T^k/k+O(b_T^{k-1})$ for every fixed $k\ge1$.
Combining this estimate with
\eqref{eq:ordinary-ou-long-time-parameters} yields
\begin{equation}
\label{eq:ordinary-ou-fixed-index-asymptotics}
L_k^B(\rho_B(T))
=\frac1{\pi k!}\left(\frac T2\right)^k
 \eee^{-T/2}\bigl(1+O(T^{-1})\bigr),
\qquad k\ge1.
\end{equation}
For $k=0$, \eqref{eq:type-B-limiting-zeroth-intrinsic-volume} gives
\[
L_0^B(\rho_B(T))
=\frac{\arctan\rho_B(T)}\pi
=\frac{\eee^{-T/2}}\pi\bigl(1+O(\eee^{-T})\bigr).
\]
Thus \eqref{eq:ordinary-ou-fixed-index-asymptotics} also holds for
$k=0$.

\medskip
\emph{Circular loop.}
Since $\theta_{\rho_A(T)}=T/2$,
Theorem~\ref{thm:finite-BA-scaling-limit} gives
\begin{align*}
L_0^A(\rho_A(T))
&=\frac1{1+\eee^{T/2}}
 =\eee^{-T/2}\bigl(1+O(\eee^{-T/2})\bigr),
\\
L_k^A(\rho_A(T))
&=\frac{(T/2)^k}{2k!\sinh(T/2)}
 =\frac1{k!}\left(\frac T2\right)^k
  \eee^{-T/2}\bigl(1+O(\eee^{-T})\bigr),
\qquad k\ge1.
\end{align*}

\medskip
For either type, the coefficient $L_{k}^\bullet(\rho_\bullet(T))$ at a fixed index $k$ is asymptotic
to a positive constant times $T^k\eee^{-T/2}$.
The largest index $d-1-2j$ in \eqref{eq:ou-nonabsorption-finite-sum} is $d-1$;
when $d\ge3$, every remaining index is at most $d-3$.
Their combined contribution, relative to the term of index $d-1$,
is therefore $O(T^{-2})$. For $d=1,2$, there are no remaining terms.
The type-$B$ leading coefficient has relative error $O(T^{-1})$,
whereas the type-$A$ leading coefficient has exponentially small
relative error, including when $d=1$.
Substituting the coefficient estimates into
\eqref{eq:ou-nonabsorption-finite-sum} proves
\eqref{eq:ordinary-ou-nonabsorption-long-time} and
\eqref{eq:circular-ou-nonabsorption-long-time}.
\end{proof}

\subsubsection{A high-dimensional transition}
\label{subsubsec:ou-absorption-high-dimension}

We now let the dimension $d$ and the time horizon $T$ grow together.
The next theorem shows that, in both Ornstein--Uhlenbeck models,
the absorption probability undergoes a transition from $0$ to $1$
as $T$ crosses the critical time $2d$. This transition occurs
in a window of order $\sqrt d$, and both models have the same
Gaussian limiting profile.

\begin{theorem}[Transition from non-absorption to absorption in high dimension]
\label{thm:ou-absorption-high-dimensional-transition}
Let $(T_d)_{d\ge1}$ be positive and suppose that, for some $\tau\in\R$,
\begin{equation}
\label{eq:ou-critical-time-scaling}
\frac{T_d-2d}{2\sqrt d}\longrightarrow\tau
\qquad\text{as }d\to\infty.
\end{equation}
Then
\begin{equation}
\label{eq:ou-high-dimensional-absorption-limit}
a_d^B(T_d)\longrightarrow\Phi(\tau),
\qquad
a_d^A(T_d)\longrightarrow\Phi(\tau),
\end{equation}
where $\Phi$ is the standard normal distribution function.
\end{theorem}

\begin{proof}
\emph{Step 1: Compare absorption probabilities with random variable tails.}
For $T>0$ and $\bullet\in\{B,A\}$, let $J_T^\bullet$ be a
nonnegative integer-valued random variable with distribution
\begin{equation}
\label{eq:ou-auxiliary-index-distribution}
\Pp(J_T^\bullet=k)
:=L_k^\bullet(\rho_\bullet(T)),\qquad k\ge0.
\end{equation}
This is a probability distribution by
\eqref{eq:finite-BA-limiting-parity-normalization}.
Set $a_0^\bullet(T):=1$.
Theorems~\ref{thm:brownian-gaussian-absorption} and
\ref{thm:ou-origin-absorption} express $a_m^\bullet(T)$ as twice
the sum of the coefficients $L_k^\bullet(\rho_\bullet(T))$ at indices $m+1,m+3,\ldots$.
Adding the formulas for dimensions $m-1$ and $m$ therefore gives
\begin{equation}
\label{eq:ou-absorption-ordinary-tail-identity}
\Pp(J_T^\bullet\ge m)
=\frac{a_{m-1}^\bullet(T)+a_m^\bullet(T)}2,
\qquad m\ge1.
\end{equation}
For $m=1$, the same identity follows from the sum over odd indices
in \eqref{eq:finite-BA-limiting-parity-normalization}.
Projecting either $(m+1)$-dimensional OU process onto its first $m$
coordinates shows that $a_{m+1}^\bullet(T)\le a_m^\bullet(T)$.
Together with \eqref{eq:ou-absorption-ordinary-tail-identity}, this
implies
\begin{equation}
\label{eq:ou-absorption-tail-interlacing}
\Pp(J_T^\bullet\ge m+1)
\le a_m^\bullet(T)
\le\Pp(J_T^\bullet\ge m),\qquad m\ge1.
\end{equation}
We will prove the central limit theorem
\begin{equation}
\label{eq:ou-auxiliary-index-central-limit}
\frac{J_T^\bullet-T/2}{\sqrt{T/2}}
\ \xrightarrow{\mathrm d}\ N(0,1)
\qquad\text{as }T\to\infty,\qquad \bullet\in\{B,A\},
\end{equation}
and then use \eqref{eq:ou-absorption-tail-interlacing}.

\medskip
\emph{Step 2: The ordinary OU process.}
To prove the CLT~\eqref{eq:ou-auxiliary-index-central-limit} for $J_T^B$, we first show that the conditional law of $J_T^B-1$ given $J_T^B>0$ is a mixture of Poisson distributions.
Put $\mu_T:=T/2$, $R_T:=\sqrt{\eee^T-1}$, and
$b_T:=\arsinh R_T$, as in
\eqref{eq:ordinary-ou-long-time-parameters}. Let
$p_T:=\Pp(J_T^B>0)=1-L_0^B(\rho_B(T))$.
Since $L_0^B(\rho_B(T))=\arctan(\rho_B(T))/\pi$, we have
$1/2<p_T<1$ and $p_T\to1$ as $T\to\infty$.
Summing the integrals in
\eqref{eq:type-B-limiting-intrinsic-volumes} over $k\ge1$ by
Tonelli's theorem shows that
\begin{equation}
\label{eq:ordinary-ou-poisson-mixing-density}
f_T(x):=\frac{\eee^x}
 {\pi p_T\sqrt{R_T^2-\sinh^2x}},\qquad 0<x<b_T,
\end{equation}
is a probability density. Let $X_T$ have this density and, conditionally
on $X_T$, let $M_T$ be Poisson with mean $X_T$. For every $k\ge1$,
\[
\Pp(1+M_T=k)
=\int_0^{b_T}\frac{\eee^{-x}x^{k-1}}{(k-1)!}
 f_T(x)\,\dint x
=\frac{L_k^B(\rho_B(T))}{p_T}.
\]
Thus the conditional law of $J_T^B$ given $J_T^B>0$ is the law of
$1+M_T$.

We next show that the random Poisson mean $X_T$ differs from $T/2$ by
only a tight error. Under the substitution
$Y_T:=\sinh X_T/R_T$, the density in
\eqref{eq:ordinary-ou-poisson-mixing-density} becomes
\[
g_T(y)
=\frac{1+R_Ty/\sqrt{1+R_T^2y^2}}
 {\pi p_T\sqrt{1-y^2}}
\le\frac4{\pi\sqrt{1-y^2}},\qquad 0<y<1.
\]
For $0\le r<b_T$, the inequality
$\sinh(b_T-r)/\sinh b_T\le\eee^{-r}$ gives
\[
\Pp(b_T-X_T>r)
\le\frac4\pi\arcsin(\eee^{-r})\le2\eee^{-r}.
\]
For $r\ge b_T$, the probability is zero. Hence $b_T-X_T$ is tight,
and \eqref{eq:ordinary-ou-long-time-parameters} yields
\begin{equation}
\label{eq:ordinary-ou-poisson-mean-concentration}
X_T=\mu_T+O_{\Pp}(1),\qquad T\to\infty.
\end{equation}
Here $O_{\Pp}(1)$ denotes a family of random variables that is tight.

We now prove the required CLT using characteristic functions. For each fixed $u\in\R$,
\begin{equation}\label{eq:char-funct-identity}
\Ee\exp\left\{\frac{\mathrm i u(M_T-\mu_T)}{\sqrt{\mu_T}}\right\}
=\Ee\exp\left\{
 X_T\bigl(\eee^{\mathrm i u/\sqrt{\mu_T}}-1\bigr)
 -\mathrm i u\sqrt{\mu_T}\right\}.
\end{equation}
Taylor expansion and
\eqref{eq:ordinary-ou-poisson-mean-concentration} give
\[
X_T\bigl(\eee^{\mathrm i u/\sqrt{\mu_T}}-1\bigr)
 -\mathrm i u\sqrt{\mu_T}
=\frac{\mathrm i u(X_T-\mu_T)}{\sqrt{\mu_T}}
 -\frac{u^2X_T}{2\mu_T}
 +O\left(\frac{X_T}{\mu_T^{3/2}}\right)
\ \longrightarrow\ -\frac{u^2}{2}
\]
in probability. Since $X_T\ge0$ and $\cos(u/\sqrt{\mu_T})\le1$, the exponential inside the expectation on the
right-hand side of~\eqref{eq:char-funct-identity} has modulus at most one. The convergence of the exponential in probability to
$\eee^{-u^2/2}$ therefore implies convergence of its expectation
to the same limit. By L\'evy's continuity theorem for characteristic functions,
we conclude that
$(M_T-\mu_T)/\sqrt{\mu_T} \ \xrightarrow{\mathrm d}\ N(0,1)$.
The added constant $1$ is negligible on this scale, and
$\Pp(J_T^B=0)=1-p_T\to0$. The conditional representation therefore
proves \eqref{eq:ou-auxiliary-index-central-limit} for type $B$.

\medskip
\emph{Step 3: The circular OU loop.}
To prove the CLT~\eqref{eq:ou-auxiliary-index-central-limit} for $J_T^A$, we show that $J_T^A$ conditioned on $J_T^A>0$ has the same distribution as a Poisson variable with mean $T/2$ conditioned on being non-zero.
Indeed, with $\mu_T=T/2$,
Theorem~\ref{thm:finite-BA-scaling-limit} gives
\[
\Pp(J_T^A=0)=\frac1{1+\eee^{\mu_T}},
\qquad
\Pp(J_T^A=k)
=\frac{\mu_T^k}{k!(\eee^{\mu_T}-\eee^{-\mu_T})},\quad k\ge1.
\]
If $N_T$ is Poisson with mean $\mu_T$, then for $k\ge1$,
\[
\Pp(J_T^A=k\mid J_T^A>0)
=\frac{\mu_T^k}{k!(\eee^{\mu_T}-1)}
=\Pp(N_T=k\mid N_T>0).
\]
Both $\Pp(J_T^A=0)$ and $\Pp(N_T=0)$ tend to zero.
The central limit theorem for the Poisson distribution gives $(N_T-\mu_T)/\sqrt{\mu_T} \ \xrightarrow{\mathrm d}\ N(0,1)$ and  proves
\eqref{eq:ou-auxiliary-index-central-limit} for type $A$.

\medskip
\emph{Step 4: Complete the proof.}
Under \eqref{eq:ou-critical-time-scaling}, we have $T_d\to\infty$,
$\mu_{T_d}/d\to1$, and
\[
\frac{d-\mu_{T_d}}{\sqrt{\mu_{T_d}}}\longrightarrow-\tau,
\qquad
\frac{d+1-\mu_{T_d}}{\sqrt{\mu_{T_d}}}\longrightarrow-\tau.
\]
By \eqref{eq:ou-auxiliary-index-central-limit} and continuity of
the normal distribution function, both
$\Pp(J_{T_d}^\bullet\ge d)$ and
$\Pp(J_{T_d}^\bullet\ge d+1)$ tend to
$1-\Phi(-\tau)=\Phi(\tau)$.
The inequalities \eqref{eq:ou-absorption-tail-interlacing} now give
\eqref{eq:ou-high-dimensional-absorption-limit}.
\end{proof}

\begin{corollary}[A linear-time threshold]
\label{cor:ou-absorption-linear-time-threshold}
Let $T_d>0$ and suppose that $T_d/d\to c\in[0,\infty]$ with $c\ne2$.
Then, for $\bullet\in\{B,A\}$,
\begin{equation}
\label{eq:ou-absorption-linear-time-threshold}
\lim_{d\to\infty}a_d^\bullet(T_d)
=\begin{cases}
0,&c<2,\\
1,&c>2.
\end{cases}
\end{equation}
\end{corollary}

\begin{proof}
For each fixed $d$ and for $\bullet\in\{B,A\}$, the function $a_d^\bullet(T)$ is nondecreasing
in $T$. Indeed, the Brownian convex hull $\mathcal K_d^\bullet$ contains zero,
so convexity implies that $q_d^\bullet(\sigma)$ is nonincreasing
in $\sigma$. Both functions $\rho_\bullet(T)$ in
\eqref{eq:ou-inverse-time-scales} are decreasing, and
$a_d^\bullet(T)=q_d^\bullet(\rho_\bullet(T))$ by Theorem~\ref{thm:ou-origin-absorption}.
Note that this argument applies also to the circular loops.

Fix $M>0$. If $c<2$, then $T_d\le2d-2M\sqrt d$ for all sufficiently
large $d$; if $c>2$, then eventually $T_d\ge2d+2M\sqrt d$.
Monotonicity and
Theorem~\ref{thm:ou-absorption-high-dimensional-transition} give
\[
\limsup_{d\to\infty}a_d^\bullet(T_d) \le\Phi(-M) \text{ if } c<2,
\qquad \qquad
\liminf_{d\to\infty}a_d^\bullet(T_d)\ge\Phi(M)  \text{ if } c>2.
\]
Letting $M\to\infty$ proves the assertion.
\end{proof}

\subsection{Expected conic intrinsic volumes of random Ornstein--Uhlenbeck cones}
\label{subsec:random-ou-cones-expected-intrinsic-volumes}

We now determine the expected conic intrinsic volumes of the positive
hulls of the ordinary OU process and the circular OU loop; see Subsection~\ref{subsec:ou-origin-absorption} for the definitions of these processes.
For $d\ge1$ and $T>0$, define the random closed convex cones
\begin{equation}
\label{eq:random-ou-cones-definition}
\mathcal U_{d,T}^B
:=\posclos\mathcal H_{d,T}^B
=\posclos\{U_u:0\le u\le T\},
\qquad
\mathcal U_{d,T}^A
:=\posclos\mathcal H_{d,T}^A
=\posclos\{V_u^{(T)}:0\le u\le T\}.
\end{equation}
We refer to $\mathcal U_{d,T}^B$ and $\mathcal U_{d,T}^A$ as the
\emph{random OU cones} associated with the ordinary OU process
and the circular OU loop, respectively.

\begin{remark}[On measurability]
We regard these cones as random closed subsets of $\R^d$ with respect to the
Effros $\sigma$-field on the space of closed subsets of $\R^d$,
which coincides with the Borel $\sigma$-field of the Fell topology;
see \cite[Section~1.1]{Molchanov2017}.
To prove measurability, it suffices to show that, for every open set $O\subseteq\R^d$,
the set of outcomes $\{\omega:\mathcal U_{d,T}^\bullet(\omega)\cap O\ne\varnothing\}$
is measurable. By path continuity, each cone $\mathcal U_{d,T}^\bullet(\omega)$ is the closure of
the countable set of finite nonnegative rational linear combinations
of process values at times $rT$, where $r\in\Q\cap[0,1]$.
Each such combination is a measurable random vector.
Since taking closure does not change whether a set intersects
an open set, the displayed set is a countable union of
measurable sets.
\end{remark}

The following lemma shows that the positive hulls in
\eqref{eq:random-ou-cones-definition} are almost surely
closed and describes their geometry according to whether
the origin is absorbed.

\begin{lemma}
\label{lem:random-ou-cones-geometry}
For every $d\ge1$, $T>0$, and $\bullet\in\{B,A\}$,
the cone $\mathcal U_{d,T}^\bullet$ is full-dimensional
almost surely.

Moreover, whenever $0\in\intr\mathcal H_{d,T}^\bullet$,
we have
$\mathcal U_{d,T}^\bullet
=\pos\mathcal H_{d,T}^\bullet=\R^d$.
Whenever $0\notin\mathcal H_{d,T}^\bullet$,
the positive hull $\pos\mathcal H_{d,T}^\bullet$ is
closed and pointed, and hence
$\mathcal U_{d,T}^\bullet=\pos\mathcal H_{d,T}^\bullet$.
Consequently,
\begin{equation}
\label{eq:random-ou-cones-almost-sure-closedness}
\mathcal U_{d,T}^\bullet=\pos\mathcal H_{d,T}^\bullet
\quad\text{almost surely}.
\end{equation}
Moreover, on an almost sure event,
$\mathcal U_{d,T}^\bullet=\R^d$ if and only if
$0\in\intr\mathcal H_{d,T}^\bullet$, and
$\mathcal U_{d,T}^\bullet$ is pointed if and only if
$\mathcal U_{d,T}^\bullet\ne\R^d$.
In particular,
\begin{equation}
\label{eq:random-ou-cones-full-space-probability}
\Pp(\mathcal U_{d,T}^\bullet=\R^d)
=\Pp(0\in\intr\mathcal H_{d,T}^\bullet)
=a_d^\bullet(T).
\end{equation}
\end{lemma}

\begin{proof}
We first establish the deterministic geometric assertions.
If $0\in\intr\mathcal H_{d,T}^\bullet$, then
\eqref{eq:origin-interior-positive-spanning} gives
$\pos\mathcal H_{d,T}^\bullet=\R^d$ and hence
$\mathcal U_{d,T}^\bullet=\R^d$.

If $0\notin\mathcal H_{d,T}^\bullet$, compactness and strict
separation give a vector $v\in\R^d$ and $\varepsilon>0$
such that $\langle v,y\rangle\ge\varepsilon$ for all
$y\in\mathcal H_{d,T}^\bullet$.
Since $\mathcal H_{d,T}^\bullet$ is convex, every point of
its positive hull can be written as $\lambda y$ with
$\lambda\ge0$ and $y\in\mathcal H_{d,T}^\bullet$.
For any convergent sequence $\lambda_m y_m$ in this hull,
the inequality
$\lambda_m\varepsilon\le\langle v,\lambda_m y_m\rangle$
shows that $(\lambda_m)$ is bounded.
Compactness of $\mathcal H_{d,T}^\bullet$ then yields a
subsequence for which both $\lambda_m$ and $y_m$ converge.
Its limit belongs to the positive hull, proving that
$\pos\mathcal H_{d,T}^\bullet$ is closed.
Moreover, every nonzero vector in this hull has strictly
positive inner product with $v$, so the cone is pointed
and cannot equal $\R^d$.

By Theorem~\ref{thm:ou-origin-absorption},
$\Pp(0\in\bd\mathcal H_{d,T}^\bullet)=0$.
Thus, almost surely, precisely one of the two cases above
occurs. In the first case, $\mathcal U_{d,T}^\bullet=\R^d$;
in the second, $\mathcal U_{d,T}^\bullet$ is a closed
pointed cone distinct from $\R^d$.
This proves \eqref{eq:random-ou-cones-almost-sure-closedness},
the two asserted equivalences, and
\eqref{eq:random-ou-cones-full-space-probability}.

It remains to establish full dimensionality.
Use Proposition~\ref{prop:ou-brownian-representations}
with $\rho=\rho_\bullet(T)$ and choose
$0<t_1<\cdots<t_d<1$.
Write $W_t^B:=W_t$ and $W_t^A:=W_t^{\mathrm{br}}$.
The $d$ random vectors
$W_{t_i}^\bullet-\rho Z$, $1\le i\le d$,
have a joint Lebesgue density on $(\R^d)^d$.
Indeed, the scalar Brownian-motion and Brownian-bridge
covariance matrices at distinct interior times are
positive definite, as follows by expressing linear
combinations in terms of independent Brownian increments.
Adding the independent Gaussian shift preserves positive
definiteness, and the spatial coordinates are independent.
Consequently, the matrix formed by these $d$ vectors
is nonsingular almost surely, since singular matrices
form a Lebesgue null set.
The positive normalizations and time changes in
Proposition~\ref{prop:ou-brownian-representations}
preserve their linear span.
Hence $\mathcal U_{d,T}^\bullet$ contains $d$ linearly
independent vectors almost surely and is therefore
full-dimensional.
\end{proof}

The next lemma expresses the  expected $j$-th Grassmann angle of the random OU cone $\mathcal U_{d,T}^\bullet$ through
the OU absorption probability in
dimension $j$.  Set $a_0^\bullet(T):=1$.

\begin{lemma}[Expected Grassmann angles of random OU cones]
\label{lem:random-ou-cones-expected-grassmann-angles}
For every $d\ge1$, $T>0$, and $\bullet\in\{B,A\}$,
\begin{equation}
\label{eq:random-ou-cones-expected-grassmann-angles}
\Ee\,\gamma_j(\mathcal U_{d,T}^\bullet)=a_j^\bullet(T),
\qquad 0\le j<d.
\end{equation}
For $j\ge d$, the Grassmann angles are zero by definition.
\end{lemma}

\begin{proof}
For $j=0$, the assertion follows from the almost-sure nontriviality
of the cone. Let $1\le j<d$, let $L$ be an independent uniformly
distributed subspace of codimension $j$ in $\R^d$, and write
$\Pi:\R^d \to L^\perp$ for orthogonal projection onto $L^\perp$.
Outside the null event $\{0\in\bd\mathcal H_{d,T}^\bullet\}$ of
Theorem~\ref{thm:ou-origin-absorption},
\begin{equation}
\label{eq:random-ou-projection-absorption-equivalence}
\mathcal U_{d,T}^\bullet\cap L\ne\{0\}
\quad\Longleftrightarrow\quad
0\in\Pi\mathcal H_{d,T}^\bullet.
\end{equation}
Indeed, if $0\in\intr\mathcal H_{d,T}^\bullet$, both events hold:
the cone $\mathcal U_{d,T}^\bullet$ equals $\R^d$, so its
intersection with $L$ is nontrivial because $\dim L=d-j>0$;
moreover, the projected hull $\Pi\mathcal H_{d,T}^\bullet$
contains zero. If $0\notin\mathcal H_{d,T}^\bullet$, the cone
$\mathcal U_{d,T}^\bullet$ equals $\pos\mathcal H_{d,T}^\bullet$
by Lemma~\ref{lem:random-ou-cones-geometry}.
A nonzero point of its intersection with $L$ is $\lambda y$ for
some $\lambda>0$ and $y\in\mathcal H_{d,T}^\bullet$, and hence
$\Pi y=0$. Conversely, any $y\in\mathcal H_{d,T}^\bullet$ with
$\Pi y=0$ is nonzero and belongs to that intersection.
This proves \eqref{eq:random-ou-projection-absorption-equivalence}.

Conditionally on $L$, the projected OU process or the OU loop, expressed in an
orthonormal basis of $L^\perp$, has the law of the corresponding
$j$-dimensional OU process or loop with the same parameter $T$.
This follows because the covariance kernels in
\eqref{eq:ordinary-ou-covariance} and
\eqref{eq:circular-ou-covariance} are scalar multiples
of the identity matrix $I_d$.
Since projection commutes with taking the convex hull,
\eqref{eq:random-ou-projection-absorption-equivalence} gives
\[
\Ee\,\gamma_j(\mathcal U_{d,T}^\bullet)
=\Pp(\mathcal U_{d,T}^\bullet\cap L\ne\{0\})
=\Pp(0\in\Pi\mathcal H_{d,T}^\bullet)
=a_j^\bullet(T).
\]
The first equality is the definition of the Grassmann angle,
conditioning on the cone.
\end{proof}

\begin{theorem}[Expected conic intrinsic volumes of random OU cones]
\label{thm:random-ou-cones-expected-intrinsic-volumes}
For every $d\ge1$, $T>0$, and $\bullet\in\{B,A\}$,
\begin{equation}
\label{eq:random-ou-expected-intrinsic-volume-coefficients}
\Ee\,\upsilon_k(\mathcal U_{d,T}^\bullet)
=\begin{cases}
L_k^\bullet(\rho_\bullet(T)),&0\le k<d,\\[2pt]
\sum_{\ell=d}^\infty L_\ell^\bullet(\rho_\bullet(T)),
&k=d,
\end{cases}
\end{equation}
where $L_k^\bullet$ is defined in
Theorem~\ref{thm:finite-BA-scaling-limit} and $\rho_\bullet(T)$ in
\eqref{eq:ou-inverse-time-scales}.
Explicitly, for the ordinary Ornstein--Uhlenbeck process,
\begin{align}
\Ee\,\upsilon_0(\mathcal U_{d,T}^B)
&=\frac1\pi\arcsin(\eee^{-T/2}),
\label{eq:ordinary-random-ou-expected-zeroth-volume}
\\
\Ee\,\upsilon_k(\mathcal U_{d,T}^B)
&=\frac1{\pi(k-1)!}
\int_0^{\arsinh\sqrt{\eee^T-1}}
\frac{x^{k-1}}{\sqrt{\eee^T-1-\sinh^2x}}\,\dint x,
\qquad 1\le k<d,
\label{eq:ordinary-random-ou-expected-intrinsic-volumes}
\\
\Ee\,\upsilon_d(\mathcal U_{d,T}^B)
&=\frac1\pi\int_0^{\arsinh\sqrt{\eee^T-1}}
\frac{\Psi_{d-1}(x)+\Psi_d(x)}
     {\sqrt{\eee^T-1-\sinh^2x}}\,\dint x.
\label{eq:ordinary-random-ou-expected-top-volume}
\end{align}
For the circular Ornstein--Uhlenbeck loop,
\begin{align}
\Ee\,\upsilon_0(\mathcal U_{d,T}^A)
&=\frac12\left(1-\tanh\frac T4\right),
\label{eq:circular-random-ou-expected-zeroth-volume}
\\
\Ee\,\upsilon_k(\mathcal U_{d,T}^A)
&=\frac{(T/2)^k}{2k!\sinh(T/2)},
\qquad 1\le k<d,
\label{eq:circular-random-ou-expected-intrinsic-volumes}
\\
\Ee\,\upsilon_d(\mathcal U_{d,T}^A)
&=\frac{\Psi_d(T/2)+\Psi_{d+1}(T/2)}{2\sinh(T/2)}.
\label{eq:circular-random-ou-expected-top-volume}
\end{align}
Here $\Psi_m$ is defined in \eqref{eq:brownian-hyperbolic-remainders};
in particular $\Psi_0(x)=\cosh x$ and $\Psi_1(x)=\sinh x$. All the integrals are absolutely convergent.
\end{theorem}

\begin{proof}
\emph{Step 1: Recover the expected conic intrinsic volumes from Grassmann angles.}
Fix $T>0$ and $\bullet\in\{B,A\}$, and abbreviate
$\rho:=\rho_\bullet(T)$ and $C:=\mathcal U_{d,T}^\bullet$.
Suppose first that $d\ge2$.
By \eqref{eq:prelim-intrinsic-volumes-from-grassmann-angles} and
Remark~\ref{rem:prelim-conic-full-space-correction}, every
full-dimensional closed convex cone in $\R^d$ satisfies
\begin{align}
\upsilon_0(C)&=\frac{1-\gamma_1(C)}2,
\label{eq:random-ou-volume-grassmann-zeroth}
\\
\upsilon_k(C)&=\frac{\gamma_{k-1}(C)-\gamma_{k+1}(C)}2,
\qquad 1\le k\le d-2.
\label{eq:random-ou-volume-grassmann-lower}
\end{align}
At the top index, \eqref{eq:prelim-top-intrinsic-volume-full-space}
gives
\begin{equation}
\label{eq:random-ou-top-volume-full-space-correction}
\upsilon_d(C)
=\frac{\gamma_{d-1}(C)+\ind_{\{C=\R^d\}}}{2}.
\end{equation}
Taking expectations in
\eqref{eq:random-ou-volume-grassmann-zeroth} and
\eqref{eq:random-ou-volume-grassmann-lower}, and using
Lemma~\ref{lem:random-ou-cones-expected-grassmann-angles}, gives
\[
\Ee\,\upsilon_0(C)=\frac{1-a_1^\bullet(T)}2=L_0^\bullet(\rho),
\qquad
\Ee\,\upsilon_k(C)
=\frac{a_{k-1}^\bullet(T)-a_{k+1}^\bullet(T)}2
=L_k^\bullet(\rho),\quad 1\le k\le d-2.
\]
The coefficient equalities follow from
$a_j^\bullet(T)=q_j^\bullet(\rho)$ and
Corollary~\ref{cor:brownian-gaussian-absorption-consequences}.
For $k=1$, the equality uses $a_0^\bullet(T)=1$ and
$1-q_2^\bullet(\rho)=2L_1^\bullet(\rho)$.

At the top index,
\eqref{eq:random-ou-top-volume-full-space-correction} and
\eqref{eq:random-ou-cones-full-space-probability} yield
\begin{equation}
\label{eq:random-ou-expected-top-volume-absorption}
\Ee\,\upsilon_d(C)
=\frac{a_{d-1}^\bullet(T)+a_d^\bullet(T)}2
=\sum_{\ell=d}^\infty L_\ell^\bullet(\rho).
\end{equation}
For the last equality, use the series in
Theorem~\ref{thm:brownian-gaussian-absorption}: the two sums have
indices $d,d+2,\ldots$ and $d+1,d+3,\ldots$, respectively.

For the remaining index $k=d-1$, we use
\eqref{eq:prelim-next-to-top-intrinsic-volume-full-space},
which accounts for the possibility that $C=\R^d$.
Taking expectations and applying
Lemma~\ref{lem:random-ou-cones-expected-grassmann-angles}
and \eqref{eq:random-ou-cones-full-space-probability}, we obtain
\[
\Ee\,\upsilon_{d-1}(C)
=\frac{a_{d-2}^\bullet(T)-a_d^\bullet(T)}2
=L_{d-1}^\bullet(\rho).
\]
The last equality follows by subtracting the  series
in Theorem~\ref{thm:brownian-gaussian-absorption};
when $d=2$, we use $a_0^\bullet(T)=1$ and
\eqref{eq:finite-BA-limiting-parity-normalization}.

\medskip
If $d=1$, Lemma~\ref{lem:random-ou-cones-geometry} shows that $C$
is almost surely either a ray or $\R$. Their intrinsic-volume
sequences are $(1/2,1/2)$ and $(0,1)$, respectively. Thus
\[
\Ee\,\upsilon_0(C)=\frac{1-a_1^\bullet(T)}2=L_0^\bullet(\rho),
\qquad
\Ee\,\upsilon_1(C)=\frac{1+a_1^\bullet(T)}2
=1-L_0^\bullet(\rho).
\]

Combining the cases $d\ge2$ and $d=1$, we have shown that
$\Ee\,\upsilon_k(C)=L_k^\bullet(\rho)$ for every $0\le k<d$,
whereas
$\Ee\,\upsilon_d(C)=\sum_{\ell=d}^\infty L_\ell^\bullet(\rho)$.
This proves \eqref{eq:random-ou-expected-intrinsic-volume-coefficients} and~\eqref{eq:random-ou-expected-top-volume-absorption}
for all $d\ge1$. It remains to derive the explicit formulas
in terms of $T$.

\medskip
\emph{Step 2: Derivation of the explicit formulas.}
We now evaluate the right-hand side of
\eqref{eq:random-ou-expected-intrinsic-volume-coefficients}
using the explicit expressions for $L_k^\bullet$
from Theorem~\ref{thm:finite-BA-scaling-limit}.
For $k<d$, substitute \eqref{eq:ou-inverse-time-scales}
into these expressions, using $\arctan((\eee^T-1)^{-1/2})=\arcsin(\eee^{-T/2})$,
$\theta_{\rho_A(T)}=T/2$, and
$(1+4\rho_A(T)^2)^{-1/2}=\tanh(T/4)$.
For the top-dimensional intrinsic volume of the ordinary OU cone,
\eqref{eq:random-ou-expected-intrinsic-volume-coefficients} gives
$\Ee\,\upsilon_d(\mathcal U_{d,T}^B)
=\sum_{\ell=d}^\infty L_\ell^B(\rho_B(T))$.
By Tonelli's theorem, we may sum the nonnegative integral
representations in \eqref{eq:type-B-limiting-intrinsic-volumes}
over $\ell\ge d$. The sum of the factors in their numerators is
\[
\sum_{\ell=d}^\infty\frac{x^{\ell-1}}{(\ell-1)!}
=\Psi_{d-1}(x)+\Psi_d(x).
\]
Substituting $\rho=\rho_B(T)$ yields
\eqref{eq:ordinary-random-ou-expected-top-volume}.
Similarly, for the circular OU cone,
\[
\sum_{\ell=d}^\infty L_\ell^A(\rho_A(T))
=\frac{\Psi_d(\theta_{\rho_A(T)})
+\Psi_{d+1}(\theta_{\rho_A(T)})}
{2\sinh\theta_{\rho_A(T)}}.
\]
Since $\theta_{\rho_A(T)}=T/2$, this proves
\eqref{eq:circular-random-ou-expected-top-volume}.
Both calculations are valid for $d=1$.

Finally, all the integrals in
\eqref{eq:ordinary-random-ou-expected-intrinsic-volumes}
and \eqref{eq:ordinary-random-ou-expected-top-volume}
are finite: they represent individual coefficients $L_k^B$
and the sum of these coefficients over $\ell\ge d$,
respectively, and both quantities are bounded by one.
\end{proof}

\begin{corollary}[Expected solid angle of random OU cones]
\label{cor:random-ou-cones-expected-solid-angles}
For every $d\ge1$, $T>0$, and $\bullet\in\{B,A\}$, with
$a_0^\bullet(T)=1$,
\begin{align}
\Ee\,\alpha(\mathcal U_{d,T}^\bullet)
&=\frac{a_{d-1}^\bullet(T)+a_d^\bullet(T)}2,
\label{eq:random-ou-cones-expected-solid-angle}
\\
\Ee\bigl[\alpha(\mathcal U_{d,T}^\bullet);
             \mathcal U_{d,T}^\bullet\ne\R^d\bigr]
&=\frac{a_{d-1}^\bullet(T)-a_d^\bullet(T)}2,
\label{eq:random-ou-cones-truncated-solid-angle}
\\
\Ee\bigl[\alpha(\mathcal U_{d,T}^\bullet)
          \mid\mathcal U_{d,T}^\bullet\ne\R^d\bigr]
&=\frac{a_{d-1}^\bullet(T)-a_d^\bullet(T)}
        {2(1-a_d^\bullet(T))}.
\label{eq:random-ou-cones-conditional-solid-angle}
\end{align}
Here $\Ee[Y;A]:=\Ee[Y\ind_A]$. The event $\{\mathcal U_{d,T}^\bullet\ne\R^d\}$ has positive
probability.
\end{corollary}

\begin{proof}
By Lemma~\ref{lem:random-ou-cones-geometry}, the cone $\mathcal U_{d,T}^\bullet$ is
full-dimensional almost surely, so its solid angle is its $d$th
conic intrinsic volume. Equation
\eqref{eq:random-ou-expected-top-volume-absorption} proves
\eqref{eq:random-ou-cones-expected-solid-angle}.
Since $\alpha(\R^d)=1$ and
$\Pp(\mathcal U_{d,T}^\bullet=\R^d)=a_d^\bullet(T)$,
subtracting the contribution of this event gives
\eqref{eq:random-ou-cones-truncated-solid-angle}.
Projection onto one coordinate shows that
$a_d^\bullet(T)\le a_1^\bullet(T)<1$, where the strict inequality
follows from Corollary~\ref{cor:ou-origin-absorption-low-dimensions}.
Division by $1-a_d^\bullet(T)>0$ gives
\eqref{eq:random-ou-cones-conditional-solid-angle}.
\end{proof}

We now consider random OU cones in $\R^d$ in the
high-dimensional regime $d\to\infty$.
The following corollary shows that, in the critical scaling
window around $T=2d$, the random cone is asymptotically
either the whole space or a proper cone of negligible
solid angle. More precisely, the solid angle converges
in distribution to a Bernoulli random variable.

\begin{corollary}[Bernoulli limit for solid angles of random OU cones]
\label{cor:random-ou-cones-bernoulli-limit}
Let $(T_d)_{d\ge1}$ be positive and suppose that
\[
\frac{T_d-2d}{2\sqrt d}\longrightarrow\tau\in\R
\qquad\text{as }d\to\infty.
\]
Then, for both $\bullet\in\{B,A\}$,
\begin{equation}
\label{eq:random-ou-cones-bernoulli-limit}
\alpha(\mathcal U_{d,T_d}^\bullet)
\xrightarrow{\mathrm d}
\operatorname{Bern}(\Phi(\tau))
\qquad\text{and} \qquad 
\Ee\left|
\alpha(\mathcal U_{d,T_d}^\bullet)
-\ind_{\{\mathcal U_{d,T_d}^\bullet=\R^d\}}
\right|\longrightarrow0.
\end{equation}
\end{corollary}

\begin{proof}
Since $\alpha(\R^d)=1$ and
$0\le\alpha(\mathcal U_{d,T_d}^\bullet)\le1$,
equation~\eqref{eq:random-ou-cones-truncated-solid-angle}
gives
\[
\Ee\left|
\alpha(\mathcal U_{d,T_d}^\bullet)
-\ind_{\{\mathcal U_{d,T_d}^\bullet=\R^d\}}
\right|
=\frac{a_{d-1}^\bullet(T_d)-a_d^\bullet(T_d)}2
\longrightarrow0.
\]
Here the convergence follows from
Theorem~\ref{thm:ou-absorption-high-dimensional-transition},
applied in dimensions $d-1$ and $d$.
By \eqref{eq:random-ou-cones-full-space-probability},
the indicator
$\ind_{\{\mathcal U_{d,T_d}^\bullet=\R^d\}}$
has a Bernoulli distribution with success probability
$a_d^\bullet(T_d)\to\Phi(\tau)$.
The convergence in distribution follows.
\end{proof}

\section{Conic intrinsic volumes of Ornstein--Uhlenbeck spirals in a Hilbert space}
\label{sec:ornstein-uhlenbeck-spirals}
In this section we prove the results on Ornstein--Uhlenbeck
spiral cones announced in
Subsection~\ref{subsec:intro-ou-spirals-conic-intrinsic-volumes}.
We first construct the Brownian-motion and Brownian-bridge
bodies in Hilbert spaces and identify the cones over them,
up to isometry, with the corresponding OU spiral cones.
After recalling the definitions of conic intrinsic volumes
in Hilbert spaces, we determine these volumes by
finite-dimensional approximation, proving
Theorem~\ref{thm:intro-ou-spiral-intrinsic-volumes}
and Corollary~\ref{cor:intro-ou-spiral-grassmann-absorption}.
We also derive the generating functions of the conic intrinsic
volumes of the OU spiral cones.
Finally, we establish the general small-opening limit stated in
Theorem~\ref{thm:intro-small-opening-intrinsic-volumes}
and apply it to recover the ordinary intrinsic volumes of
the Brownian bodies given in
Corollary~\ref{cor:intro-hilbert-brownian-intrinsic-volumes}.

\subsection{Brownian bodies and Ornstein--Uhlenbeck spirals}
\label{subsec:brownian-bodies-ou-spirals}
Let $\mathbb H^B:=L^2[0,1]$ be the real Hilbert space of equivalence
classes of square-integrable real-valued functions on $[0,1]$
with respect to Lebesgue measure, equipped with the inner product
$\langle f,g\rangle:=\int_0^1 f(s)g(s)\,\dint s$
and the induced norm $\|\cdot\|_2$.
Let  $\mathbb H^A:=\{f\in \mathbb H^B:\int_0^1 f(s)\,\dint s=0\}$.
Adjoin a unit vector $e_0$ orthogonal to $\mathbb H^B$, and put
$E^\bullet:=\R e_0\oplus \mathbb H^\bullet$ for $\bullet\in\{B,A\}$.
All closures below are in the norm topology. For $0\le t\le1$, define
\begin{equation}
\label{eq:hilbert-brownian-generators}
h_t^B:=\ind_{[0,t]}\in \mathbb H^B,
\qquad
h_t^A:=\ind_{[0,t]}-t\ind_{[0,1]} \in \mathbb H^A.
\end{equation}
Values at individual endpoints are immaterial in $L^2[0,1]$;
in particular, $h_0^B=h_0^A=h_1^A=0$.
The two Gram kernels are
\begin{equation}
\label{eq:hilbert-brownian-gram-kernels}
\langle h_s^B,h_t^B\rangle=\min(s,t),
\qquad
\langle h_s^A,h_t^A\rangle=\min(s,t)-st,
\qquad 0\le s,t\le1.
\end{equation}
Define the \emph{Brownian-motion body} $\bK_\infty^B$ and the \emph{Brownian-bridge
body} $\bK_\infty^A$ by
\begin{equation}
\label{eq:hilbert-brownian-bodies}
\bK_\infty^\bullet:=\convclos\{h_t^\bullet:0\le t\le1\}
\subseteq \mathbb H^\bullet,
\qquad \bullet\in\{B,A\}.
\end{equation}
These bodies were studied by Gao--Vitale~\cite{GaoVitale2001}
and Gao~\cite{Gao2003}, respectively, who determined their
ordinary intrinsic volumes. For $\rho>0$, define the cones over these Brownian bodies  by
\begin{equation}
\label{eq:hilbert-brownian-cones}
\bW_\infty^\bullet(\rho)
:=\pos\bigl(\rho e_0+\bK_\infty^\bullet\bigr)
\subseteq E^\bullet.
\end{equation}
The orthogonal projection $P_\infty:E^B\to E^A$ is given by
\begin{equation}
\label{eq:hilbert-centering-projection}
P_\infty(ae_0+f)
=ae_0+f-\left(\int_0^1 f(s)\,\dint s\right)\ind_{[0,1]}.
\end{equation}

\begin{lemma}[Basic properties of the Brownian bodies and their cones]
\label{lem:hilbert-brownian-bodies-and-cones}
The bodies $\bK_\infty^B$ and $\bK_\infty^A$ are compact. For each
$\rho>0$ and $\bullet\in\{B,A\}$, the cone
$\bW_\infty^\bullet(\rho)$ is closed and pointed, has closed linear
span $E^\bullet$, and has compact intersection with every closed ball.
Moreover,
\begin{equation}
\label{eq:hilbert-cone-base-and-curve}
\bW_\infty^\bullet(\rho)
=\{\lambda(\rho e_0+f):
       \lambda\ge0,\ f\in\bK_\infty^\bullet\}
=\posclos\{\rho e_0+h_t^\bullet:0\le t\le1\},
\end{equation}
and the projection satisfies
\begin{equation}
\label{eq:hilbert-centering-bodies-and-cones}
P_\infty\bK_\infty^B=\bK_\infty^A,
\qquad
P_\infty\bW_\infty^B(\rho)=\bW_\infty^A(\rho).
\end{equation}
\end{lemma}

\begin{proof}
By \eqref{eq:hilbert-brownian-gram-kernels},
$\|h_t^B-h_s^B\|_2^2=|t-s|$ and
$\|h_t^A-h_s^A\|_2^2=|t-s|(1-|t-s|)$.
Thus both curves are norm-continuous and have compact ranges.
A finite $\varepsilon$-net of a compact set approximates each of
its finite convex combinations, within $\varepsilon$, by a point
in the convex hull of that net. This latter hull is a compact
finite-dimensional polytope and has a finite $\varepsilon$-net of
its own. Hence the closed convex hulls in
\eqref{eq:hilbert-brownian-bodies} are totally bounded; completeness
proves compactness.

Convexity of the base gives the first equality in
\eqref{eq:hilbert-cone-base-and-curve}.
The $e_0$-coordinate of $\lambda(\rho e_0+f)$ is $\lambda\rho$,
so the cone is pointed. For $R>0$, its intersection with the closed
ball of radius $R$ centered at zero is the intersection of that ball
with the compact set
\[
\{\lambda(\rho e_0+f):
   0\le\lambda\le R/\rho,\ f\in\bK_\infty^\bullet\}.
\]
It is therefore compact. This also proves that the cone is closed
and that its intersection with any closed ball is compact.
Approximating $f\in\bK_\infty^\bullet$ by finite convex combinations
of the generators proves the second equality in
\eqref{eq:hilbert-cone-base-and-curve}.

Since $P_\infty h_t^B=h_t^A$ for every $t\in[0,1]$,
the compact image $P_\infty\bK_\infty^B$ equals the
closed convex hull of the functions $h_t^A$, $0\le t\le1$,
and hence coincides with $\bK_\infty^A$. This proves the first identity in~\eqref{eq:hilbert-centering-bodies-and-cones}; the second
follows by linearity and $P_\infty e_0=e_0$.
Finally, the indicators $h_t^B$ have dense linear span in $\mathbb H^B$,
and their centered projections have dense linear span in $\mathbb H^A$.
Each cone contains $\rho e_0$ because $h_0^\bullet=0$; subtracting
this vector from the other generators shows that its closed linear
span is $E^\bullet$.
\end{proof}

The following remarks record alternative descriptions
of the Brownian bodies. Since they are not needed in the
subsequent proofs, we state these properties without proof.

\begin{remark}[Brownian bodies as classes of functions]
\label{rem:hilbert-brownian-function-descriptions}
The two Brownian bodies have the following descriptions, parallel to
\eqref{eq:type-B-body-coordinates} and
\eqref{eq:type-A-body-coordinates}:
\begin{align}
\bK_\infty^B
&=\left\{f\in \mathbb H^B:
\begin{array}{l}
f\text{ has a nonincreasing representative on }(0,1),\\
0\le f\le1\text{ almost everywhere}
\end{array}\right\},
\label{eq:type-B-hilbert-body-functions}
\\
\bK_\infty^A
&=\left\{f\in \mathbb H^A:
\begin{array}{l}
f\text{ has a nonincreasing representative on }(0,1),\\
\operatorname*{ess\,sup} f-\operatorname*{ess\,inf} f\le1
\end{array}\right\}.
\label{eq:type-A-hilbert-body-functions}
\end{align}
The essential bounds are taken over $(0,1)$.
The projection in \eqref{eq:hilbert-centering-projection} transforms
the first description into the second.
\end{remark}

\begin{remark}[Brownian bodies as Wasserstein spaces]
\label{rem:hilbert-brownian-wasserstein}
The preceding remark  yields a Wasserstein-space interpretation of Brownian bodies.
For a Borel probability measure $\nu$ on $\R$, let
\[
Q_\nu(u):=\inf\{x\in\R:\nu((-\infty,x])\ge u\},
\qquad 0<u<1,
\]
denote its quantile function. For probability measures
$\mu,\nu$ on $\R$ with finite second moments, the one-dimensional
$2$-Wasserstein distance is given by
\[
W_2(\mu,\nu)^2
=\int_0^1|Q_\mu(u)-Q_\nu(u)|^2\,\dint u.
\]
The \emph{reversed quantile function} of $\nu$ is the
nonincreasing function $s\mapsto Q_\nu(1-s)$, $0<s<1$.
The correspondence $\nu\mapsto Q_\nu(1-\cdot)$ is an
isometry from the $2$-Wasserstein space of probability
measures on $[0,1]$ onto $\bK_\infty^B$.
The same correspondence identifies $\bK_\infty^A$
isometrically with the space of probability measures
on $\R$ having mean zero and support of diameter at
most one, equipped with the $2$-Wasserstein distance.
Under these identifications, the orthogonal projection
$P_\infty$ corresponds to centering a probability measure
by subtracting its mean.
\end{remark}

\begin{remark}[Bochner-integral representations]
\label{rem:hilbert-brownian-bochner-representations}
Let $\mathcal P([0,1])$ denote the set of Borel probability measures on
$[0,1]$. Every element of $\bK_\infty^\bullet$ has a representation
\begin{equation}
\label{eq:hilbert-brownian-bochner-representations}
f_\mu^\bullet:=\int_{[0,1]}h_t^\bullet\,\mu(\dint t),
\qquad
\bK_\infty^\bullet
=\{f_\mu^\bullet:\mu\in\mathcal P([0,1])\},
\qquad \bullet\in\{B,A\}.
\end{equation}
The integrals are Bochner integrals in $\mathbb H^\bullet$; their existence
follows from norm-continuity and boundedness of the generating curves.
Explicitly, for almost every $s\in(0,1)$,
\begin{equation}
\label{eq:hilbert-brownian-bochner-tail-functions}
f_\mu^B(s)=\mu([s,1]),
\qquad
f_\mu^A(s)=\mu([s,1])-\int_{[0,1]}t\,\mu(\dint t).
\end{equation}
In type $B$, the representing probability measure $\mu$ is unique.
In type $A$, its restriction to $(0,1)$ is uniquely determined;
the only possible ambiguity is the division of the remaining mass
between $0$ and $1$, since $h_0^A=h_1^A=0$. The measure $\mu$ in the Bochner representation is generally
different from the measure whose reversed quantile function
represents the same element in the preceding Wasserstein
interpretation.
\end{remark}

We next relate the cones over the Brownian bodies  to Ornstein--Uhlenbeck spirals. For $T>0$,
by an \emph{ordinary Ornstein--Uhlenbeck spiral} on $[0,T]$ and a
\emph{circular Ornstein--Uhlenbeck spiral} of circumference $T$ we
mean curves $(\zeta_T^B(u))_{u\in [0,T]}$ and $(\zeta_T^A(u))_{u\in [0,T]}$ in real Hilbert spaces with
Gram kernels
\begin{align}
\langle\zeta_T^B(u),\zeta_T^B(v)\rangle
&=\eee^{-|u-v|/2},
\qquad 0\le u,v\le T,
\label{eq:ordinary-ou-spiral-kernel}
\\
\langle\zeta_T^A(u),\zeta_T^A(v)\rangle
&=\frac{\cosh(T/4-|u-v|/2)}{\cosh(T/4)},
\qquad 0\le u,v\le T,
\label{eq:circular-ou-spiral-kernel}
\end{align}
respectively. These are norm-continuous curves in the unit spheres of the respective Hilbert spaces,  and
$\zeta_T^A(0)=\zeta_T^A(T)$.
A realization of the OU spirals  is obtained by taking the processes from
Subsection~\ref{subsec:ou-origin-absorption}, specialized to the case $d=1$, and viewing their random
variables as vectors in $L^2(\Omega,\mathcal F,\Pp)$.
Note that an OU spiral is a deterministic Hilbert-space curve, not a random
sample path in a Euclidean space.
We define the closed cones $\bU_{\infty,T}^\bullet$ generated by the OU spirals and take
the closed linear spans of the OU spirals as their
ambient Hilbert spaces $\mathbb G_T^\bullet$:
\begin{equation}
\label{eq:ou-spiral-spans-and-cones}
\bU_{\infty, T}^\bullet
:=\posclos\{\zeta_T^\bullet(u):0\le u\le T\},
\qquad
\mathbb G_T^\bullet
:=\overline{\lin\{\zeta_T^\bullet(u):0\le u\le T\}},
\qquad \bullet\in\{B,A\}.
\end{equation}
We refer to $\bU_{\infty, T}^B$ and $\bU_{\infty, T}^A$ as the
\emph{OU spiral cones}. Note that $\bU_{\infty, T}^\bullet\subseteq\mathbb G_T^\bullet$.

The next proposition shows that, after suitable time changes, the cones over the Brownian bodies
become isometric to OU spiral cones.

\begin{proposition}[Isometric identification of Brownian cones and OU spiral cones]
\label{prop:hilbert-brownian-cones-ou-spirals}
For $\rho>0$, let $T_\bullet(\rho)$ and $u_\rho^\bullet$ be as in
\eqref{eq:ou-time-lengths}, \eqref{eq:ordinary-ou-time-change}, and
\eqref{eq:circular-ou-time-change}. Define the unit-norm curves by
\begin{equation}
\label{eq:hilbert-normalized-ou-generators}
\eta_\rho^\bullet\bigl(u_\rho^\bullet(t)\bigr)
:=\frac{\rho e_0+h_t^\bullet}
        {\sqrt{\rho^2+\|h_t^\bullet\|_2^2}},
\qquad 0\le t\le1,\quad \bullet\in\{B,A\}.
\end{equation}
Then $(\eta_\rho^B(u))_{0\le u\le T_B(\rho)}$ is an ordinary
Ornstein--Uhlenbeck spiral, and
$(\eta_\rho^A(u))_{0\le u\le T_A(\rho)}$ is a circular
Ornstein--Uhlenbeck spiral of circumference $T_A(\rho)$.  Their closed positive hulls are
$\bW_\infty^\bullet(\rho)$.
There is a surjective linear isometry
$\mathcal J_\rho^\bullet:E^\bullet\to
\mathbb G_{T_\bullet(\rho)}^\bullet$ mapping
$\eta_\rho^\bullet(u)$ to $\zeta_{T_\bullet(\rho)}^\bullet(u)$ for
every $u\in[0,T_\bullet(\rho)]$, and hence
\begin{equation}
\label{eq:hilbert-brownian-ou-cone-isometry}
\mathcal J_\rho^\bullet\bigl(\bW_\infty^\bullet(\rho)\bigr)
=\bU_{\infty, T_\bullet(\rho)}^\bullet.
\end{equation}
In particular, for every $T>0$, the cone $\bU_{\infty, T}^\bullet$ is isometric
to $\bW_\infty^\bullet(\rho_\bullet(T))$, with
$\rho_\bullet(T)$  from \eqref{eq:ou-inverse-time-scales}, that is
\begin{equation}
\label{eq:ou-inverse-time-scales-repeat}
\rho_B(T)=\frac1{\sqrt{\eee^T-1}},
\qquad
\rho_A(T)=\frac1{2\sinh(T/4)}.
\end{equation}
\end{proposition}

\begin{proof}
By \eqref{eq:hilbert-brownian-gram-kernels},
$\|h_t^B\|_2^2=t$ and $\|h_t^A\|_2^2=t(1-t)$.
The inner products of the unit vectors in
\eqref{eq:hilbert-normalized-ou-generators} are therefore
exactly the covariance expressions in
\eqref{eq:ordinary-ou-brownian-covariance} and
\eqref{eq:circular-ou-bridge-covariance},
specialized to the case $d=1$.
Those identities and the bijective time changes give the kernels
\eqref{eq:ordinary-ou-spiral-kernel} and
\eqref{eq:circular-ou-spiral-kernel}, with $T=T_\bullet(\rho)$.
Since all normalizing terms $\sqrt{\rho^2+\|h_t^\bullet\|_2^2}$ are positive, we can drop them without changing the positive hulls.
Equation~\eqref{eq:hilbert-cone-base-and-curve} proves the assertion about the closed positive hulls.

Finally, equality of the Gram kernels makes the map sending each
$\eta_\rho^\bullet(u)$ to $\zeta_{T_\bullet(\rho)}^\bullet(u)$ a
well-defined linear isometry on their linear spans: the squared norm
of any finite linear combination is determined by the Gram kernel.
It extends to a linear isometry between their closed spans.
The source span is $E^\bullet$ by
Lemma~\ref{lem:hilbert-brownian-bodies-and-cones}.
Taking closed positive hulls gives
\eqref{eq:hilbert-brownian-ou-cone-isometry}, and substituting
$\rho=\rho_\bullet(T)$ proves the last assertion.
\end{proof}

\subsection{Conic intrinsic volumes in Hilbert spaces}
\label{subsec:hilbert-conic-intrinsic-volumes}
We recall the definitions of Grassmann angles and conic intrinsic
volumes for cones in Hilbert spaces, together with an approximation
result that will allow us to determine these volumes for the
OU spiral cones introduced in
Subsection~\ref{subsec:brownian-bodies-ou-spirals}.
Let $\mathbb H$ be a real separable Hilbert space. Throughout this subsection,
cones are nonempty, closed, and convex.
For a finite-dimensional cone $C_0\subseteq \mathbb H$, its Grassmann angles
are those computed in $\lin C_0$, extended by zero for
$j\ge\dim(\lin C_0)$. They do not depend on the choice of a
finite-dimensional ambient space containing $C_0$.

Following Dospolova~\cite[Section~2, equation~(7)]{Dospolova2023},
for a cone $C\subseteq \mathbb H$ define
\begin{equation}
\label{eq:hilbert-space-grassmann-angle}
\gamma_j(C)
:=\sup\left\{\gamma_j(C_0):
 C_0\subseteq C\text{ is a finite-dimensional closed convex cone}
 \right\},
\qquad j\ge0.
\end{equation}
This definition is also recalled in
\cite[Section~3, equation~(8)]{DospolovaZaporozhets2024}.
If $C$ is finite-dimensional, it is itself admissible in the
supremum; monotonicity of the finite-dimensional Grassmann angles
then shows that \eqref{eq:hilbert-space-grassmann-angle} agrees with
their usual definition. The quantities $\gamma_j(C)$ are invariant
under linear isometries and monotone under inclusion of cones.
For fixed $C$, they form a nonincreasing sequence in $[0,1]$.
Moreover, $\gamma_0(C)=1$ when $C\ne\{0\}$, whereas
$\gamma_j(\{0\})=0$ for every $j\ge0$.

For a cone $C\subseteq \mathbb H$ which is not a linear subspace, define its conic
intrinsic volumes, following
Dospolova--Zaporozhets~\cite[Section~3, equations~(9)--(10)]{DospolovaZaporozhets2024},
by
\begin{equation}
\label{eq:hilbert-space-conic-intrinsic-volumes}
\upsilon_0(C):=1-\frac{\gamma_0(C)+\gamma_1(C)}2,
\qquad
\upsilon_k(C):=\frac{\gamma_{k-1}(C)-\gamma_{k+1}(C)}2,
\quad k\ge1.
\end{equation}
These numbers are nonnegative and invariant under linear isometries.
For finite-dimensional cones that are not linear subspaces,
they agree with the usual conic intrinsic volumes by
\eqref{eq:prelim-intrinsic-volumes-from-grassmann-angles}.

The cones $\bW_\infty^\bullet(\rho)$ are nonzero and pointed by
Lemma~\ref{lem:hilbert-brownian-bodies-and-cones}.
By Proposition~\ref{prop:hilbert-brownian-cones-ou-spirals},
the OU spiral cones $\bU_{\infty, T}^\bullet$ have the same properties.
Consequently, \eqref{eq:hilbert-space-conic-intrinsic-volumes}
applies to both families of cones.

\begin{remark}[Linear subspaces]
\label{rem:hilbert-conic-intrinsic-volumes-subspaces}
For a finite-dimensional linear subspace $L\subseteq \mathbb H$, we retain the classical
convention $\upsilon_k(L)=\ind_{\{k=\dim L\}}$, including
$\upsilon_0(\{0\})=1$. Formula
\eqref{eq:hilbert-space-conic-intrinsic-volumes} is not used to define
these values for subspaces.
\end{remark}

\begin{remark}[Total mass]
\label{rem:hilbert-conic-intrinsic-volumes-total-mass}
For a cone $C\subseteq \mathbb H$ which is not a linear subspace, summing
\eqref{eq:hilbert-space-conic-intrinsic-volumes} gives
\begin{equation}
\label{eq:hilbert-conic-intrinsic-volume-partial-sums}
\sum_{k=0}^N\upsilon_k(C)
=1-\frac{\gamma_N(C)+\gamma_{N+1}(C)}2,
\qquad N\ge0.
\end{equation}
Consequently, the total mass $\sum_{k=0}^\infty \upsilon_k(C)$ is at most one, and it equals one if
and only if $\gamma_N(C)\to0$ as $N\to\infty$; compare
\cite[Section~3, Remark~4]{DospolovaZaporozhets2024}.
In infinite dimensions, the equality
$\sum_{k=0}^\infty\upsilon_k(C)=1$ need not hold,
even for nonzero closed pointed cones.
For example, consider the nonnegative orthant
\[
C:=\{(x_i)_{i\ge1}\in\ell^2:
x_i\ge0\text{ for all }i\}.
\]
For the $n$-dimensional orthant,
$\upsilon_k(\R_+^n)=2^{-n}\binom nk$, $0\le k\le n$.
By the finite-dimensional Crofton formula
\eqref{eq:finite-dimensional-conic-crofton},
$\gamma_j(\R_+^n)\to1$ for every fixed $j\ge0$,
since the sums of these intrinsic volumes over even and odd
indices are both $1/2$, whereas each fixed-index term
tends to zero.
Hence $\gamma_j(C)=1$ for all $j\ge0$, so
$\upsilon_k(C)=0$ for every $k\ge0$ and therefore
$\sum_{k=0}^\infty\upsilon_k(C)=0$.
\end{remark}

To determine the conic intrinsic volumes of the OU spiral cones, we use the following result on monotone approximation from inside; see
\cite[Section~4, Lemma~2 and the subsequent argument on
pp.~112--113]{DospolovaZaporozhets2024}, with the underlying
Grassmann-angle approximation proved in
\cite[Section~3, Assertion~2 and the subsequent argument]{Dospolova2023}.

\begin{proposition}[Approximation from inside]
\label{prop:hilbert-cone-inner-approximation}
Let $C\subseteq \mathbb H$ be a closed convex cone, and let
$C_1,C_2,\ldots\subseteq \mathbb H$ be finite-dimensional closed convex cones
such that
\[
C_1\subseteq C_2\subseteq\cdots\subseteq C,
\qquad
\overline{\bigcup_{n\ge1}C_n}=C,
\]
where the closure is in the norm topology. Then, for every $j\ge0$,
\begin{equation}
\label{eq:hilbert-cone-grassmann-angle-approximation}
\gamma_j(C_n)\uparrow\gamma_j(C)
\qquad\text{as }n\to\infty.
\end{equation}
If $C$ and all $C_n$ are not linear subspaces, then, for every $k\ge0$,
\begin{equation}
\label{eq:hilbert-cone-intrinsic-volume-approximation}
\upsilon_k(C_n)\longrightarrow\upsilon_k(C)
\qquad\text{as }n\to\infty.
\end{equation}
\end{proposition}

The intrinsic-volume conclusion~\eqref{eq:hilbert-cone-intrinsic-volume-approximation} follows directly from
\eqref{eq:hilbert-space-conic-intrinsic-volumes} and
\eqref{eq:hilbert-cone-grassmann-angle-approximation}.
Only the Grassmann angles are asserted to converge monotonically.

\subsection{Finite-dimensional approximation and conic intrinsic volumes}
\label{subsec:hilbert-finite-approximation-intrinsic-volumes}
We approximate the Brownian bodies $\bK_\infty^B$ and
$\bK_\infty^A$ by their grid hulls $\widehat{\bK}_n^B$ and
$\widehat{\bK}_n^A$, which we introduce below, and identify
these hulls, up to isometry, with the rescaled finite-dimensional
bodies $n^{-1/2}\bK_n^B$ and $n^{-1/2}\bK_n^A$ from
Subsection~\ref{subsec:type-B-bodies-and-cones}.
Combining the approximation result in
Proposition~\ref{prop:hilbert-cone-inner-approximation}
with the scaling limit in
Theorem~\ref{thm:finite-BA-scaling-limit},
we determine the conic intrinsic volumes of the Brownian cones
$\bW_\infty^\bullet(\rho)$.
The isometry of
Proposition~\ref{prop:hilbert-brownian-cones-ou-spirals}
then gives the corresponding intrinsic volumes of the
Ornstein--Uhlenbeck spiral cones.

For $n\ge1$, put $I_{n,j}:=((j-1)/n,j/n]$, $1\le j\le n$, and let
$\mathbb H_n^B\subseteq \mathbb H^B = L^2[0,1]$ be the space of functions that are constant
almost everywhere on each interval $I_{n,j}$. Set $\mathbb H_n^A:=\mathbb H_n^B\cap \mathbb H^A$ and
$\widehat E_n^\bullet:=\R e_0\oplus \mathbb H_n^\bullet$.
Let $P_n:E^B\to\widehat E_n^B$ be the orthogonal projection given by
\begin{equation}
\label{eq:hilbert-grid-projection}
P_n(ae_0+f)
=ae_0+\sum_{j=1}^n
\left(n\int_{(j-1)/n}^{j/n}f(s)\,\dint s\right)\ind_{I_{n,j}}.
\end{equation}
This projection preserves the constant function and the integral of
$f$, so it commutes with $P_\infty$ from
\eqref{eq:hilbert-centering-projection}. Its restriction to $E^A$ is
the orthogonal projection onto $\widehat E_n^A$.
Define the grid bodies and cones by
\begin{align}
\widehat{\bK}_n^B
&:=\conv\{h_{i/n}^B:i=0,\ldots,n\},
\label{eq:type-B-hilbert-grid-body}
\\
\widehat{\bK}_n^A
&:=\conv\{h_{i/n}^A:i=0,\ldots,n-1\},
\label{eq:type-A-hilbert-grid-body}
\\
\widehat{\bW}_n^\bullet(\rho)
&:=\pos\bigl(\rho e_0+\widehat{\bK}_n^\bullet\bigr)
\subseteq\widehat E_n^\bullet,
\qquad \rho>0,\quad\bullet\in\{B,A\}.
\label{eq:hilbert-grid-cones}
\end{align}
In the type-$A$ case, the index $i=n$ may be included without changing the
body, since $h_1^A=h_0^A=0$.

\medskip
The next proposition identifies, up to isometry, the grid hulls
$\widehat{\bK}_n^B$ and $\widehat{\bK}_n^A$ with the rescaled bodies
$n^{-1/2}\bK_n^B$ and $n^{-1/2}\bK_n^A$ from
Subsection~\ref{subsec:type-B-bodies-and-cones},
and the grid cones $\widehat{\bW}_n^\bullet(\rho)$ with
the finite-dimensional cones $\bW_n^\bullet(\rho\sqrt n)$.
It also provides a uniform approximation bound for the
Brownian bodies.

\begin{proposition}[Grid approximation]
\label{prop:hilbert-brownian-grid-approximation}
For every $n\ge1$, $\rho>0$, and $\bullet\in\{B,A\}$,
\begin{equation}
\label{eq:hilbert-grid-projection-images}
P_n\bK_\infty^\bullet=\widehat{\bK}_n^\bullet,
\qquad
P_n\bW_\infty^\bullet(\rho)=\widehat{\bW}_n^\bullet(\rho),
\end{equation}
and
\begin{equation}
\label{eq:hilbert-grid-uniform-approximation}
\sup_{f\in\bK_\infty^\bullet}\|f-P_nf\|_2
\le\frac1{2\sqrt n}.
\end{equation}
The linear isometry
$\iota_n:\R e_0\oplus\R^n\to\widehat E_n^B$ defined by
\begin{equation}
\label{eq:hilbert-grid-isometry}
\iota_n e_0=e_0,
\qquad
\iota_n e_j=\sqrt n\,\ind_{I_{n,j}},\quad 1\le j\le n,
\end{equation}
restricts to an isometry from $\R e_0\oplus s_n^\perp$ onto
$\widehat E_n^A$, and satisfies
\begin{align}
\iota_n\bigl(n^{-1/2}\bK_n^\bullet\bigr)
&=\widehat{\bK}_n^\bullet,
\label{eq:hilbert-grid-body-isometry}
\\
\iota_n\bigl(\bW_n^\bullet(\rho\sqrt n)\bigr)
&=\widehat{\bW}_n^\bullet(\rho).
\label{eq:hilbert-grid-cone-isometry}
\end{align}
In particular, the grid cones $\widehat{\bW}_n^\bullet(\rho)$ are closed, nonzero, and pointed, with
dimensions $n+1$ in type $B$ and $n$ in type $A$.
\end{proposition}

\begin{proof}
For $i\in\{0,\ldots,n-1\}$ and $t\in[i/n,(i+1)/n]$, put
$r:=nt-i\in[0,1]$. Averaging the indicator $h_t^B$ gives
\begin{equation}
\label{eq:hilbert-grid-generator-interpolation}
P_nh_t^\bullet
=(1-r)h_{i/n}^\bullet+r h_{(i+1)/n}^\bullet,
\qquad \bullet\in\{B,A\}.
\end{equation}
For type $A$, the same identity follows by subtracting
$t\ind_{[0,1]}$, since
$(1-r)i/n+r(i+1)/n=t$ and $P_n$ fixes the constant function.
Thus the image of every generator $h_t^\bullet$ under $P_n$ belongs to the corresponding
grid body $\widehat{\bK}_n^\bullet$, and every grid generator $h_{i/n}^\bullet$ is fixed by $P_n$.
Taking closed convex hulls proves the first identity in
\eqref{eq:hilbert-grid-projection-images}; here the grid bodies
are finite-dimensional polytopes and hence closed.
The cone identity in \eqref{eq:hilbert-grid-projection-images} follows from linearity, $P_ne_0=e_0$, and
\eqref{eq:hilbert-cone-base-and-curve}.

The error $h_t^B-P_nh_t^B$ is supported on $I_{n,i+1}$, and
\[
\|h_t^B-P_nh_t^B\|_2^2
=\frac{r(1-r)^2+(1-r)r^2}{n}
=\frac{r(1-r)}n\le\frac1{4n}.
\]
Moreover, $h_t^A-P_nh_t^A=h_t^B-P_nh_t^B$.
The triangle inequality extends the bound $1/(2\sqrt n)$ to every
finite convex combination of the generators; continuity then
extends it to their closed convex hulls. This proves
\eqref{eq:hilbert-grid-uniform-approximation}.

The vectors $\sqrt n\,\ind_{I_{n,j}}$, $1\le j\le n$, form an
orthonormal basis of $\mathbb H_n^B$, so $\iota_n$ is an isometry.
Since $\iota_ns_n=\sqrt n\,\ind_{[0,1]}$, it maps
$s_n^\perp$ onto $\mathbb H_n^A$. The partial sums satisfy
\[
\iota_ns_i=\sqrt n\,h_{i/n}^B,
\qquad
\iota_n\left(s_i-\frac in s_n\right)
=\sqrt n\,h_{i/n}^A.
\]
Consequently, the definitions of the finite-dimensional bodies
and their grid counterparts give
\eqref{eq:hilbert-grid-body-isometry}.
Since
\[
\bW_n^\bullet(\rho\sqrt n)
=\pos\bigl(\rho e_0+n^{-1/2}\bK_n^\bullet\bigr),
\]
the cone identity
\eqref{eq:hilbert-grid-cone-isometry}
follows from linearity and $\iota_ne_0=e_0$.
The remaining assertions follow from the corresponding properties
of the finite cones established in
Subsection~\ref{subsec:type-B-bodies-and-cones}.
\end{proof}

\begin{theorem}[Conic intrinsic volumes of the OU spiral cones]
\label{thm:hilbert-brownian-ou-intrinsic-volumes}
For every integer $k\ge0$, every $\rho,T>0$, and
$\bullet\in\{B,A\}$,
\begin{align}
\upsilon_k\bigl(\bW_\infty^\bullet(\rho)\bigr)
&=L_k^\bullet(\rho),
\label{eq:hilbert-brownian-intrinsic-volumes}
\\
\upsilon_k(\bU_{\infty, T}^\bullet)
&=L_k^\bullet(\rho_\bullet(T)),
\label{eq:ou-spiral-intrinsic-volume-identification}
\end{align}
where $L_k^\bullet$ is defined in
Theorem~\ref{thm:finite-BA-scaling-limit} and $\rho_\bullet(T)$ in
\eqref{eq:ou-inverse-time-scales-repeat}.
Explicitly, for the ordinary Ornstein--Uhlenbeck spiral,
\begin{align}
\upsilon_0(\bU_{\infty, T}^B)
&=\frac1\pi\arcsin(\eee^{-T/2}),
\label{eq:ordinary-ou-spiral-zeroth-intrinsic-volume}
\\
\upsilon_k(\bU_{\infty, T}^B)
&=\frac1{\pi(k-1)!}
\int_0^{\arsinh\sqrt{\eee^T-1}}
\frac{x^{k-1}\,\dint x}{\sqrt{\eee^T-1-\sinh^2x}},
\qquad k\ge1.
\label{eq:ordinary-ou-spiral-intrinsic-volumes}
\end{align}
For the circular Ornstein--Uhlenbeck spiral,
\begin{align}
\upsilon_0(\bU_{\infty, T}^A)
&=\frac12\left(1-\tanh\frac T4\right),
\label{eq:circular-ou-spiral-zeroth-intrinsic-volume}
\\
\upsilon_k(\bU_{\infty, T}^A)
&=\frac{(T/2)^k}{2k!\sinh(T/2)},
\qquad k\ge1.
\label{eq:circular-ou-spiral-intrinsic-volumes}
\end{align}
For each of these cones, the sums of the intrinsic volumes over
the even and odd indices are both $1/2$; in particular, the total
mass is one.
\end{theorem}

\begin{proof}
Fix $\rho>0$ and $\bullet\in\{B,A\}$.
The dyadic grid cones $\widehat{\bW}_{2^m}^\bullet(\rho)$,
$m\ge0$, are nested and contained in $\bW_\infty^\bullet(\rho)$,
since each dyadic grid is contained in the next one.
For type $A$, the omitted endpoint generator is already present
at time zero. Every $x\in\bW_\infty^\bullet(\rho)$ can be written
as $x=\lambda(\rho e_0+f)$ with $\lambda\ge0$ and
$f\in\bK_\infty^\bullet$. By
\eqref{eq:hilbert-grid-projection-images} and
\eqref{eq:hilbert-grid-uniform-approximation},
$P_{2^m}x\in\widehat{\bW}_{2^m}^\bullet(\rho)$ and
$\|x-P_{2^m}x\|\le\lambda/(2\sqrt{2^m})\to0$.
Thus
\[
\overline{\bigcup_{m\ge0}\widehat{\bW}_{2^m}^\bullet(\rho)}
=\bW_\infty^\bullet(\rho).
\]
All these cones are nonzero and pointed by
Lemma~\ref{lem:hilbert-brownian-bodies-and-cones} and
Proposition~\ref{prop:hilbert-brownian-grid-approximation}.
Proposition~\ref{prop:hilbert-cone-inner-approximation}, followed
by \eqref{eq:hilbert-grid-cone-isometry}, therefore gives
\[
\upsilon_k\bigl(\bW_\infty^\bullet(\rho)\bigr)
=\lim_{m\to\infty}
  \upsilon_k\bigl(\widehat{\bW}_{2^m}^\bullet(\rho)\bigr)
=\lim_{m\to\infty}
  \upsilon_k\bigl(\bW_{2^m}^\bullet(\rho\sqrt{2^m})\bigr)
=L_k^\bullet(\rho).
\]
The last equality is
Theorem~\ref{thm:finite-BA-scaling-limit} with $\rho_n=\rho\sqrt n$
and $c=\rho$, restricted to the dyadic subsequence.
This proves \eqref{eq:hilbert-brownian-intrinsic-volumes}.
Equation~\eqref{eq:ou-spiral-intrinsic-volume-identification} follows from
Proposition~\ref{prop:hilbert-brownian-cones-ou-spirals} and
isometry invariance of the conic intrinsic volumes.

The explicit formulas for the conic intrinsic volumes
$\upsilon_k(\bU_{\infty,T}^\bullet)$ follow by substituting
$c=\rho_\bullet(T)$ from \eqref{eq:ou-inverse-time-scales-repeat}
into the expressions for $L_k^\bullet(c)$ in
Theorem~\ref{thm:finite-BA-scaling-limit}.
For type $B$ and $k=0$, we use
$\arctan((\eee^T-1)^{-1/2})=\arcsin(\eee^{-T/2})$.
For type $A$, we use $\theta_{\rho_A(T)}=T/2$ and
$(1+4\rho_A(T)^2)^{-1/2}=\tanh(T/4)$.
The equalities between the sums of the intrinsic volumes over
even and odd indices follow from
\eqref{eq:finite-BA-limiting-parity-normalization}.
\end{proof}

\begin{corollary}[Grassmann angles of the OU spiral cones]
\label{cor:hilbert-ou-grassmann-absorption}
For every integer $d\ge1$, every $\rho,T>0$, and
$\bullet\in\{B,A\}$,
\begin{align}
\gamma_d\bigl(\bW_\infty^\bullet(\rho)\bigr)
&=2\sum_{j=0}^\infty
\upsilon_{d+1+2j}\bigl(\bW_\infty^\bullet(\rho)\bigr)
=q_d^\bullet(\rho),
\label{eq:hilbert-brownian-grassmann-absorption}
\\
\gamma_d(\bU_{\infty, T}^\bullet)
&=2\sum_{j=0}^\infty\upsilon_{d+1+2j}(\bU_{\infty, T}^\bullet)
=a_d^\bullet(T).
\label{eq:ou-spiral-grassmann-absorption}
\end{align}
Here $q_d^\bullet(\rho)=\Pp(\rho Z\in\mathcal K_d^\bullet)$
is the Gaussian-point absorption probability for the corresponding Brownian hull $\mathcal K_d^\bullet$ in $\R^d$, evaluated in
Theorem~\ref{thm:brownian-gaussian-absorption}, and
$a_d^\bullet(T)=\Pp(0\in\mathcal H_{d,T}^\bullet)$
is the origin absorption probability for the corresponding
OU convex hull in $\R^d$, evaluated in
Theorem~\ref{thm:ou-origin-absorption}.
\end{corollary}

\begin{proof}
Fix $\bullet\in\{B,A\}$ and $\rho,T>0$, and let
$C$ denote either $\bW_\infty^\bullet(\rho)$ or
$\bU_{\infty, T}^\bullet$. By
Theorem~\ref{thm:hilbert-brownian-ou-intrinsic-volumes},
we have
$\sum_{k=0}^\infty\upsilon_k(C)=1$.
Hence \eqref{eq:hilbert-conic-intrinsic-volume-partial-sums}
implies $\gamma_N(C)\to0$ as $N\to\infty$.
For $M\ge0$, telescoping
\eqref{eq:hilbert-space-conic-intrinsic-volumes} gives
\[
2\sum_{j=0}^M\upsilon_{d+1+2j}(C)
=\gamma_d(C)-\gamma_{d+2M+2}(C).
\]
Letting $M\to\infty$ proves the series representations
in \eqref{eq:hilbert-brownian-grassmann-absorption} and
\eqref{eq:ou-spiral-grassmann-absorption}.
Finally, substituting
\eqref{eq:hilbert-brownian-intrinsic-volumes} and
\eqref{eq:ou-spiral-intrinsic-volume-identification}
into these series and applying
Theorems~\ref{thm:brownian-gaussian-absorption} and
\ref{thm:ou-origin-absorption} identifies their sums with
$q_d^\bullet(\rho)$ and $a_d^\bullet(T)$, respectively.
\end{proof}

\subsection{Generating functions of the conic intrinsic volumes of the OU spiral cones}
\label{subsec:conic-generating-functions}
For $\rho>0$ and $\bullet\in\{B,A\}$, consider the generating function of the conic intrinsic volumes of the cone $\bW_\infty^\bullet(\rho)$:
\begin{equation}
\label{eq:hilbert-conic-generating-functions-definition}
G_\rho^\bullet(z)
:=\sum_{k=0}^\infty z^k
\upsilon_k\bigl(\bW_\infty^\bullet(\rho)\bigr)
=\sum_{k=0}^\infty z^kL_k^\bullet(\rho),
\end{equation}
where the last equality uses
Theorem~\ref{thm:hilbert-brownian-ou-intrinsic-volumes}.

\begin{proposition}[Generating functions for $\bW_\infty^\bullet(\rho)$]
\label{prop:hilbert-conic-generating-functions}
For every $\rho>0$, both power series defined in
\eqref{eq:hilbert-conic-generating-functions-definition}
converge absolutely and locally uniformly on $\mathbb C$.
They define entire functions given by
\begin{align}
G_\rho^B(z)
&=\frac{\arctan\rho}{\pi}
+\frac z\pi\int_0^{\arsinh(1/\rho)}
\frac{\eee^{zx}\,\dint x}{\sqrt{\rho^{-2}-\sinh^2x}},
\label{eq:type-B-hilbert-conic-generating-function}
\\
G_\rho^A(z)
&=\frac{\eee^{\theta_\rho z}-\eee^{-\theta_\rho}}
        {2\sinh\theta_\rho},
\qquad \theta_\rho=2\arsinh\frac1{2\rho}.
\label{eq:type-A-hilbert-conic-generating-function}
\end{align}
\end{proposition}

\begin{proof}
We substitute the explicit formulas for $L_k^\bullet(\rho)$
from Theorem~\ref{thm:finite-BA-scaling-limit} into the
power series \eqref{eq:hilbert-conic-generating-functions-definition}
and evaluate the resulting sums.
For type $B$, put $b:=\arsinh(1/\rho)$.
By \eqref{eq:type-B-limiting-intrinsic-volumes}, for $k\ge1$,
\[
0\le L_k^B(\rho)
\le\frac{b^{k-1}}{(k-1)!}L_1^B(\rho).
\]
This proves absolute and locally uniform convergence of the
series defining $G_\rho^B$. It also justifies interchanging
summation and integration in
\eqref{eq:type-B-limiting-intrinsic-volumes}, since
\[
\sum_{k=1}^\infty
\frac{z^k x^{k-1}}{(k-1)!}
=z\eee^{zx}.
\]
Together with \eqref{eq:type-B-limiting-zeroth-intrinsic-volume},
this proves \eqref{eq:type-B-hilbert-conic-generating-function}.
For type $A$, \eqref{eq:type-A-limiting-intrinsic-volumes}
gives
\[
G_\rho^A(z)
=L_0^A(\rho)
+\frac{\eee^{\theta_\rho z}-1}{2\sinh\theta_\rho}.
\]
Since $\tanh(\theta_\rho/2)=(1+4\rho^2)^{-1/2}$,
\eqref{eq:type-A-limiting-zeroth-intrinsic-volume} yields
$L_0^A(\rho)=(1-\eee^{-\theta_\rho})/(2\sinh\theta_\rho)$.
This proves \eqref{eq:type-A-hilbert-conic-generating-function}.
Absolute and locally uniform convergence of the series defining
$G_\rho^A$ follows from~\eqref{eq:type-A-limiting-intrinsic-volumes}.
\end{proof}


\begin{corollary}[Generating functions for $\bU_{\infty, T}^\bullet$]
\label{cor:ou-spiral-conic-generating-functions}
For every $T>0$ and $z\in\mathbb C$,
\begin{align}
\sum_{k=0}^\infty z^k\upsilon_k(\bU_{\infty, T}^B)
&=\frac{\arcsin(\eee^{-T/2})}{\pi}
+\frac z\pi\int_0^{\arsinh\sqrt{\eee^T-1}}
\frac{\eee^{zx}\,\dint x}{\sqrt{\eee^T-1-\sinh^2x}},
\label{eq:ordinary-ou-spiral-conic-generating-function}
\\
\sum_{k=0}^\infty z^k\upsilon_k(\bU_{\infty, T}^A)
&=\frac{\eee^{zT/2}-\eee^{-T/2}}{2\sinh(T/2)}.
\label{eq:circular-ou-spiral-conic-generating-function}
\end{align}
\end{corollary}
\begin{proof}
By Proposition~\ref{prop:hilbert-brownian-cones-ou-spirals},
the cones $\bU_{\infty, T}^\bullet$ and
$\bW_\infty^\bullet(\rho_\bullet(T))$ are isometric,
where $\rho_\bullet(T)$ is given in
\eqref{eq:ou-inverse-time-scales-repeat}.
By isometry invariance of the conic intrinsic volumes
and \eqref{eq:hilbert-conic-generating-functions-definition},
\begin{equation}
\label{eq:ou-spiral-conic-generating-functions}
\sum_{k=0}^\infty z^k\upsilon_k(\bU_{\infty, T}^\bullet)
=G_{\rho_\bullet(T)}^\bullet(z),
\qquad z\in\mathbb C.
\end{equation}
Substituting \eqref{eq:ou-inverse-time-scales-repeat} into
\eqref{eq:type-B-hilbert-conic-generating-function} and
\eqref{eq:type-A-hilbert-conic-generating-function},
and using
$\rho_B(T)^{-2}=\eee^T-1$,
$\arctan\rho_B(T)=\arcsin(\eee^{-T/2})$, and
$\theta_{\rho_A(T)}=T/2$,
gives the asserted formulas.
\end{proof}

\subsection{Small-opening limits and ordinary intrinsic volumes}
\label{subsec:small-opening-euclidean-intrinsic-volumes}

We finally relate conic intrinsic volumes to the ordinary intrinsic
volumes of compact convex sets. For a bounded base set $K$, the cone
$\pos(\sigma e_0+K)$ becomes concentrated around the positive
$e_0$-ray as $\sigma\to\infty$. The following theorem shows that,
after suitable rescaling, the conic intrinsic volumes of $\pos(\sigma e_0+K)$ converge to
explicit multiples of the ordinary intrinsic volumes of $K$.

For a nonempty compact convex set $K\subseteq\R^m$,  the ordinary intrinsic volumes are defined by the Steiner formula
\begin{equation}
\label{eq:small-opening-euclidean-steiner-normalization}
\operatorname{vol}_m(K+rB_m)
=\sum_{j=0}^m\kappa_{m-j}V_j(K)r^{m-j},\qquad r\ge0,
\end{equation}
where $B_m$ is
the closed unit ball, and $\kappa_m=\pi^{m/2}/\Gamma(m/2+1)$,
with $\kappa_0:=1$. For a nonempty compact convex set $K$ in a
real separable Hilbert space $\mathbb H$, define
\begin{equation}
\label{eq:small-opening-hilbert-euclidean-intrinsic-volumes}
V_j(K):=\sup_{K'\subseteq K}V_j(K')\in[0,\infty],
\qquad j\ge0,
\end{equation}
where $K'$ ranges over nonempty finite-dimensional compact convex
sets, with their intrinsic volumes computed in any finite-dimensional
Euclidean space containing them. In particular, $V_0(K)=1$.
The theory originates in work of Sudakov~\cite{Sudakov1976} and
Chevet~\cite{Chevet1976} and was further developed by
Tsirelson~\cite{Tsirelson1986}, who related intrinsic volumes to
Gaussian spectra; see also Vitale~\cite{Vitale2001} and
\cite[Sections~1.1 and~1.4]{KabluchkoZaporozhets2016}.

\begin{theorem}[Small opening limit]
\label{thm:general-small-opening-intrinsic-volumes}
Let $K$ be a nonempty compact convex subset of a real separable
Hilbert space $\mathbb H$, and suppose that $V_1(K)<\infty$.
Adjoin a unit vector $e_0$ orthogonal to $\mathbb H$, and define
\begin{equation}
\label{eq:small-opening-lifted-cones}
C_\sigma(K):=\pos(\sigma e_0+K)\subseteq\R e_0\oplus \mathbb H,
\qquad \sigma>0.
\end{equation}
These cones are closed, nonzero, and pointed. For every integer
$j\ge0$,
\begin{align}
\lim_{\sigma\to\infty}\sigma^j\gamma_j(C_\sigma(K))
&=\frac{2V_j(K)}{(j+1)\kappa_{j+1}},
\label{eq:general-small-opening-grassmann-limit}
\\
\lim_{\sigma\to\infty}\sigma^j\upsilon_{j+1}(C_\sigma(K))
&=\frac{V_j(K)}{(j+1)\kappa_{j+1}}.
\label{eq:general-small-opening-conic-limit}
\end{align}
Moreover, $\sigma\mapsto\sigma^j\gamma_j(C_\sigma(K))$ is
nondecreasing.
\end{theorem}

The finiteness condition $V_1(K)<\infty$ in the theorem
has a probabilistic characterization in terms of
isonormal Gaussian processes. An
\emph{isonormal process} on $\mathbb H$ is a centered Gaussian family
$(\xi(h))_{h\in \mathbb H}$ with covariance
$\Ee[\xi(h)\xi(g)]=\langle h,g\rangle$.
A compact set is called \emph{Gaussian bounded} (GB) if the restriction
of this process has a version with almost surely bounded sample
paths. For compact convex $K$, Gaussian boundedness is equivalent
to $V_1(K)<\infty$.
Moreover, Chevet's inequality gives
\begin{equation}
\label{eq:small-opening-chevet-bound}
V_j(K)\le\frac{V_1(K)^j}{j!},\qquad j\ge1;
\end{equation}
see \cite{Chevet1976} and \cite[Theorem~1]{Vitale2001}.
Thus the assumption $V_1(K)<\infty$  guarantees finiteness of all
intrinsic volumes.

\begin{proof}[Proof of Theorem~\ref{thm:general-small-opening-intrinsic-volumes}.]
\emph{Step 1.} Convexity of $K$ gives
$C_\sigma(K)=\{\lambda(\sigma e_0+h):\lambda\ge0,\ h\in K\}$.
To see that $C_\sigma(K)$ is closed, suppose that
$\lambda_m(\sigma e_0+h_m)\to x$, where $\lambda_m\ge0$
and $h_m\in K$.
Taking $e_0$-coordinates yields
$\lambda_m\to\lambda:=\langle x,e_0\rangle/\sigma\ge0$.
By compactness of $K$, we may pass to a subsequence along which
$h_m\to h\in K$. Hence
$x=\lambda(\sigma e_0+h)\in C_\sigma(K)$.
The cone $C_\sigma(K)$ is nonzero because $K$ is nonempty, and it is pointed
because every nonzero vector in it has positive $e_0$-coordinate.

\medskip
\emph{Step 2: Finite-dimensional bases.}
Fix a nonempty finite-dimensional compact convex set $K'\subseteq K$
and an integer $j\ge1$. In this step we show that
\begin{equation}
\label{eq:small-opening-finite-base-limit}
\sigma^j\gamma_j(C_\sigma(K'))
\uparrow\frac{j!\kappa_j}{(2\pi)^j}V_j(K')
=\frac{2V_j(K')}{(j+1)\kappa_{j+1}}.
\end{equation}
Let $F:=\lin K'$, let
$G:F\to\R^j$ be a standard Gaussian linear map, and let
$Z\sim N(0,I_j)$ be independent of $G$.
The map $A:\R e_0\oplus F\to\R^j$ given by
$A(ae_0+h)=-aZ+Gh$ is standard Gaussian in orthonormal bases.
When $j<\dim(\R e_0\oplus F)$, its kernel is uniformly distributed
among subspaces of codimension $j$. A nonzero vector
$\lambda(\sigma e_0+h)$, with $h\in K'$, belongs to this kernel
if and only if $Gh=\sigma Z$. Consequently,
\begin{equation}
\label{eq:small-opening-gaussian-kernel-identity}
\gamma_j(C_\sigma(K'))=\Pp(\sigma Z\in GK').
\end{equation}
If $j\ge\dim(\R e_0\oplus F)$, the map $A$ is injective almost
surely and both sides are zero, so the identity remains valid.
Thus \eqref{eq:small-opening-gaussian-kernel-identity}
holds for every integer $j\ge1$.
Conditioning on $G$ gives
\begin{equation}
\label{eq:small-opening-gaussian-density-identity}
\sigma^j\gamma_j(C_\sigma(K'))
=\frac1{(2\pi)^{j/2}}\,
\Ee\int_{GK'}\eee^{-\|x\|^2/(2\sigma^2)}\,\dint x.
\end{equation}
On the other hand, the Gaussian representation of Euclidean intrinsic volumes gives
\begin{equation}
\label{eq:small-opening-finite-gaussian-volume}
\Ee\,\operatorname{vol}_j(GK')
=\frac{j!\kappa_j}{(2\pi)^{j/2}}V_j(K');
\end{equation}
see \cite{Vitale2008}, or the finite-dimensional case of
Tsirelson's theorem~\cite{Tsirelson1986}, recalled in
\cite[Theorem~1.9]{KabluchkoZaporozhets2016}.
When $j>\dim K'$, both sides of
\eqref{eq:small-opening-finite-gaussian-volume} are zero.
The integrand in \eqref{eq:small-opening-gaussian-density-identity}
increases to one as $\sigma\uparrow\infty$. Monotone convergence
and the gamma duplication formula therefore yield~\eqref{eq:small-opening-finite-base-limit}.

\medskip
\emph{Step 3: Passage to $K$.}
We claim that
\begin{equation}
\label{eq:small-opening-finite-base-supremum}
\gamma_j(C_\sigma(K))
=\sup_{K'\subseteq K}\gamma_j(C_\sigma(K')),
\end{equation}
where $K'$ ranges over nonempty finite-dimensional compact convex
subsets of $K$, as in
\eqref{eq:small-opening-hilbert-euclidean-intrinsic-volumes}.
Indeed, each $C_\sigma(K')$ is an admissible finite-dimensional
closed subcone of $C_\sigma(K)$. Conversely, for any nonzero finite-dimensional
closed subcone $D\subseteq C_\sigma(K)$, set
$K_D:=\{h\in K:\sigma e_0+h\in D\}$.
This is a nonempty compact convex set contained in the finite-dimensional
affine space $\lin D-\sigma e_0$. Normalizing every nonzero vector
of $D$ to have $e_0$-coordinate $\sigma$ gives
$D=C_\sigma(K_D)$. The zero subcone contributes nothing to the
supremum, proving \eqref{eq:small-opening-finite-base-supremum}.

By \eqref{eq:small-opening-finite-base-supremum}, the function
$\sigma\mapsto\sigma^j\gamma_j(C_\sigma(K))$ is a supremum of
nondecreasing functions and hence is nondecreasing. Taking suprema
in either order gives
\begin{equation}\label{eq:lim-sigma-gamma-C-sigma-proof}
\lim_{\sigma\to\infty}\sigma^j\gamma_j(C_\sigma(K))
=\sup_{\sigma>0}\sup_{K'\subseteq K}
  \sigma^j\gamma_j(C_\sigma(K'))
=\sup_{K'\subseteq K}
  \frac{2V_j(K')}{(j+1)\kappa_{j+1}}
=\frac{2V_j(K)}{(j+1)\kappa_{j+1}}.
\end{equation}
This argument assumed $j\geq 1$. For $j=0$, the assertion follows directly from
$\gamma_0(C_\sigma(K))=V_0(K)=1$ and $\kappa_1=2$.

\medskip
\emph{Step 4: Conic intrinsic volumes.}
Since $C_\sigma(K)$ is not a linear subspace,
\eqref{eq:hilbert-space-conic-intrinsic-volumes} gives
\[
\sigma^j\upsilon_{j+1}(C_\sigma(K))
=\frac12\sigma^j\gamma_j(C_\sigma(K))
 -\frac12\sigma^j\gamma_{j+2}(C_\sigma(K)).
\]
By \eqref{eq:small-opening-chevet-bound}, $V_{j+2}(K)<\infty$.
Equation~\eqref{eq:lim-sigma-gamma-C-sigma-proof}, applied at index $j+2$, therefore implies
$\sigma^j\gamma_{j+2}(C_\sigma(K))=O(\sigma^{-2})$.
Combining this with \eqref{eq:lim-sigma-gamma-C-sigma-proof}
proves \eqref{eq:general-small-opening-conic-limit}, including $j=0$.
\end{proof}

\begin{remark}
\label{rem:small-opening-without-finiteness}
Neither the proof of
\eqref{eq:general-small-opening-grassmann-limit}
nor that of the monotonicity assertion uses $V_1(K)<\infty$.
Consequently, for every nonempty compact convex $K$ and every
integer $j\ge0$,
\[
\sigma^j\gamma_j(C_\sigma(K))
\uparrow\frac{2V_j(K)}{(j+1)\kappa_{j+1}}
\qquad\text{as }\sigma\to\infty,
\]
where the limit may be infinite.
In particular, $V_1(K)=\pi\lim_{\sigma\to\infty} \sigma\gamma_1(C_\sigma(K))$.
Thus, finiteness of this limit verifies the assumption
$V_1(K)<\infty$ of
Theorem~\ref{thm:general-small-opening-intrinsic-volumes},
allowing us to apply its conclusion for conic intrinsic volumes.
\end{remark}

We apply Theorem~\ref{thm:general-small-opening-intrinsic-volumes}
to the two Brownian bodies. Together with the explicit formulas
for the conic intrinsic volumes of $\bW_\infty^\bullet(\rho)$
in Theorem~\ref{thm:hilbert-brownian-ou-intrinsic-volumes},
this yields the following formulas, originally obtained by
Gao--Vitale~\cite{GaoVitale2001} and Gao~\cite{Gao2003}.

\begin{corollary}[Intrinsic volumes of Brownian bodies]
\label{cor:hilbert-brownian-euclidean-intrinsic-volumes}
For every integer $j\ge0$,
\begin{equation}
\label{eq:hilbert-brownian-euclidean-intrinsic-volumes}
V_j(\bK_\infty^B)=\frac{\kappa_j}{j!},
\qquad
V_j(\bK_\infty^A)=\frac{\kappa_{j+1}}{2j!}.
\end{equation}
\end{corollary}

\begin{proof}
We first verify the assumption
$V_1(\bK_\infty^\bullet)<\infty$ of
Theorem~\ref{thm:general-small-opening-intrinsic-volumes}
without relying on the formulas to be proved.
Remark~\ref{rem:small-opening-without-finiteness} at index $1$
and \eqref{eq:hilbert-brownian-grassmann-absorption} give
\[
V_1(\bK_\infty^\bullet)
=\pi\lim_{\rho\to\infty}\rho\,q_1^\bullet(\rho).
\]
By \eqref{eq:type-B-gaussian-absorption-dimension-one} and
\eqref{eq:type-A-gaussian-absorption-dimension-one},
the right-hand side equals $2$ in the type-$B$ case and
$\pi/2$ in the type-$A$ case. Thus both bodies satisfy
$V_1(\bK_\infty^\bullet)<\infty$, and Theorem~\ref{thm:general-small-opening-intrinsic-volumes}
applies.

We now evaluate the limits on the left-hand side of
\eqref{eq:general-small-opening-conic-limit}
for the two Brownian bodies.
For type $B$, \eqref{eq:hilbert-brownian-intrinsic-volumes} and the
substitution $u=\rho\sinh x$ in
\eqref{eq:type-B-limiting-intrinsic-volumes} yield
\[
\rho^j\upsilon_{j+1}(\bW_\infty^B(\rho))
=\frac1{\pi j!}\int_0^1
\frac{\bigl(\rho\arsinh(u/\rho)\bigr)^j}
     {\sqrt{1-u^2}\sqrt{1+u^2/\rho^2}}\,\dint u.
\]
Since $0\le\rho\arsinh(u/\rho)\le u$ and
$\rho\arsinh(u/\rho)\to u$, dominated convergence and the beta
integral give
\begin{equation}
\label{eq:type-B-hilbert-small-opening-limit}
\lim_{\rho\to\infty}\rho^j\upsilon_{j+1}(\bW_\infty^B(\rho))
=\frac1{\pi j!}\int_0^1\frac{u^j}{\sqrt{1-u^2}}\,\dint u
=\frac{\Gamma((j+1)/2)}{2\sqrt\pi\,j!\,\Gamma(j/2+1)}
=\frac{\kappa_j}{j!(j+1)\kappa_{j+1}}.
\end{equation}
For $j=0$, the powers in the integrand are interpreted as one.
For type $A$, $\theta_\rho=2\arsinh(1/(2\rho))$ satisfies
$\theta_\rho\to0$ and $\rho\theta_\rho\to1$, so
\begin{equation}
\label{eq:type-A-hilbert-small-opening-limit}
\rho^j\upsilon_{j+1}(\bW_\infty^A(\rho))
=\frac{(\rho\theta_\rho)^j}{2(j+1)!}
  \frac{\theta_\rho}{\sinh\theta_\rho}
\longrightarrow\frac1{2(j+1)!}.
\end{equation}
Multiplying the two limits by $(j+1)\kappa_{j+1}$ and using
\eqref{eq:general-small-opening-conic-limit} proves
\eqref{eq:hilbert-brownian-euclidean-intrinsic-volumes}.
\end{proof}

\section*{Acknowledgements}
This work was supported by the DFG under Germany's Excellence Strategy
EXC 2044/2--390685587, Mathematics M\"unster:
\emph{Dynamics-Geometry-Structure}; by the DFG Priority Programme
SPP 2265, \emph{Random Geometric Systems}; and by the DFG Research
Training Group \emph{Rigorous Analysis of Complex Random Systems}
(RTG 3027, Project Number 524444762).

\section*{Declarations}

\subsection*{Statement on the use of generative AI}

This paper grew out of an extended dialogue between the human author and ChatGPT (OpenAI, GPT-6). One of the starting points was the approach of~\cite{KabluchkoZaporozhets2020}, where absorption probabilities for Gaussian polytopes are studied by means of conic integral geometry, in particular through Grassmann angles and conic intrinsic volumes of suitable deterministic cones. The initial program for the present work was to develop an analogous geometric approach to Gaussian-point absorption by Gaussian random-walk and random-bridge convex hulls and then to pass to the Brownian setting by a scaling limit. The latter step, as well as the determination and asymptotic analysis of the corresponding type~\(B\) and type~\(A\) cones, introduces substantial additional difficulties not present in the Gaussian-polytope setting.

ChatGPT played a key role throughout the mathematical development
and writing of the paper. Its contributions included proposing and refining mathematical ideas, carrying out technical calculations, suggesting proof strategies and generalizations, deriving and checking formulas, identifying connections with related areas of probability and convex geometry, researching relevant literature, and assisting with the organization and writing of the manuscript. The work developed iteratively: suggestions generated by ChatGPT were examined, corrected, refined, or omitted from the final manuscript, while the human author supplied mathematical direction, additional ideas, research questions and editorial decisions throughout.
Mathematical arguments and claims included in the final manuscript were critically examined and independently verified by the human author, who takes full responsibility for the contents of the paper.

\subsection*{Conflict of interest statement}
The author declares that he has no conflict of interest.

\subsection*{Data availability statement}
We do not analyze or generate any datasets.

\bibliographystyle{amsplain}
\bibliography{brownian_convex_hulls_d_dim_bib}

\end{document}